\documentclass[letterpaper, 11pt,  reqno]{amsart}
\usepackage[margin=1.2in,marginparwidth=1.5cm, marginparsep=0.5cm]{geometry}

\usepackage{amsmath,amssymb,amscd,amsthm,amsxtra, esint, xcolor}

\usepackage[implicit=true]{hyperref}
\allowdisplaybreaks[2]

\usepackage{color}

\definecolor{gr}{rgb}   {0.,   0.69,   0.23 }
\definecolor{bl}{rgb}   {0.,   0.5,   1. }
\definecolor{mg}{rgb}   {0.85,  0.,    0.85}
\definecolor{yl}{rgb}   {0.8,  0.7,   0.}
\definecolor{or}{rgb}  {0.7,0.2,0.2}

\newtheorem{theorem}{Theorem} [section]

\newtheorem{lemma}[theorem]{Lemma}
\newtheorem{proposition}[theorem]{Proposition}
\newtheorem{remark}[theorem]{Remark}
\newtheorem{example}[theorem]{Example}

\newtheorem{definition}[theorem]{Definition}
\newtheorem{corollary}[theorem]{Corollary}

\newcommand{\dr}{\theta}
\newcommand{\Dr}{\Theta}

\newcommand{\I}{\hspace{0.5mm}\text{I}\hspace{0.5mm}}
\newcommand{\II}{\text{I \hspace{-2.8mm} I} }
\newcommand{\noi}{\noindent}
\newcommand{\Z}{\mathbb{Z}}
\newcommand{\R}{\mathbb{R}}
\newcommand{\C}{\mathbb{C}}
\newcommand{\T}{\mathbb{T}}
\newcommand{\bul}{\bullet}
\newcommand{\cK}{\mathcal{Q}}

\newcommand{\bv}{\big\vert}

\newcommand{\Gdl}{\mathcal{G}_{\dl} }

\newcommand{\mubo}{\mu_{\infty} }

\newcommand{\mukdv}{\wt \mu_{0} }

\newcommand{\too}{\longrightarrow}

\newcommand{\BO}{\text{\rm BO} }
\newcommand{\KDV}{\text{\rm KdV} }

\newcommand{\W}{\mathcal{W}}
\newcommand{\U}{\mathcal{U}}

\newcommand{\Ha}{\mathbb{H}_a}

\let\Re=\undefined\DeclareMathOperator*{\Re}{Re}
\let\Im=\undefined\DeclareMathOperator*{\Im}{Im}

\let\P= \undefined
\newcommand{\P}{\mathbf{P}}

\newcommand{\E}{\mathbb{E}}
\newcommand{\EE}{\mathcal{E}}

\renewcommand{\L}{\mathcal{L}}

\newcommand{\RR}{\mathcal{R}}

\newcommand{\CC}{\mathcal{C}}
\newcommand{\D}{\mathcal{D}}

\newcommand{\J}{\mathcal{J}}

\newcommand{\GG}{\mathcal{G}}

\newcommand{\F}{\mathcal{F}}

\newcommand{\al}{\alpha}
\newcommand{\be}{\beta}
\newcommand{\dl}{\delta}

\newcommand{\nb}{\nabla}

\newcommand{\eps}{\varepsilon}

\newcommand{\g}{\gamma}
\newcommand{\G}{\Gamma}
\newcommand{\ld}{\lambda}

\newcommand{\s}{\sigma}
\newcommand{\Si}{\Sigma}
\newcommand{\ft}{\widehat}

\newcommand{\wt}{\widetilde}
\newcommand{\cj}{\overline}
\newcommand{\dx}{\partial_x}
\newcommand{\dt}{\partial_t}
\newcommand{\dd}{\partial}

\newcommand{\ta}{\theta}

\renewcommand{\l}{\ell}
\renewcommand{\o}{\omega}
\renewcommand{\O}{\Omega}

\newcommand{\les}{\lesssim}
\newcommand{\ges}{\gtrsim}

\newcommand{\jb}[1]
{\langle #1 \rangle}

\newcommand{\jbb}[1]
{\big\langle #1 \big\rangle}

\renewcommand{\b}{\beta}
\newcommand{\ind}{\mathbf 1}

\newcommand{\PP}{\mathbb{P}}
\newcommand{\M}{\mathcal{M}}
\def\e{\varepsilon}

\newcommand{\N}{\mathbb{N}}

\newtheorem*{ackno}{Acknowledgements}

\newcommand{\Pk}{\rho_\dl}

\renewcommand{\H}{\mathcal{H}}

\newcommand{\dtv}{d_{\rm TV}} 
\newcommand{\dlp}{d_{\rm LP}} 
\newcommand{\dhh}{d_{\rm H}} 
\newcommand{\dkl}{d_{\rm KL}} 
\newcommand{\dkf}{d_{\rm KF}} 

\def\sgn{\textup{sgn}}

\newcommand{\low}{\textup{low}}
\newcommand{\high}{\textup{high}}
\newcommand{\Id}{\textup{Id}}

\numberwithin{equation}{section}
\numberwithin{theorem}{section}

\makeatletter
\@namedef{subjclassname@2020}{\textup{2020} Mathematics Subject Classification}
\makeatother

\begin{document}
\baselineskip = 14pt

\title[On the deep-water and shallow-water limits of the ILW equation]
{On the deep-water and shallow-water limits  of \\
the intermediate long wave
equation \\ from a statistical viewpoint}

\author[G.~Li, T.~Oh, and G.~Zheng]
{Guopeng Li, Tadahiro Oh, and Guangqu Zheng}

 \address{
Guopeng Li, 
School of Mathematics and Statistics, Beijing Institute of Technology,
Beijing 100081, China\\
and
School of Mathematics\\
The University of Edinburgh\\
and The Maxwell Institute for the Mathematical Sciences\\
James Clerk Maxwell Building\\
The King's Buildings\\
Peter Guthrie Tait Road\\
Edinburgh\\ 
EH9 3FD\\
 United Kingdom}

\email{guopeng.li@bit.edu.cn}

\address{
Tadahiro Oh, 
School of Mathematics\\
The University of Edinburgh\\
and The Maxwell Institute for the Mathematical Sciences\\
James Clerk Maxwell Building\\
The King's Buildings\\
Peter Guthrie Tait Road\\
Edinburgh\\ 
EH9 3FD\\
 United Kingdom,
 and
 School of Mathematics and Statistics, Beijing Institute of Technology,
Beijing 100081, China
}

\email{hiro.oh@ed.ac.uk}
\address{
Guangqu Zheng, School of Mathematics\\
The University of Edinburgh\\
and The Maxwell Institute for the Mathematical Sciences\\
James Clerk Maxwell Building\\
The King's Buildings\\
Peter Guthrie Tait Road\\
Edinburgh\\ 
EH9 3FD\\
 United Kingdom, 
 Department of Mathematical Sciences\\
University of Liverpool\\
Mathematical Sciences Building\\ 
Liverpool, L69 7ZL
United Kingdom, 
and
Department of Mathematics and Statistics, Boston University, 665 Commonwealth Avenue, Boston, MA 02215, USA
 }

\email{gzheng90@bu.edu}
%\email{guangqu.zheng@liverpool.ac.uk}

%
%


\begin{abstract}
We study convergence problems for the intermediate long wave equation (ILW), with the depth parameter $\dl > 0$, 
in the deep-water limit ($\dl \to \infty$) and the shallow-water limit ($\dl \to 0$) from a statistical
point of view. 
In particular, we establish convergence
of invariant Gibbs dynamics for ILW 
in both the deep-water and shallow-water limits.
For this purpose, we first construct
the Gibbs measures for 
 ILW, $0 < \dl < \infty$.
 As they are  supported on distributions, 
  a renormalization is required.
With the Wick renormalization, we carry out the construction of 
the Gibbs measures  for ILW.
We then   prove that
 the Gibbs measures for ILW converge in total variation to that for the
Benjamin-Ono equation (BO)
in the deep-water limit ($\dl \to \infty$).
In the shallow-water regime, after applying a scaling transformation, we prove that, 
as $\dl \to 0$,  the  Gibbs measures for the scaled ILW
converge weakly to that for the  Korteweg-de Vries  
equation (KdV).
We point out that this second result is of particular interest since
the Gibbs measures for the scaled ILW and KdV are mutually singular
(whereas
the Gibbs measures for ILW and BO are equivalent).

 In terms of dynamics, we use a compactness
argument to construct invariant Gibbs dynamics for ILW 
(without uniqueness).
Furthermore, we show that, by extracting a sequence~$\dl_m$, 
this invariant Gibbs dynamics for  ILW 
converges to that for BO in the deep-water limit ($\dl_m \to \infty$)
 and to that for KdV (after
the scaling) in the shallow-water limit ($\dl_m \to 0$), respectively. 

Lastly, we point out that our results also apply to
the generalized ILW equation in the defocusing case, 
converging to the generalized BO in the deep-water limit
and to the generalized KdV in the shallow-water limit.
In the non-defocusing case, however, 
our results can not be extended  to a nonlinearity with a higher power
due to the non-normalizability of the corresponding Gibbs measures.

\end{abstract}

%\date{\today}
%%
%
\maketitle

\tableofcontents

\baselineskip = 14pt

\section{Introduction}

\label{SEC:1}

\subsection{Intermediate long wave equation}
\label{SUBSEC:1.1}

In this paper, we study 
the  intermediate long wave equation (ILW) 
on the circle $\T = \R/(2\pi \Z)$:
\begin{align}
\begin{cases}
\dt u -
 \Gdl \partial_x^2 u  =  \dx (u^2)    \\ 
u|_{t = 0} = u_0,
\end{cases}
\qquad ( t, x) \in \R \times \T.
\label{ILW1}
\end{align}

\noi
The equation \eqref{ILW1}, also known 
as the finite-depth fluid equation, 
 models 
the internal wave propagation of the interface 
in a stratified fluid of finite depth $\dl > 0$, 
and the unknown $u :\R\times \T \to \R$
denotes the amplitude of the internal wave at the interface.
See also Remark~\ref{REM:Q0}.
The dispersion operator $\Gdl $ characterizes
 the phase speed and
it is defined  as the following Fourier multiplier operator:
\begin{align}
%f \in \mathcal{D}'(\T) 
%\too
\ft{\Gdl f}(n)  =  -  i \Big( \coth(\dl n )-\frac{1}{\dl  n} \Big) \ft f(n)\,, 
 \hspace{1mm} 
\quad n \in\Z, 
       \label{OP}
\end{align}

\noi
where  $\coth$ denotes the usual hyperbolic cotangent function:
\begin{align*}
\coth(x ) = \frac{   e^x + e^{-x}  }{ e^x - e^{-x}  }
= \frac{e^{2x}+1}{e^{2x}-1}, \quad x \in \R\setminus\{0\}
\end{align*}

\noi
with the convention $\coth(\dl n) - \frac{1}{\dl n} = 0$ for $n=0$;
see \eqref{tan1}.
See \eqref{FT1} below for our convention of the (spatial) Fourier transform. 
 ILW \eqref{ILW1} is an important physical model,  
 providing 
 a natural connection between 
the deep-water regime (= the Benjamin-Ono regime) and the shallow-water regime
(= the KdV regime).
As such, it has been studied extensively from both the applied and theoretical points of view.
See, for example,  
a recent book  \cite[Chapter~3]{KS2022} by  Klein and Saut
for an overview of the subject
and the references therein.
See also a survey \cite{Saut2019}.
These two references indicate that the rigorous mathematical study of 
ILW is still widely  open. 
In particular, one of the fundamental, but challenging questions
is the convergence properties
of ILW in the deep-water limit (as
the depth parameter $\dl$ tending to $\infty$)
and in the shallow-water limit 
(as $\dl \to 0$).
In this paper, we make the first study 
on this convergence issue of ILW
from a statistical viewpoint.

 \begin{remark}\label{REM:Q0}\rm
 In \cite{KKD}, the equation for the motion of the internal wave in a finite depth fluid was 
 derived with two depth parameters $\dl_j$, $j = 1, 2$, 
 where $\dl_1$ and $\dl_2$ represent the depths
 of the upper and lower fluids, respectively, 
 and is given by 
 \begin{align}
 \dt u -
c_1  \mathcal{G}_{\dl_1} \partial_x^2 u - c_2  \mathcal{G}_{\dl_2} \partial_x^2 u  =  \dx (u^2) .  
\label{ILW2}
\end{align}

 \noi
 See (25a)-(25b) and (35a)-(35b) in \cite{KKD}.
 In \cite[VI Summary]{KKD}, the authors
 proposed a special case of interest
 when the internal wave is located halfway between
 the upper and lower fluid boundaries, 
 namely, $\dl_1 = \dl_2$.
 In this case, by setting $\dl = \dl_1 + \dl_2 = 2\dl_1$, 
 the equation~\eqref{ILW2} 
 reduces to the ILW equation \eqref{ILW1}
 (up to some inessential multiplicative constants).
 We also point out that 
 by taking $\dl_1 \to 0$ while keeping $\dl_2$ fixed
 (or by taking $\dl_2 \to 0$ while keeping $\dl_1$ fixed), 
 we also see that the equation \eqref{ILW2}
 reduces to the ILW equation \eqref{ILW1}.

 \end{remark}

\subsection{Deep-water and shallow-water limits of the generalized ILW}
\label{SUBSEC:1.2}

In the following, we consider  the generalized intermediate long wave   equation 
(gILW)
on $\T$:
\begin{align}
\begin{cases}
\dt u -
 \Gdl \partial_x^2 u 
           =  \dx (u^k)    \\ 
u|_{t = 0} = u_0,
\end{cases}
\qquad ( t, x) \in \R \times \T,
\label{gILW1}
\end{align}

\noi
where $k \geq 2$ is an integer.
When $k = 2$,  
the equation \eqref{gILW1}  corresponds to
ILW \eqref{ILW1}, while, when $k = 3$,  it is known as 
 the modified ILW equation.
The equation  \eqref{gILW1} 
can be written in the following Hamiltonian formulation:
\begin{align*}
\dt u = \dx \frac{d E_\dl(u)}{du} ,
\end{align*}

\noi
where $E_\dl(u)$ is  the Hamiltonian (= energy) given by 
\begin{align} \label{Hu}
E_\dl(u) =
\frac 12   \int_\T  u \Gdl \dx u   dx
+ \frac{1}{k+1} \int_\T u^{k+1} dx.
\end{align}

\noi
In particular, $E_\dl(u)$ is conserved under the dynamics of \eqref{gILW1}.
Moreover, it is easy to check that the following two quantities are conserved
under the gILW dynamics:
\begin{align*}
\text{mean: } \int_\T u  dx
\qquad \text{and}\qquad 
\text{mass: } M(u) = \int_\T u^2  dx.
\end{align*}

\noi
We also point out that 
ILW ($k = 2$) is known to be completely integrable.
We, however, do not make use of the completely integrable structure
in the following.
See Remarks~\ref{REM:XX1} and~\ref{REM:int}.

For simplicity of the presentation, we impose
the mean-zero condition on the initial condition $u_0$, 
 namely, 
$\int_\T u_0 dx = 0$, 
in the remaining part of the paper.
In  view of the conservation of the (spatial) mean,  
this implies that 
 the solution $u(t)$ has mean zero as long as it exists. 
In other words, defining the Fourier coefficient $\ft f(n)$
by 
\begin{align}
\ft f(n) = \F(f)(n) = \int_\T f(x) e^{- inx} dx,
\label{FT1}
\end{align}

\noi
we will work with real-valued functions of the form:\footnote{Hereafter, we may drop
the harmless factor $2\pi$.}
\begin{align}
f(x)         = \frac{1}{2\pi}   \sum_{n \in \Z^\ast} \ft{f} (n) e_n(x)   , 
\label{FT2}
\end{align}

\noi
where  $e_n(x) = e^{inx}$ and $\Z^\ast = \Z \setminus \{0\}$.

Our main goal is to study 
the deep-water limit ($\dl \to \infty$) and the shallow-water limit ($\dl \to 0$)
of solutions to gILW \eqref{gILW1} from a statistical viewpoint.
In particular, 
we study the convergence problem with rough and random initial data, 
more precisely, with the Gibbs measure initial data.
In the next subsection, we provide a detailed discussion on (the construction 
and convergence of) Gibbs measures.
In the deterministic setting, 
the  convergence problem of gILW in both the deep-water and shallow-water limits 
has been  studied in 
\cite{ABFS, GW, HW, MPS, Gli, Gli2}, 
providing rigorous mathematical support
of the numerical study performed in   \cite{KKD}.
We point out that the recent work~\cite{Gli, Gli2} by the first author
is the only convergence result of gILW  on the circle $\T$.
In the following, 
let us   briefly go over the
formal derivation of the limiting equation 
in each of the deep-water and shallow-water limits.
With a slight abuse of notation, we
set
\begin{align}
\ft \Gdl(n) 
 =    - i   \Big( \coth(\dl n)-\frac{1}{\dl n} \Big) .
\label{Gdx1}
\end{align}

\medskip

\noi
$\bullet$ 
\textbf{Deep-water limit} 
($\delta \to \infty$). \rule[-2.5mm]{0mm}{0mm}
\\
\indent
In this case, an elementary computation shows that
\begin{align} \label{lim0}
\lim_{\dl\to\infty}\ft \Gdl(n)
= - i \sgn(n) 
\end{align}
 
\noi
for any $n \in \Z$.
Indeed,  defining $q_\dl(n)$ by\footnote{While it is not needed in the mean-zero case, 
we may  set
$q_\dl(0) = 0$ by continuity.}
\begin{align}
%\ft{\cK_\dl u}(n)
%= q_{\dl}(n) \ft{u}(n) \quad \text{with}\quad 
%%
%%%
q_{\dl}(n) 
= 
|n| +   \frac{1}{  \dl } -  n  \coth(\dl n) , 
\label{PO1a}
\end{align}

\noi
 one may easily verify that 
\begin{align}
\label{BOp}
0\leq q_{\dl}(n) = q_\dl(-n) \leq \frac{2}{ \dl} 
\end{align}
for any $ n\in\Z$; 
see Lemma 4.1 in \cite{ABFS}.
In fact, \eqref{BOp} holds with the right-hand side
replaced by $\frac 1\dl$;
see Remark \ref{REM:q1} below.
%
%%
%%%

The limit \eqref{lim0} indicates that, 
in the deep-water limit, namely, as $\dl \to \infty$, 
the  gILW equation \eqref{gILW1}
converges to 
the following generalized Benjamin-Ono equation (gBO) on $\T$:
\begin{equation}
    \partial_t u     -  \mathcal{H} (\partial_x^2 u)
    =  \dx(u^k)   ,
\label{BO}
\end{equation}

\noi
where 
$\mathcal{H}$ is the Hilbert transform 
defined
by
\[
\ft{\H f}(n)=-i \sgn(n)\ft f(n) .
\]

\noi
Formally speaking, 
by  recasting \eqref{gILW1} as
\begin{equation}
\label{PE1}
\partial_t u - \H  (\partial_x^2 u) +
\cK_\dl \dx u
=  \dx(u^k),
\end{equation}

\noi
where  $\cK_\dl = (\H-  \Gdl )\dx $ is defined 
as a Fourier multiplier operator with symbol
 $ q_{\dl}$ in \eqref{PO1a}.
 Then, the bound \eqref{BOp} shows
 that $\cK_\dl$ tends to $0$ in a suitable sense, 
 thus yielding 
 the formal convergence of \eqref{PE1}
 (and hence of \eqref{gILW1})
 to gBO \eqref{BO} as $\dl \to \infty$.
 In proving rigorous convergence, 
 one indeed 
needs to show that 
$\cK_\dl \dx$ tends to $0$ in a suitable sense
(instead of $\cK_\dl$),
and thus, 
in view of the bound~\eqref{BOp}, 
it indicates that 
 in the deep-water regime $\dl \gg1 $,  
long waves (with relatively small frequencies $|n| \ll \dl$)
``well approximate''
  long waves of infinitely  deep water ($\dl = \infty$).

\medskip

\noi
$\bullet$
\textbf{Shallow-water limit} ($\dl\to 0$).\rule[-2.5mm]{0mm}{0mm}
\\
\indent
A direct computation shows that, 
for $n\in\Z^\ast$, we have 
\begin{align}
\begin{aligned}
 \ft{\Gdl \dx^2 u}(n) 
& =  i \Big( \coth( \delta n )-\frac{1}{ \delta n} \Big)
  n^2 \ft{u}(n) \\
&  =  i \frac{ \dl}{3} n^3 \ft{u}(n) + o(1), 
 \end{aligned} \label{asy1}
 \end{align}

\noi
as $\dl \to 0$.
The identity
 \eqref{asy1} follows from the following 
 identity with  
 $x=\dl n$:\footnote{The limiting behavior \eqref{tan1} also follows from the Taylor
 expansion of the hyperbolic cotangent function.}
\begin{align}
\coth(x) - \frac{1}{x} = \frac{x(e^{2x}-1) - (e^{2x}- 2x -1)   }{x(e^{2x}-1 )} 
=  \frac{x}{3} + o(1), 
\label{tan1}
\end{align}

\noi
 as $x\to 0$, 
 which can be verified by  using   the Taylor expansion:
$e^{2x} = 1 + 2x + \sum_{k=2}^\infty (2x)^k/k!$.

The identity \eqref{asy1} shows that, 
the dispersion in \eqref{gILW1} disappears as $\dl \to 0$, 
formally yielding the inviscid Burgers equation in the limit
(when $k = 2$).
In order to circumvent this issue, we introduce the following scaling transformation
for each $\dl > 0$, \cite{ABFS}:
\noi
\begin{align}
v(t,x) = (\tfrac{3}{ \dl})^{\frac{1}{k-1}} u\big(\tfrac{3}{ \dl} t,x\big),
\label{trans}
\end{align}

\noi
which leads to the following scaled gILW equation:
\begin{equation}
\label{gILW3}
\dt v   -  \frac{3}{ \dl}  \Gdl   \dx^2 v= \dx(v^k) .
\end{equation}

\noi
Namely, $v$ is a solution to the scaled gILW \eqref{gILW3}
(with the scaled initial data)
if and only if~$u$ is a solution to the original gILW \eqref{gILW1}.
Note that the scaled gILW \eqref{gILW3}
is a Hamiltonian PDE with the Hamiltonian:
\begin{align}
\EE_\dl(v) =  \frac{3}{2\dl} 
\int_\T v \Gdl\dx v dx + \frac{1}{k+1} \int_\T v^{k+1}dx \,,
\label{Hv}
\end{align}

\noi
which differs from the Hamiltonian $E_\dl(u)$ in \eqref{Hu} by a 
{\it divergent} multiplicative constant in the kinetic part
(= the quadratic part) of the Hamiltonian.
In view of~\eqref{asy1},
the scaled gILW~\eqref{gILW3} formally converges to the  following 
generalized KdV equation (gKdV) on $\T$:
\begin{align}
\dt v + \dx^3 v = \dx(v^k) . 
\label{KDV}
\end{align}

From the physical point of view,
the scaling transformation \eqref{trans}
is a very natural operation to perform, when $k = 2$.
%We point out that the scaling transformation is consistent with the physical
%setting.
The ILW equation \eqref{ILW1} describes the motion
of the fluid interface in a stratified fluid of depth $\dl > 0$,
where $u$ denotes the amplitude of the internal wave at the interface.
As $\dl \to 0$, the entire fluid depth tends to $0$
and, in particular, 
the amplitude of the internal wave at the interface is $O(\dl)$,  which also  tends to $0$, 
in the physical model.
Hence, if we want to observe any meaningful limiting behavior, 
we need to magnify the fluid motion by a factor $\sim \frac 1\dl$, which
is exactly 
what the scaling transformation \eqref{trans} does when $k = 2$.
We also point out that studying the convergence problem
for the scaled ILW~\eqref{gILW3} (with $k = 2$)
with $O(1)$ initial data means that we are indeed studying
the original ILW \eqref{ILW1} with $O(\dl)$ initial data, 
which is consistent with the physical viewpoint
explained above.

\medskip

As mentioned above, in the deterministic setting, 
the convergence problem of the gILW dynamics 
(and the scaled gILW dynamics, respectively) 
to  the gBO dynamics (and to the gKdV dynamics, respectively) has been studied 
   in \cite{ABFS,GW,HW,Gli, Gli2, CLOP}. 
 These works studied the convergence issue from a 
 {\it microscopic
 viewpoint} in the sense that
 convergence was established
 for {\it each fixed}   initial data $u|_{t = 0} = u_0$ to gILW \eqref{gILW1}
(or  each fixed   initial data $v|_{t = 0} =  v_0$ 
to the scaled gILW~\eqref{gILW3}).
In the present work, 
we study the convergence problem from a {\it macroscopic viewpoint}.
Namely,
rather than considering the limiting behavior of
a single trajectory,
we study the limiting behavior of solutions
as a statistical ensemble.
Such an approach is of fundamental importance in statistical mechanics, 
where 
one replaces  ``the study of the microscopic dynamical trajectory of an individual macroscopic system by the study of appropriate ensembles or probability measures on the phase space of the system''~\cite{LRS}.
In the present work,  we in particular  study convergence of the dynamics 
at the Gibbs equilibrium 
for the gILW equation~\eqref{gILW3}
in both the deep-water and shallow-water limits.
From the physical point of view, 
it is quite natural to study the fluid motion as a statistical ensemble,  
since one is often interested in a prediction of typical behavior
of the fluid.
%namely, study 
%the behavior of solutions 
%as a statistical ensemble.
From the theoretical point of view, 
it is an interesting and challenging question to study 
 convergence of 
invariant Gibbs dynamics
associated with the gILW equation \eqref{gILW3}, 
in particular due to the low regularity of the support of the Gibbs measures. 

\medskip

Our strategy for establishing convergence
of invariant Gibbs dynamics for the (scaled) gILW  consists of the following three steps.
For simplicity, we only discuss the deep-water limit in the following, 
where we treat the original gILW \eqref{gILW1}
(rather than the scaled gILW~\eqref{gILW3}
relevant in the shallow-water limit), unless we need to make a specific 
point in the shallow-water limit.

In the following, we will restrict our attention
to (i)~$k = 2$, corresponding to ILW \eqref{ILW1}, 
and (ii)~$k \in 2 \N + 1$ in 
\eqref{gILW1}, corresponding to the defocusing case.
This restriction comes from the Gibbs measure construction.
See Remark \ref{REM:Q2}
for a discussion on the general  focusing\footnote{Strictly speaking, 
the case (iii)~even $k \ge 4$ is {\it non-defocusing}, not focusing.
For simplicity, however, we may refer to the non-defocusing case as focusing in the remaining part of the paper.} case,
namely,  either for (iii)~even $k \ge 4$
or (iv)~$k \in 2\N+ 1$ with the focusing sign:
\begin{align}
\dt u -
 \Gdl \partial_x^2 u 
=  - \dx (u^k)   .
\label{gILW7}
\end{align}

%\noi
% See Remark \ref{REM:Q2} below.

\smallskip

\noi
$\bullet$ {\bf Step 1:}
{\bf Construction and convergence of the Gibbs measures.}\rule[-2.5mm]{0mm}{0mm}
\\
\indent
For each finite $\dl > 0$,  
we  first  construct  a Gibbs measure $\rho_\dl$
for  gILW \eqref{gILW1}
with the Hamiltonians $E_\dl(u)$ in \eqref{Hu}, 
formally written as\footnote{Henceforth,  constants such as $Z_\dl$ denote
various normalizing constants, which may be different line by line.}
\begin{align}
\begin{split}
\rho_\dl (du)
& = Z_\dl^{-1} e^{-E_\dl(u)} du \\
&= Z_\dl^{-1} e^{-\frac 1{k+1} \int_\T u^{k+1} dx} 
e^{-\frac 12 \int_\T u\Gdl \dx  u dx} du.
\end{split}
\label{Gibbs1}
\end{align}

\noi
The expression \eqref{Gibbs1} is merely formal
and we aim to construct $\rho_\dl$
as a weighted Gaussian measure
with the base Gaussian measure given 
by 
\begin{align}
\mu_\dl(du) =  Z_\dl^{-1} e^{-\frac 12 \int_\T u\Gdl \dx  u dx } du.
\label{Gibbs2}
\end{align}

\noi
See the next subsection for a precise definition of $\mu_\dl$.
For each $\dl > 0$, the Gaussian measure 
$\mu_\dl$ is supported on distributions $\mathcal D'(\T)$ of negative regularity
and thus the potential energy 
$\int_\T u^{k+1} dx$ in \eqref{Gibbs1} is divergent.
In order to overcome this issue, we introduce a renormalization
on the potential energy, 
just as in the construction of the $\Phi^{k+1}_2$-measures \cite{Simon, GJ, DPT1, OTh1}.
When $k = 2$, the potential energy is not sign-definite, causing a further problem.
By following the work~\cite{Tz10}, we overcome this issue by introducing a Wick-ordered $L^2$-cutoff.
See Subsection~\ref{SUBSEC:Gibbs}
for a further discussion.

Once the Gibbs measure $\rho_\dl$ is constructed for each $\dl > 0$, 
we then proceed to prove convergence of the Gibbs measures $\rho_\dl$
for gILW \eqref{gILW3} to the Gibbs measure $\rho_\BO$ for gBO~\eqref{BO}
in the deep-water limit ($\dl \to \infty$).
This step involves establishing the $L^p$-integrability bound
on the densities, {\it uniformly in $\dl \gg 1$.}
We point out that, for each $\dl \gg1$, the base Gaussian measure $\mu_\dl$
is different and thus an extra care is needed in discussing what we mean by the ``density''.
See Section \ref{SEC:Gibbs1} for further details.

In order to study the shallow-water limit, 
we need to consider the scaled gILW \eqref{gILW3}
with the Hamiltonian
$\EE_\dl(v)$ in \eqref{Hv}.
This leads to the construction of the following Gibbs measure:
\begin{align}
\begin{split}
\wt \rho_\dl (dv)
& =  Z_\dl^{-1} e^{-\EE_\dl(v)} dv \\
&= Z_\dl^{-1} e^{-\frac 1{k+1} \int_\T v^{k+1} dx} 
e^{-\frac 3{2\dl} \int_\T v\Gdl \dx  v dx} dv.
\end{split}
\label{Gibbs3}
\end{align}

\noi
For each fixed $\dl> 0$, 
we construct the Gibbs measure $\wt \rho_\dl$
as a weighted Gaussian measure with the base Gaussian measure $\wt \mu_\dl$ given by 
\begin{align}
   \wt{\mu}_\dl(dv) =  Z_\dl^{-1} e^{   - \frac {3}{2\dl} \int_\T v\Gdl \dx  v dx } dv.
\label{Gibbs4}
\end{align}

\noi
The construction of the Gibbs measure $\wt \rho_\dl$, $\dl > 0$, 
follows exactly the same lines as that for the Gibbs measure $\rho_\dl$ in \eqref{Gibbs1}.
There is, however, a crucial difference in the shallow-water limit
in establishing convergence of the Gibbs measures $\wt \rho_\dl$, $\dl \ll 1$, 
for the scaled gILW \eqref{gILW3}
to the Gibbs measure $\rho_{\KDV}$ for gKdV \eqref{KDV}.
More precisely, it turns out that 
 the Gibbs measures $\wt \rho_\dl$, $\dl > 0$, for the scaled gILW \eqref{gILW3} and 
$\rho_{\KDV}$ for gKdV \eqref{KDV}
are mutually singular and the mode of convergence of $\wt \rho_\dl$ to $\rho_{\KDV}$ is weaker
(than that in the deep-water limit).

This first step is one of the main novelties of the paper, where we establish a uniform bound
on the densities (with respect to the underlying probability measure $\PP$).

\medskip

\noi
$\bullet$ {\bf Step 2:}
{\bf Construction of invariant Gibbs dynamics for the (scaled) gILW.}\rule[-2.5mm]{0mm}{0mm}
\\
\indent
In this second step, we construct dynamics for gILW \eqref{gILW1}
at the Gibbs equilibrium constructed in Step 1.
This step follows the compactness argument introduced by 
Burq, Thomann, and Tzvetkov \cite{BTT1} in the context of dispersive PDEs.
See \cite{AC, DPD} for the first instance of this argument in the context of
fluid.  See also \cite{OTh1, ORT}.
Due to the use of the compactness argument, 
the dynamics constructed in this step lacks a uniqueness statement.

\medskip

\noi
$\bullet$ {\bf Step 3:}
{\bf Convergence of the (scaled) gILW dynamics at the Gibbs equilibrium.}\rule[-2.5mm]{0mm}{0mm}
\\
\indent
This last step essentially follows from the previous two steps together with the triangle inequality.
In Step~2, we construct limiting Gibbs dynamics as a limit of
the frequency-truncated dynamics (via the compactness argument mentioned above).
In this last step, we characterize the convergence established in Step 2
in the L\'evy-Prokhorov metric and 
conclude the desired convergence of the dynamics at the Gibbs equilibrium
for the (scaled) gILW to that for gBO  (or to gKdV)
via a diagonal argument.
The use of the L\'evy-Prokhorov metric in this context is new
as far as our knowledge is concerned.

\begin{remark}\rm

There is also a slightly different formulation for the  ILW equation;
 see   \cite[p.\,211]{AS}. 
In this formulation, the generalized ILW equation on $\T$ reads as 
\begin{align}
\dt u -
\Big(1+\frac{1}{\dl}\Big)
 \Gdl \partial_x^2 u 
           =  \dx (u^k)    
\label{gILW4}
\end{align}

\noi
with the Hamiltonian $\wt E_\dl(u)$ given by 
\begin{align*} 
\wt E_\dl(u) =
\frac{\dl + 1}{2\dl}  \int_\T  u \Gdl \dx u   dx
+ \frac{1}{k+1} \int_\T u^{k+1} dx.
\end{align*}

\noi
In taking $\dl \to \infty$, 
we formally have
\begin{align*}
\dt u - \Gdl  \dx^2 u= \dx(u^k)  + O(\dl^{-1}),
\end{align*}

\noi
which indicates that the same convergence result holds for 
this version \eqref{gILW4} of gILW
in the deep-water limit.
On the other hand, in the shallow-water regime, 
in view of \eqref{asy1}, 
the equation \eqref{gILW4} can be formally written as
\begin{align*}
\dt u - \frac 1\dl \Gdl  \dx^2 u = \dx(u^k)  + O(\dl),
\end{align*}

\noi
which indicates convergence of \eqref{gILW4}  to the following gKdV:
\begin{align}
\dt u + \frac 13 \dx^3 u = \dx(u^k) 
\label{KDV2}
\end{align}

\noi
{\it without} any scaling transformation.
Indeed, in the shallow-water limit, 
a slight modification of our argument shows that an analogue
of our main result holds
for the version \eqref{gILW4} converging to gKdV  \eqref{KDV2}
in the shallow-water limit.
On the one hand,  the formulation~\eqref{gILW4} may seem to be a convenient
model since it does not require a scaling transformation in the shallow-water limit.
On the other hand, 
it does not seem to reflect the physical behavior in the shallow-water regime
(where the entire depth and thus the amplitude $u$ are $O(\dl)$).

\end{remark}

\subsection{Construction and convergence of Gibbs measures}\label{SUBSEC:Gibbs}

Consider  a finite-dimensional Hamiltonian flow on $\R^{2n}$:
\begin{align} \label{HR2}
\dt p_j = \frac{\partial H}{\partial q_j} 
\qquad
\text{and}
\qquad 
\dt q_j = - \frac{\partial H}{\partial p_j}
\end{align}

\noi
with Hamiltonian 
\[
H (p, q)= H(p_1, \cdots, p_n, q_1, \cdots, q_n).
\]

\noi
The classical 
Liouville's theorem states that the Lebesgue measure
$dp dq = \prod_{j = 1}^n dp_j dq_j$
on~$\mathbb{R}^{2n}$ is invariant under the dynamics \eqref{HR2}.
Then, together with the conservation of the Hamiltonian $H(p, q)$, 
we see that 
  the Gibbs measure
$Z^{-1}e^{- H(p, q)} dpdq $ is invariant 
under the dynamics of~\eqref{HR2}.
By drawing an analogy, we may hope to construct invariant Gibbs dynamics
for Hamiltonian PDEs.
This program was initiated by the seminal works 
by Lebowitz, Rose, and Speer \cite{LRS}
and 
Bourgain \cite{BO94, BO96}, 
leading to the construction of invariant Gibbs dynamics
as well  as probabilistic well-posedness.
See also \cite{Fried, Zhid, McKean}.
This subject has been increasingly more popular over the last fifteen years;
see, for example, survey papers~\cite{OHRIMS, BOP4}.

Our first main goal is to construct 
Gibbs measures for  gILW \eqref{gILW1}
(and the scaled gILW~\eqref{gILW3}).
For this purpose, let us first 
go over the known results in
the limiting cases $\dl =0$ and $\dl = \infty$.

\medskip

\noi
$\bullet$
{\bf Construction of  Gibbs measures 
for gKdV on $\T$.} \rule[-2.5mm]{0mm}{0mm}
\\
\indent
This corresponds to the shallow-water limit ($\dl = 0$) in our problem.
Consider gKdV \eqref{KDV} posed on the circle
with the Hamiltonian $\EE_0(u)$:
\begin{align*}
\EE_0(v) = 
 \frac{1}{2} 
\int_\T (\dx v)^2  dx + \frac{1}{k+1} \int_\T v^{k+1}dx, 
\end{align*}

\noi
which, in view of \eqref{asy1}, 
 is a formal limit of $\EE_\dl(v)$ in \eqref{Hv} as $\dl \to 0$.
The Gibbs measure $\rho_{\KDV}$ for gKdV is formally given 
by 
\begin{align}
\begin{split}
\rho_{\KDV}(dv)
& = Z_0^{-1} e^{-\EE_0(v)} dv \\
&= Z_0^{-1} e^{-\frac 1{k+1} \int_\T v^{k+1} dx} 
e^{-\frac 12 \int_\T (\dx v)^2  dx} dv.
\end{split}
\label{Gibbs5}
\end{align}

\noi
The Gibbs measure $\rho_{\KDV}$ can be constructed 
as a weighted Gaussian measure with the base Gaussian measure given by 
the periodic Wiener measure $\wt \mu_0$ (restricted to mean-zero functions):
\begin{align}
 \mukdv(dv) =  Z_0^{-1} e^{-\frac 12 \int_\T (\dx v)^2  dx } dv.
\label{Gibbs6}
\end{align}

\noi
More precisely, the periodic Wiener measure $\mukdv$
is defined as the induced probability measure under the map:\footnote{Note that $X_\KDV$ is nothing but the Brownian loop on $\T$
(with the zero spatial mean).}
\begin{align}
\omega \in \O \longmapsto
X_\KDV(\o) = \frac 1{2\pi} \sum_{n \in \Z^*} \frac{g_n(\o)}{ | n | }e_n,
\label{Gibbs7}
\end{align}

\noi
where $e_n(x) = e^{inx}$ and $\{g_n\}_{n \in \Z^*}$ 
is a sequence of independent standard\footnote{By convention, 
we assume that  $g_n$
has mean 0 and variance $2\pi$, $n \in \Z^*$.
See \eqref{Gibbs8} below.} complex-valued Gaussian random variables
on a probability space $(\O, \F, \PP)$
conditioned that $g_{-n} = \cj {g_n}$, $n \in \Z^*$.
Indeed, by Plancherel's theorem (see \eqref{FT1} and \eqref{FT2}
for our convention of the Fourier transform), we have
\begin{align*}
\int_\T (\dx v)^2 dx 
= \frac 1{2\pi} \sum_{n \in \Z^*}n^2 |\ft v(n)|^2
= \frac 1 \pi \sum_{n \in \N} n^2 |\ft v(n)|^2, 
\end{align*}

\noi
where the second equality follows from the fact that $v$ is real-valued, 
i.e.~$\ft v(-n) = \cj{\ft v(n)}$.
This shows that we formally have
\begin{align}
\begin{split}
e^{-\frac 12 \int_\T (\dx v)^2  dx } dv
& \sim  \prod_{n \in \N} e^{-\frac 1{2\pi} n^2 |\ft v(n)|^2} d \ft v(n)
 \sim \prod_{n \in \N} e^{-\frac 1{2\pi} |g_n|^2 } d g_n\\
&   \sim \bigg(\prod_{n \in \N} e^{-\frac 1{2\pi} (\Re g_n)^2 } d \Re g_n\bigg)
\bigg(\prod_{n \in \N} e^{-\frac 1{2\pi} (\Im g_n)^2 } d \Im g_n\bigg)
\end{split}
\label{Gibbs8}
\end{align}

\noi
in the limiting sense with the identification $\ft v(n) = \frac{g_n}{|n|}$.
This shows that $\Re g_n$ and $\Im g_n$ 
are given by mean-zero Gaussian random variables with variance $\pi$.
Hence, $g_n = \Re g_n + i \Im g_n$ has variance $2\pi$.

It is easy to show that the support of $\mukdv$
is 
contained $H^{\frac 12 - \eps}(\T) \setminus H^{\frac 12} (\T)$
for any $\eps > 0$.
By  Khintchine's inequality, one may also show that 
 the support of $\wt \mu_0$
is indeed contained in $W^{\frac 12 - \eps, \infty}(\T)$. 
See, for example,  \cite{Benyi}
for a further discussion on the regularity of the Brownian loop $X_\KDV$ in \eqref{Gibbs7}.
Hence, in the defocusing case, namely, 
when $k \in 2\N + 1$, 
the density $e^{-\frac 1{k+1} \int_\T v^{k+1} dx} $ in \eqref{Gibbs5}
with respect to 
$\wt \mu_0$
satisfies 
$ 0 <  e^{-\frac 1{k+1} \int_\T v^{k+1} dx} \le 1$, 
almost surely, which is in particular integrable with respect to $\wt \mu_0$.
This shows that 
 the Gibbs measure $\rho_{\KDV}$
can be  realized as 
a weighted $\wt \mu_0$:
\begin{align}
\rho_{\KDV}(dv)
&= Z_0^{-1} e^{-\frac 1{k+1} \int_\T v^{k+1} dx} d\wt \mu_0(v)
\label{Gibbs5a}
\end{align}

\noi
 in this case.

In the focusing case, 
namely, when $k \in 2\N$
or when the potential energy $\frac 1{k+1} \int_\T v^{k+1} dx$ in \eqref{Gibbs5a}
comes with the $+$ sign, 
the situation is completely different, 
since, in this case, the density is no longer integrable
with respect to the base Gaussian measure $\wt \mu_0$.
In the seminal work \cite{LRS}, 
Lebowitz, Rose, and Speer proposed to consider the Gibbs measure with an $L^2$-cutoff:
\begin{align}
\rho_{\KDV}(dv)
&= Z_0^{-1} \ind_{\{ \int_\T v^2 dx \le K\}} e^{-\frac 1{k+1} \int_\T v^{k+1} dx} 
d\mukdv (v)
\label{Gibbs9}
\end{align}

\noi
for $k \in 2\N$ in the non-defocusing case, and more generally in the focusing case:
\begin{align}
\rho_{\KDV}(dv)
&= Z_0^{-1} \ind_{\{ \int_\T v^2 dx \le K\}} e^{\frac 1{k+1} \int_\T |v|^{k+1} dx} 
d\mukdv (v)
\label{Gibbs10}
\end{align}

\noi
for any real number $k > 1$.
In \cite{LRS, BO94}, it was shown that, when $k < 5$,  
the Gibbs measures $\rho_{\KDV}$ in \eqref{Gibbs9}
and \eqref{Gibbs10}
can be constructed as a probability measure
for any $K > 0$, 
while it is not normalizable for any cutoff size when $k > 5$.
The situation at the critical case\footnote{From a PDE point of view,
this criticality corresponds to the so-called $L^2$-criticality (or mass-criticality), 
while,  from the viewpoint of mathematical physics, 
this criticality corresponds to 
the phase transitions for (non-)normalizability of the focusing Gibbs measure.
Here, the phases transitions are two-fold: normalizability for $k < 5$ and non-normalizability for $k \ge 5$.
Also, when $k = 5$, normalizability below or at the critical mass and non-normalizability 
above the critical mass.}
 $k = 5$ (for~\eqref{Gibbs10}) is more subtle.
Note that  the critical value $k = 5$ corresponds
 to the smallest power of the nonlinearity,
 where the focusing gKdV (namely, \eqref{KDV} with the $-$ on the nonlinearity) on the real line
 possesses finite-time blowup
solutions \cite{MM00, Me01}.
The Gibbs measure construction when $k = 5$
remained a challenging open problem for  thirty years
and it 
was completed only recently
in the work \cite{OST22} by Sosoe, Tolomeo, and the second author;
when $k = 5$, the focusing Gibbs measure in \eqref{Gibbs10}
can be constructed
if and only if the cutoff size $K$ is less than or equal to the
mass of the so-called ground state on the real line.
See~\cite{OST22} for a further discussion on this issue.

As we see below, in the non-defocusing case, 
only the $k = 2$ case is relevant to us.
In this case, the Gibbs measure for KdV relevant to us
is given by\footnote{Hereafter, we use a continuous cutoff function $\chi_K$
as in \cite{Tz10}.} 
\begin{align}
\rho_{\KDV}(dv)
&= Z_0^{-1} 
\chi_K\bigg(\int_\T v^2 dx - 2\pi  \s_\KDV\bigg) e^{-\frac 1{3} \int_\T v^{3} dx} 
d\mukdv (v), 
\label{Gibbs11}
\end{align}

\noi
where 
$\chi_K: \R \to  [0,1]$
is  a continuous function %with compact support such that 
such that $\chi_K(x) = 1$ for $|x| \le K$
and $\chi_K(x) = 0$ for $|x| \ge 2K$.

See Theorem \ref{THM:Gibbs2} below.
Here,  $\s_\KDV$ denotes the variance of $X_\KDV(x)$
in \eqref{Gibbs7} given by 
\begin{align}
\s_\KDV = \E\big[X_\KDV^2(x)\big] = \frac{1}{4\pi^2} \sum_{n \in \Z^*} \frac{2\pi}{n^2}
= \frac{\pi}{6}, 
\label{s1}
\end{align}

\noi
which is independent of $x \in \T$ due to the translation invariant nature of the problem.

\medskip

\noi
$\bullet$
{\bf Construction of 
 Gibbs measures 
for gBO on $\T$.} 
\rule[-2.5mm]{0mm}{0mm}
\\
\indent
Next, we go over the (non-)construction 
of the Gibbs measures associated with gBO \eqref{BO},
which
corresponds to the deep-water limit ($\dl = \infty$) in our problem.
The Hamiltonian for gBO \eqref{BO} 
is given by 
\begin{align*}
E_\infty(u) = 
 \frac{1}{2} 
\int_\T u \H \dx u   dx + \frac{1}{k+1} \int_\T u^{k+1}dx, 
\end{align*}

\noi
which, in view of \eqref{Gdx1}, 
 is a formal limit of $E_\dl(u)$ in \eqref{Hu} as $\dl \to \infty$.
 Here, $\H$ denotes the Hilbert transform.
Then, the Gibbs measure $\rho_{\BO}$ for gBO is formally given 
by 
\begin{align*}
\rho_{\BO}(du)
& = Z_\infty^{-1} e^{-E_\infty(u)} du \\
&= Z_\infty^{-1} e^{-\frac 1{k+1} \int_\T u^{k+1} dx} 
e^{-\frac 12 \int_\T u \H \dx u  dx} du.
\end{align*}

\noi
As in the gKdV case, 
we first introduce the base Gaussian measure
$\mubo$ by 
\begin{align}
 \mubo(du) =  Z_\infty^{-1} e^{-\frac 12 \int_\T u \H \dx u  dx } du.
\label{Gbo1}
\end{align}

\noi
More precisely, the Gaussian measure  $\mubo$
is defined as the induced probability measure under the map:
\begin{align}
\o \in \O \longmapsto
 X_\BO(\o) = \frac 1{2\pi}
 \sum_{n \in \Z^*} \frac{g_n(\o)}{ | n |^\frac12}e_n, 
\label{Gbo2}
 \end{align}

\noi
where $\{g_n\}_{n \in \Z^*}$ is as in \eqref{Gibbs7}.
In this case, 
the support of $\mubo$ is contained in $H^{-\eps}(\T)\setminus L^2(\T)$
for any $\eps > 0$; see \eqref{GX3} below.
Namely, a typical element $u$ in the support of $\mubo$ is merely a distribution
and thus the potential energy is divergent in this case.

Let us first consider the defocusing case $k \in 2\N + 1$.
Noting that the Gaussian measure $\mubo$ is logarithmically correlated, 
by introducing a Wick renormalized power $\W(u^{k+1})$ (see~\eqref{WO1} below), 
Nelson's estimate allows us to  define the Gibbs measure $\rho_\BO$:
\begin{align*}
\rho_{\BO}(du)
&= Z_\infty^{-1} e^{-\frac 1{k+1} \int_\T \W(u^{k+1}) dx} 
d \mubo (u)
\end{align*}

\noi
as a limit of the frequency-truncated version, 
just as in the construction of the $\Phi^{k+1}_2$-measure~\cite{DPT1, OTh1};
see Theorem \ref{THM:Gibbs1} below.
See Subsection \ref{SUBSEC:2.2}
for a precise definition of the Wick power 
$\W(u^{k+1})$.

Let us now turn to the focusing case.
 When $k = 2$, 
 Tzvetkov~\cite{Tz10} constructed 
 the Gibbs measure for the Benjamin-Ono equation (BO) by introducing a Wick-ordered $L^2$-cutoff:
\begin{align}
\begin{split}
\rho_{\BO}(du)
&= Z_\infty^{-1} 
\chi_K\bigg( \int_\T \W(u^2) dx \bigg)e^{-\frac 1{3} \int_\T u^{3} dx} 
d \mubo(u).
\end{split}
\label{Gx3}
\end{align}

 \noi
 See  \cite{OST} for an alternative, concise proof.
 Note that under the mean-zero assumption, 
 there is no need to introduce a renormalization in this case.
 See Remark \ref{REM:k2}.
 
In  \cite{OST},  Seong, Tolomeo, and the second author showed that 
 the Gibbs measure for the focusing modified BO
(with $k = 3$):
\begin{align}
\begin{split}
\rho_{\BO}(du)
&= Z_\infty^{-1} 
\chi_K\bigg( \int_\T \W(u^2) dx \bigg)e^{\frac 1{k+1} \int_\T \W(u^{k+1}) dx} 
d \mubo(u)
\end{split}
\label{Gx4}
\end{align}
 
\noi
is not normalizable.
Their argument can also be adapted to show that 
the focusing Gibbs measure is not normalizable for any $k \geq 3$.
We mention the work 
 \cite{BS} by Brydges and Slade
 on a similar non-normalizability result (but with a completely different proof)
 in the context of the focusing $\Phi^4_2$-measure.
 See also Remark \ref{REM:Q2} below.
 
Lastly, we point out that, 
due to the use of the Wick renormalization, 
we can only consider integer values for $k$
in this case ($\dl = \infty$) and also in the intermediate case $0 < \dl < \infty$
which we will discuss next.

\medskip

\noi
$\bullet$
{\bf Construction of 
 Gibbs measures 
for gILW on $\T$.} 
\rule[-2.5mm]{0mm}{0mm}
\\
\indent
We finally discuss the construction of the Gibbs measure
for the (scaled) gILW.
Let us first consider the unscaled gILW \eqref{gILW1}
with the Hamiltonian $E_\dl(u)$ in \eqref{Hu}.
Fix $0 < \dl < \infty$.
Our first goal is to construct the Gibbs measure $\rho_\dl$
of the form \eqref{Gibbs1}.
Let $ \mu_\dl$ be 
the base Gaussian measure of the form  \eqref{Gibbs2}, 
which is nothing but the induced probability measure under the map:
\begin{align}
\o \in \O \longmapsto
 X_\dl(\o) =   \frac 1{2\pi}  \sum_{n\in\Z^\ast} \frac{g_n(\o) }{| K_\dl(n)|^{ \frac{1}{2} } }e_n.
\label{GG3}
 \end{align}

\noi
Here, $\{g_n\}_{n \in \Z^*}$ is as in \eqref{Gibbs7} and $K_\dl(n)$ is given by 
\begin{align}
K_\dl(n) := in \ft \Gdl(n)
=   n\coth( \delta n )-\frac{1}{ \delta} 
\label{GG4}
\end{align}

\noi
with $\ft \Gdl(n)$ as in \eqref{Gdx1}.
For each $n \in \Z^*$, we have $K_\dl(n) > 0$
and moreover, it follows from~\eqref{PO1a} and \eqref{BOp}
that 
\begin{align*}
K_\dl(n) = |n| + O\big(\tfrac 1\dl\big).
\end{align*}

\noi
See also Lemma \ref{LEM:p2} and Remark \ref{REM:q1}.
This asymptotics allows us to show that,
for any given $0 < \dl < \infty$, 
the Gaussian measures $\mu_\dl$ in \eqref{Gibbs2}
and $\mubo$ in \eqref{Gbo1} are equivalent.\footnote{Namely, mutually absolutely continuous.} 
See
Proposition~\ref{PROP:equiv}.
In particular, as in the $\dl = \infty$ case, 
the Gaussian measure $\mu_\dl$ is supported
on 
 $H^{-\eps}(\T)\setminus L^2(\T)$
for any $\eps > 0$ (see  \eqref{GX3} below)
and thus
we need to renormalize the potential energy.

Given $N \in \N$, 
let $\P_N$ be the Dirichlet projection onto the frequencies $\{|n|\leq N\}$
and set 
$  X_{\dl, N} := \P_N X_\dl $.
 Note that, for each fixed $\dl>0$ and  $x \in \T$,  
the random variable $X_{\dl, N}(x)$ is a real-valued, mean-zero  Gaussian 
random variable 
with variance
\begin{align}
\begin{split}
\s_{\dl, N}: \!& = \E \big[ X_{\dl, N}^2(x) \big]
= \frac{1}{4\pi^2} \sum_{0 < |n| \le N}   \frac{2\pi}{K_\dl(n)} \\
& \sim_\dl \log (N+1).
\end{split}
 \label{Wick1a}
\end{align}

\noi
Given an integer $k \ge 2$, 
we define the Wick ordered monomial  $\W(  X_{\dl, N}^k )
= \W_{\dl, N}(  X_{\dl, N}^k )$
by setting
\begin{align}
 \W(  X_{\dl, N}^k ) = H_k( X_{\dl, N} ;  \s_{\dl, N}   ),
 \label{WO1}
\end{align}

\noi
where $H_k(x; \s)$ is the Hermite polynomial of degree $k$; see Subsection \ref{SUBSEC:2.2}.
Then, $ \W(  X_{\dl, N}^k )$ converges, 
in $L^p(\O)$
for any finite $p \ge 1$
and also almost surely,   
 to a limit, denoted by 
$\W(  X_{\dl}^k )$, in $H^{-\eps}(\T)$
for any $\eps > 0$;
see Proposition~\ref{PROP:FN}.
In particular, the truncated renormalized
potential energy
$\int_\T \W(  X_{\dl, N}^{k+1} )dx$ converges, 
 in $L^p(\O)$
for any finite $p \ge 1$
and also almost surely,  to a limit denoted by 
$\int_\T \W(  X_{\dl}^{k+1} )dx$.

With $u_N = \P_N u$, 
we define the truncated Gibbs measure $\rho_{\dl, N}$
by\footnote{Here, with a slight abuse of notation, 
we use the notation $\W(u_N^{k+1})$ to mean $H_{k+1}(u_N; \s_{\dl, N})$.
In the following, we use the notation 
$\W(u_N^{k+1})$ with the understanding that there is the underlying Gaussian measure~$\mu_\dl$.
} 
\begin{align}
\rho_{\dl,  N} (du)
= Z_{\dl, N}^{-1} e^ {-\frac 1{k+1} \int_\T \W(u_N^{k+1})dx} 
d\mu_\dl(u).
\label{GG6}
\end{align}

\noi
We also define the truncated density $G_{\dl, N}(u) $ by 
\begin{align}
G_{\dl, N}(u) = e^{-\frac 1 {k+1}  \int_\T \W(u_N^{k+1})dx}
= e^{-\frac 1 {k+1}  \int_\T H_{k+1}(u_N ; \s_{\dl, N})dx}.
\label{WO2}
\end{align}

\noi
In view of the convergence of the truncated renormalized potential energy
mentioned above, 
we see that 
the truncated density $G_{\dl, N}$ converges
to the limiting density
\begin{align*}
G_{\dl}(u) = e^{-\frac 1 {k+1} \int_\T \W(u^{k+1})dx}
\end{align*}

\noi
 in probability with respect to $\mu_\dl$, as $N \to \infty$.
We now state the construction of the limiting Gibbs measure~$\rho_\dl$
and its convergence property in the deep-water limit.

\begin{theorem}\label{THM:Gibbs1}
Let $k \in 2\N + 1$. Then, the following statements hold.

\smallskip
\noi
\textup{(i)} Let $0 < \dl \le \infty$.
Then,  
for any finite $p \ge 1$, we have
\begin{align}
\lim_{N \to \infty} G_{\dl, N}(u) = G_\dl(u) 
\quad \text{in $L^p(\mu_\dl)$.}
\label{THM1a}
\end{align}

\noi
As a consequence, the truncated Gibbs measure $\rho_{\dl, N}$
in \eqref{GG6} converges, in the sense of~\eqref{THM1a}, 
to the limiting Gibbs measure $\rho_\dl$ given by 
\begin{align}
\begin{split}
\rho_\dl(du) 
& = Z_\dl^{-1} G_\dl(u) d\mu_\dl(u)\\
& = Z_\dl^{-1}e^{-\frac 1 {k+1} \int_\T \W(u^{k+1})dx} d\mu_\dl(u) .
\end{split}
\label{THM1b}
\end{align}

\noi
In particular, $\rho_{\dl, N}$ converges to 
$\rho_{\dl}$ in total variation.
The resulting Gibbs measure $\rho_\dl$ and the base Gaussian measure
$\mu_\dl$ are equivalent.

For $2\le \dl \le \infty$, 
 the rate of convergence \eqref{THM1a} is uniform 
 and thus the rate of convergence in total variation
 of $\rho_{\dl, N}$ to $\rho_\dl$ as $N \to \infty$ is uniform
 for $2 \le \dl \le \infty$.

\smallskip

\noi
\textup{(ii) (deep-water limit of the Gibbs measures).}
 Let $0 < \dl < \infty$. Then, the Gibbs measures $\rho_\dl$ for gILW \eqref{gILW1}
and $\rho_\BO = \rho_\infty$ for gBO \eqref{BO} constructed in Part \textup{(i)}
are equivalent.
Moreover, $\rho_\dl$ converges to $\rho_\BO$ in total variation,  as $\dl \to \infty$.

\smallskip

Furthermore, when $k =2$, 
by replacing
the truncated Gibbs measure $\rho_{\dl,  N}$ in \eqref{GG6}
by the truncated Gibbs measure with a Wick-ordered $L^2$-cutoff\textup{:}
\begin{align}
\begin{split}
\rho_{\dl, N}(du)
&= Z_{\dl, N}^{-1} 
\chi_K\bigg( \int_\T \W(u_N^2) dx \bigg)e^{-\frac 1{3} \int_\T u_N^{3} dx} 
d \mu_\dl(u), 
\end{split}
\label{THM1c}
\end{align}
 
\noi
the statements \textup{(i)} and \textup{(ii)} hold true for any fixed $K > 0$.
Namely, for each $0 < \dl \le \infty$,  
 the truncated Gibbs measure $\rho_{\dl, N}$ in \eqref{THM1c}
converges to the limiting Gibbs measure\textup{:}
\begin{align}
\begin{split}
\rho_{\dl}(du)
&= Z_{\dl}^{-1} \chi_K\bigg(\int_\T \W(u^2) dx \bigg) e^{-\frac 1{3} \int_\T u^{3} dx} 
d \mu_\dl(u)
\end{split}
\label{THM1d}
\end{align}

\noi
in the sense of the $L^p(\mu_\dl)$-convergence 
of the truncated densities as in \eqref{THM1a}.
Moreover, 
the resulting Gibbs measure $\rho_\dl$ in \eqref{THM1d} and the base Gaussian measure
endowed with the Wick-ordered $L^2$-cutoff
$\chi_K\big(\int_\T \W(u^2) dx \big) d \mu_\dl(u)$ are equivalent.
For $2\le \dl \le \infty$, 
 the rate of convergence in total variation
 of $\rho_{\dl, N}$ to $\rho_\dl$ as $N \to \infty$ is uniform
 for $2 \le \dl \le \infty$.

For  $0 < \dl < \infty$,   the Gibbs measures $\rho_\dl$ in \eqref{THM1d}
for ILW \eqref{ILW1}
and $\rho_\BO$ in \eqref{Gx3} for BO are equivalent
and, as $\dl \to \infty$, 
the Gibbs measure $\rho_\dl$ converges to $\rho_\BO$ in total variation.

\end{theorem}

%Regarding the construction of the Gibbs measures, 
Theorem \ref{THM:Gibbs1} % (and also Theorem \ref{THM:Gibbs2})
provides the first result
on the construction of the Gibbs measures for the (generalized) ILW equation
and also for the defocusing gBO equation (for $k \ge 3$).

Given a parameter-dependent Hamiltonian dynamics, 
it is of significant physical interest to study convergence of 
the associated Gibbs measures, 
which may be viewed as the first step toward 
studying convergence of dynamics at the Gibbs equilibrium.
Theorem \ref{THM:Gibbs1} (and Theorem~\ref{THM:Gibbs2})
is the first such result 
for the (generalized) ILW equation, appearing in the study of fluids.
We also mention
a series of recent breakthrough results on 
the convergence of the Gibbs measures
for quantum many-body  systems 
to that for the nonlinear Schr\"odinger equation, 
led 
by two groups 
\cite{LNR}~(Lewin, Nam, Rougerie) and 
 \cite{FKSS}~(Fr\"olich, Knowles, Schlein, and Sohinger).
See these papers for the references therein.
While these works establish only the convergence
of the Gibbs measures, 
we also establish convergence of the corresponding dynamics;
see Theorems~\ref{THM:6} and~\ref{THM:5} below.

Fix $k \in 2\N + 1$.
For each fixed $0 < \dl \le \infty$, 
the construction of the Gibbs measure
(Theorem~\ref{THM:Gibbs1}\,(i))
follows from 
 a standard application of Nelson's estimate.
The main novelty is Part (ii) of Theorem \ref{THM:Gibbs1}.
 In order to prove convergence of $\rho_\dl$ in the deep-water limit, 
 we need to estimate the truncated densities
 $G_{\dl, N}(u)$, 
{\it uniformly in both $\dl \gg 1$ and $N \in \N$}.
 One subtle point is that for different values of $\dl \gg1$, 
 the base Gaussian measures $\mu_\dl$ are different.
 In order to overcome this issue, we indeed estimate
 $G_{\dl, N}(X_{\dl})$ in $L^p(\O)$, 
 uniformly in both $\dl \gg 1$ and $N \in \N$.
 Namely, we need to directly work with the probability measure $\PP$
 on $\O$.
  See Section \ref{SEC:Gibbs1} for details.
We point out that this uniform bound 
on the truncated densities  in $\dl \gg1$ and $N \in \N$
 also plays an important role in the dynamical part, 
 which we discuss in the next subsection.
Another key ingredient in establishing convergence of the Gibbs measures
is `strong' convergence of the base Gaussian measures $\mu_\dl$ (namely, convergence in the Kullback-Leibler divergence defined in \eqref{KL4};
see Proposition \ref{PROP:equiv}).

When $k = 2$, the problem is no longer defocusing
and thus Nelson's argument is not directly applicable.
While we could adapt the argument by Tzvetkov \cite{Tz10}
for the BO equation, 
we instead  use the variational approach
as in the work \cite{OST} by Seong, Tolomeo, and the second author, 
which provides a slightly simpler argument.

\begin{remark}\label{REM:k2}\rm
We point out that, when $k = 2$, 
there is no need for a renormalization.
Indeed, recalling that $H_3(x; \s) = x^3 - 3\s x$, under the mean-zero condition, we have
\[ \int_\T \W(u_N^3) dx = \int_\T u_N^3 dx - 3 \s_{\dl, N} \int_\T u_N dx
=  \int_\T u_N^3 dx, \]

\noi
showing that a renormalization is not necessary in the $k =2$ case.
The same comment applies to Theorem \ref{THM:Gibbs2}
in the shallow-water limit.

\end{remark}

Next, we  consider the scaled gILW \eqref{gILW3}
with the Hamiltonian $\EE_\dl(v)$ in \eqref{Hv}.
Let $k \in 2\N + 1$.
For each fixed finite $\dl > 0$, the construction of the Gibbs measure $\wt \rho_\dl$ 
in \eqref{Gibbs3}
follows exactly the same lines as above.
Define  
the base Gaussian measure $ \wt \mu_\dl$ in \eqref{Gibbs4}
as the induced probability measure under the map:
\begin{align}
\o \in \O \longmapsto
 \wt X_\dl(\o) =   \frac 1{2\pi}  \sum_{n\in\Z^\ast} \frac{g_n(\o)}{| L_\dl(n)|^{ \frac{1}{2} } }e_n, 
\label{GH3}
 \end{align}

\noi
\noi
where $\{g_n\}_{n \in \Z^*}$ is as in \eqref{Gibbs7} and $L_\dl(n)$ is given by 
\begin{align}
L_\dl(n) := \frac 3\dl K_\dl(n) 
=  \frac{3in}{\dl} \ft \Gdl(n)
=  \frac{3}{ \delta}   \Big( n\coth( \delta n )-\frac{1}{ \delta} \Big).
\label{GH4}
\end{align}

\noi
From \eqref{GG3}, \eqref{GH3}, and \eqref{GH4},  we have
\begin{align}
\wt X_\dl = \sqrt{\frac{\dl}{3}} X_\dl
\label{GH4a}
\end{align}

\noi
for any $0 < \dl < \infty$.
%For each fixed $\dl > 0$, 
%we have $L_\dl(n) \sim_\dl |n|$.
Hence, by setting $\wt X_{\dl, N} = \P_N \wt X_\dl$, 
it follows from \eqref{Wick1a} that 
\noi
\begin{align}
\begin{aligned}
 \wt \s_{\dl, N}: \!&= \E \big[ \wt X_{\dl, N}^2(x) \big] 
= \frac 1 {4\pi^2} \sum_{0 < |n| \le N}   \frac{2\pi}{L_\dl(n)} \\
& = \frac\dl 3 \s_{\dl, N} \sim_\dl  \log (N+1), 
\end{aligned} 
\label{Wick11a}
\end{align} 
 
\noi
where  $\s_{\dl, N}$ is as in \eqref{Wick1a}.

Given $N \in \N$, 
we define the truncated density $\wt G_{\dl, N}(u) $ by 
\begin{align*}
\wt G_{\dl, N}(v) = e^{-\frac 1 {k+1}  \int_\T \W(v_N^{k+1})dx}, 
\end{align*}

\noi
where $v_N  = \P_N v$
and 
\begin{align}
\W(v_N^{k+1}) = H_{k+1}(v_N; \wt \s_{\dl, N}).
\label{GH5a}
\end{align}
Then, 
we define the truncated Gibbs measure $\wt \rho_{\dl,  N}$
by 
\begin{align}
\wt \rho_{\dl,  N} (dv)
= Z_{\dl, N}^{-1} e^ {-\frac 1{k+1} \int_\T \W(v_N^{k+1})dx} 
d\wt \mu_\dl(v).
\label{GH6}
\end{align}

We now state our main result on 
convergence of the Gibbs measures in the shallow-water limit.
Due to the use of the Wick renormalization for $\dl > 0$, 
we need to consider a ``renormalized'' power even in the shallow-water limit ($\dl = 0$):
\begin{align}
\rho_{\KDV}(dv)
&= Z_0^{-1} e^{-\frac 1{k+1} \int_\T \W(v^{k+1}) dx} d\wt \mu_0(v), 
\label{Gibbs5b}
\end{align}

\noi
associated with the following gKdV:
\begin{align}
\dt v + \dx^3 v = \dx \W(v^k). 
\label{KDV3}
\end{align}

\noi
Here, $\W(v^{\l})$ is given by 
\begin{align}
\W(v^{\l}) = H_{\l}(v; \s_\KDV), 
\label{s2}
\end{align}

\noi
where $\s_\KDV$ is as in \eqref{s1}.
In particular, when $\dl = 0$, 
$\W(v^{\l})$ is nothing but the usual Hermite polynomial
of degree $\l$ with the  finite variance parameter $\s_\KDV$, 
which is well defined without any limiting procedure.

\begin{theorem}\label{THM:Gibbs2}
Let $k \in 2\N + 1$. Then, the following statements hold.

\smallskip
\noi
\textup{(i)} Let $0 < \dl <  \infty$.
Then,  
for any finite $p \ge 1$, we have
\begin{align}
\lim_{N \to \infty} \wt G_{\dl, N}(v) = \wt G_\dl(v) : = 
 e^ {-\frac 1{k+1} \int_\T \W(v^{k+1})dx} 
\quad \text{in $L^p(\wt \mu_\dl)$.}
\label{THM2a}
\end{align}

\noi
As a consequence, the truncated Gibbs measure $\wt \rho_{\dl, N}$
in \eqref{GH6}  converges, in the sense of~\eqref{THM2a}, 
to the limiting Gibbs measure $\wt \rho_\dl$ given by 
\begin{align}
\begin{split}
\wt \rho_\dl(dv) 
& = Z_\dl^{-1} \wt G_\dl(v) d\wt \mu_\dl(v)\\
& = Z_\dl^{-1}e^{-\frac 1 {k+1} \int_\T \W(v^{k+1})dx} d\wt \mu_\dl(v) .
\end{split}
\label{THM2b}
\end{align}

\noi
In particular, $\wt \rho_{\dl, N}$ converges to 
$\wt \rho_{\dl}$ in total variation.
The resulting Gibbs measure $\wt \rho_\dl$ and the base Gaussian measure
$\wt \mu_\dl$ are equivalent.

For $0 <  \dl \le 1$, 
 the rate of convergence \eqref{THM2a} is uniform 
 and thus the rate of convergence in total variation
 of $\wt \rho_{\dl, N}$ to $\wt \rho_\dl$ as $N \to \infty$ is uniform
 for $0<  \dl \le 1$.

\smallskip

\noi
\textup{(ii) (shallow-water limit of the Gibbs measures).} 
Let $0 < \dl < \infty$. Then, the Gibbs measures
 $\wt \rho_\dl$  
 for the scaled gILW \eqref{gILW3}
 constructed in Part \textup{(i)}
and $ \rho_\KDV$ in \eqref{Gibbs5b} 
for gKdV \eqref{KDV3}
are mutually singular.
As $\dl \to 0$,  however, $\wt \rho_\dl$ converges weakly to $\rho_\KDV$.

\smallskip

Furthermore, when $k =2$, 
by replacing
the truncated Gibbs measure $\wt \rho_{\dl,  N}$ in \eqref{GH6}
by the truncated Gibbs measure with a Wick-ordered $L^2$-cutoff\textup{:}
\begin{align}
\wt \rho_{\dl, N}(dv)
= Z_{\dl, N}^{-1} 
\chi_K\bigg( \int_\T \W(v_N^2) dx \bigg) e^{-\frac 1{3} \int_\T v_N^{3} dx} 
d \wt \mu_\dl(v), 
\label{THM2c}
\end{align}
 
\noi
the statements \textup{(i)} and \textup{(ii)} hold true
for any fixed $K > 0$.
Namely, for each $0 < \dl < \infty$,  
 the truncated Gibbs measure $\wt \rho_{\dl, N}$ in \eqref{THM2c}
converges to the limiting Gibbs measure\textup{:}
\begin{align}
\begin{split}
\wt \rho_{\dl}(dv)
&= Z_{\dl}^{-1} 
\chi_K\bigg( \int_\T \W(v^2) dx \bigg)e^{-\frac 1{3} \int_\T v^{3} dx} 
d \wt \mu_\dl(v)
\end{split}
\label{THM2d}
\end{align}

\noi
in the sense of the $L^p(\wt \mu_\dl)$-convergence 
of the truncated densities as in \eqref{THM2a}.
Moreover, 
the resulting Gibbs measure $\wt \rho_\dl$ in \eqref{THM2d} and the base Gaussian measure
endowed with the Wick-ordered $L^2$-cutoff
$\chi_K\big( \int_\T \W(v^2) dx \big)d \wt \mu_\dl (v)$ are equivalent.
For $0 <  \dl \le 1$, 
 the rate of convergence in total variation
 of $\wt \rho_{\dl, N}$ in \eqref{THM2c} to $\wt \rho_\dl$
 in \eqref{THM2d} as $N \to \infty$ is uniform
 for $0<  \dl \le 1$.

For  $0 < \dl < \infty$,   the Gibbs measures $\wt \rho_\dl$ in \eqref{THM2d}
for the scaled ILW \eqref{gILW3} \textup{(}with $k =2$\textup{)}
and $\rho_\KDV$ in \eqref{Gibbs11} for KdV 
\textup{(}with an $L^2$-cutoff\textup{)} are mutually singular.
As $\dl \to 0$, however, 
the Gibbs measure $\wt \rho_\dl$ converges weakly to $\rho_\KDV$ in \eqref{Gibbs11}.
\end{theorem}

As compared to the deep-water limit ($\dl \to \infty$) studied in Theorem \ref{THM:Gibbs1}, 
we have an interesting phenomenon
in this shallow-water limit ($\dl \to 0$).
This is due to the fact that, while $L_\dl(n) \sim_\dl |n|$ for each $\dl > 0$, 
we have 
\begin{align*}
\lim_{\dl \to 0} L_\dl(n) = n^2
\end{align*}

\noi
for each $n \in \Z^*$.
See Lemma \ref{LEM:p1}.
This causes $\wt \mu_\dl$, $\dl > 0$, in \eqref{Gibbs4}
and the limiting Gaussian measure $\wt \mu_0$ in~\eqref{Gibbs6}
to be mutually singular.
(For each finite $\dl > 0$, the Gaussian measure 
$\wt \mu_\dl$ is supported on $H^{-\eps}(\T) \setminus L^2(\T)$, $\eps > 0$, 
whereas $\wt \mu_0$ is supported on 
$H^{\frac 12 -\eps}(\T) \setminus H^\frac 12 (\T)$, $\eps > 0$.)
In view of the equivalence of the Gibbs measures 
and the base Gaussian measures, 
the first claim in Theorem \ref{THM:Gibbs2}\,(ii) 
essentially follows from this observation.
Due to this mutual singularity, 
the mode of convergence of
the Gibbs measures  $\wt \rho_\dl$ to $\rho_\KDV$ in the shallow-water limit
is much weaker as compared to that in the deep-water limit
stated in Theorem \ref{THM:Gibbs1}\,(i).
See Section \ref{SEC:4} for details.

\begin{remark}\rm
Let $k \in 2\N+1$.
Then, the Gibbs measure 
 $\rho_\KDV$ in 
\eqref{Gibbs5b} 
for the gKdV equation is a well-defined probability measure on $H^{\frac 12 - \eps}(\T)\setminus H^\frac 12(\T)$, $\eps > 0$.
In view of \eqref{s2} with \eqref{s1}, 
we have
$0 < e^{-\frac 1{k+1} \int_\T \W(v^{k+1}) dx}
= e^{-\frac 1{k+1} \int_\T H_{k+1}(v; \s_\KDV) dx} \les 1$
on $H^{\frac 12 - \eps}(\T)$,
which is clearly integrable with respect to the base Gaussian measure $\wt \mu_0$
in \eqref{Gibbs6}.

\end{remark}

\begin{remark}\label{REM:Q2}
\rm
(i) As mentioned above, in \cite{OST}, 
Seong, Tolomeo, and the second author
proved non-normalizability of the Gibbs measure \eqref{Gx4}
(with $k = 3$)
for the focusing modified BO
(for any cutoff size $K > 0$ on the Wick-ordered $L^2$-cutoff). 
For each fixed $\dl > 0$, 
the  same argument allows us to prove  non-normalizability of 
the Gibbs measure (with $k = 3$):
\begin{align}
\begin{split}
\rho_{\dl}(du)
&= Z_\dl^{-1} 
\chi_K\bigg( \int_\T \W(u^2) dx \bigg) e^{\frac 1{k+1} \int_\T \W(u^{k+1}) dx} 
d  \mu_\dl(u)
\end{split}
\label{GH7}
\end{align}
 
\noi
for the focusing modified ILW equation \eqref{gILW7} with $k = 3$.
A straightforward modification of the argument in \cite{OST}
also yields non-normalizability of
the focusing\footnote{Recall our convention that by focusing, 
we also include the non-defocusing case, 
namely, \eqref{GH7} with $k \in 2\N$.} Gibbs measures 
\eqref{GH7} for any $k \ge 3$ and $0 < \dl \le \infty$.
For any $k \ge 3$ and $\dl > 0$, 
the same non-normalizability result also applies to the Gibbs measure:
\begin{align}
\wt \rho_{\dl}(dv)
= Z_\dl^{-1} 
\chi_K\bigg( \int_\T \W(v^2) dx \bigg) e^{\frac 1{k+1} \int_\T \W(v^{k+1}) dx} 
d  \wt \mu_\dl(v)
\label{GH8}
\end{align}

\noi
for the focusing scaled gILW (namely, \eqref{gILW3}
with the $-$ sign on the nonlinearity).

\smallskip

\noi
(ii)
In the shallow-water limit ($\dl = 0$), the  Gibbs measure $\rho_\KDV$
for the focusing gKdV (with an appropriate $L^2$-cutoff) 
exists up to the $L^2$-critical case ($k = 5$).
For each $\dl > 0$, however, the  Gibbs measure
for the focusing scaled gILW, $\dl > 0$, 
is not normalizable and thus it is not possible to study 
the convergence problem for the Gibbs measures (as well as 
dynamics at the Gibbs equilibrium) in this case. 
One possible approach may be to study convergence of the 
truncated Gibbs measure $\wt \rho_{\dl, N}$ in \eqref{GH6}
(with a Wick-ordered $L^2$-cutoff)
for the frequency-truncated scaled gILW
to the Gibbs measure $\rho_\KDV$ in \eqref{Gibbs5b}
for the focusing gKdV \eqref{KDV3}, 
by taking $N \to \infty$ and $\dl \to 0$
in a related manner.
The associated dynamical convergence  problem may be of interest as well.

\end{remark}

\subsection{Dynamical problem}

Our next goal is to study the associated dynamical problems.
More precisely, 
our goal is to  construct 
dynamics for the (scaled) gILW 
at the Gibbs equilibrium
and then to show that 
the invariant Gibbs dynamics
for the (scaled) gILW
converges to that for  gBO 
in the deep-water limit
(and for gKdV in the shallow-water limit, respectively)
in some appropriate sense.
In the following, for the sake of the presentation, 
we refer to the study of the original (unscaled) gILW equation (and the gBO equation)
for $0 < \dl \le \infty$ as the deep-water regime, 
and the study of the scaled gILW equation 
for $0 \le \dl < \infty$ (and the gKdV equation) as the shallow-water regime, 

Let us first consider 
 the deep-water regime.
 In Theorem \ref{THM:Gibbs1}, 
we constructed the Gibbs measure $\rho_\dl$ in~\eqref{THM1b}
associated with the following renormalized Hamiltonian:
\begin{align*} 
E_\dl(u) =
\frac 12   \int_\T  u \Gdl \dx u   dx
+ \frac{1}{k+1} \int_\T \W(u^{k+1}) dx, 
\end{align*}

\noi
when $k \in 2\N + 1$.
The corresponding Hamiltonian dynamics is formally given by
the following renormalized gILW:
\begin{align}
\dt u -
 \Gdl \partial_x^2 u 
           =  \dx \W(u^k)    , 
\label{WILW1}
\end{align}

\noi
which needs to be interpreted in a suitable limiting sense.
When $k = 2$, the measure construction does not require any renormalization (see Remark \ref{REM:k2})
and thus we study ILW~\eqref{ILW1}
as the corresponding dynamical problem.
As mentioned above, our first main goal is to construct dynamics
at the Gibbs equilibrium.
It is, however, a rather challenging problem
to construct strong solutions 
to these equations
with the Gibbsian initial data, 
even in a probabilistic sense.
This is mainly due to the low regularity 
(namely, $H^{-\eps}(\T) \setminus L^2(\T)$, $\eps > 0$) of the Gibbsian initial data
when $\dl > 0$.
In fact, for $0 < \dl \le \infty$, the only known case is 
for the Benjamin-Ono equation ($k = 2$
with $\dl = \infty$) by Deng~\cite{De15}, 
where he established deterministic local well-posedness
result in a space,  containing the support for the Gibbs measure, 
by a rather intricate argument
and then used Bourgain's invariant measure argument \cite{BO94}
to construct global-in-time dynamics at the Gibbs equilibrium.
By invariance, we mean that (with $\dl = \infty$ in the BO case)
\begin{align}
\rho_\dl    \big(\Phi_\dl(-t) A\big) =  \rho_\dl(A)
\label{1inv1}
\end{align}

\noi
for any measurable  set  $A\subset H^{-\eps}(\T)$ with some small  $\eps >  0$ 
and any $t \in \R$, 
where $\Phi_\dl(t)$ denotes the solution map:
\[
\Phi_\dl(t): u_0 \in H^{-\eps}(\T) \longmapsto 
u(t) = \Phi_\dl(t) u_0 \in  H^{-\eps}(\T), 
\]

\noi
satisfying the flow property
\begin{align}
\Phi_\dl(t_1 + t_2) =\Phi_\dl(t_1)\circ \Phi_\dl(t_2) 
\label{flow}
\end{align}

\noi
for any $t_1, t_2 \in \R$.
Here, we used $H^{-\eps}(\T)$ for simplicity but
it may be another Banach space, containing the support of the Gibbs measure
(as in \cite{De15}).
We also mention a recent work \cite{GKT}
on sharp global well-posedness of BO in almost critical spaces
$H^s(\T)$, $s > -\frac 12$, 
 based on the complete integrability of the equation.
When $0 < \dl < \infty$, 
the construction of strong solutions with the Gibbsian initial data
is widely open even for $k = 2$.
When $k \ge 3$, the difficulty of the problem
increases significantly and 
nothing is known up to date
for the renormalized gBO (with the Gibbs measure initial data):
\begin{align}
\dt u -
\H \partial_x^2 u 
=  \dx \W(u^k)    .
\label{BO1}
\end{align}

\noi
For example, when $k = 3$ corresponding 
to the (renormalized) modified BO equation (mBO), 
 the best known (deterministic) well-posedness result for mBO 
 is 
in $H^{\frac 12}(\T)$~\cite{GLM}, 
while
the scaling-critical space is $L^2(\T)$
and the support of the Gibbs measure is contained in $H^{-\eps}(\T)\setminus L^2(\T)$.
When 
 $0 < \dl < \infty$, we expect that the problem is much harder
due to a rather complicated, non-algebraic nature of the dispersion symbol (see~\eqref{OP}).

In this paper, 
we do not aim to construct strong solutions.
By a compactness argument, 
we instead
construct global-in-time dynamics of weak solutions 
at the Gibbs equilibrium (without uniqueness), 
including the gBO case ($\dl = \infty$).
In the deep-water limit, 
we also 
 show that
there exists a sequence $\{\dl_m\}_{m \in \N}$
of the depth parameters, tending to $\infty$, 
and 
 solutions, at the Gibbs equilibrium, 
to  the renormalized gILW~\eqref{WILW1}
with $\dl = \dl_m$, 
converging almost surely to 
solutions, at the Gibbs equilibrium,  to
the renormalized gBO~\eqref{BO1}.

\begin{theorem}
[deep-water regime]
\label{THM:6}

Let $k \in 2\N + 1$. Then, the following statements hold.

\smallskip
\noi
\textup{(i)}
Let $0 < \dl \le \infty$.
Then, there exists a set $\Si_\dl$ of full measure with respect to 
the Gibbs measure $\rho_\dl$ 
in~\eqref{THM1b}
constructed in Theorem \ref{THM:Gibbs1}
such that for every $u_0 \in \Si_\dl$, 
there exists a global-in-time solution 
$u \in C(\R; H^s(\T))$, $s < 0$,  to 
the renormalized gILW equation~\eqref{WILW1}
\textup{(}and to the renormalized gBO equation \eqref{BO1} when $\dl = \infty$\textup{)}
with \textup{(}mean-zero\textup{)} initial data $u|_{t = 0} = u_0$.
Moreover, for any $t \in \R$, 
the law of the solution $u(t)$ at time $t$ is given by 
the Gibbs measure $\rho_\dl$.

\smallskip
\noi
\textup{(ii)}
There exists an increasing sequence $\{\dl_m\}_{m \in \N} \subset \N$
tending to $\infty$
such that the following holds.

\begin{itemize}
\item
For each $m \in \N$, 
there exists a \textup{(}random\textup{)} global-in-time solution 
$u_{\dl_m} \in C(\R; H^s(\T))$, $s < 0$,  to 
the renormalized gILW equation~\eqref{WILW1}, 
with the depth parameter $\dl = \dl_m$, 
with the Gibbsian initial data distributed by the Gibbs measure $\rho_{\dl_m}$.
Moreover, for any $t \in \R$, 
the law of the solution $u_{\dl_m}(t)$ at time $t$ is given by 
the Gibbs measure $\rho_{\dl_m}$.

\smallskip

\item
As $m \to \infty$, 
$u_{\dl_m}$
converges almost surely in $C(\R; H^s(\T))$ to 
a \textup{(}random\textup{)} solution $u$ to the renormalized gBO equation~\eqref{BO1}.
Moreover, for any $t \in \R$, 
the law of the limiting solution $u(t)$ at time $t$ is given by 
the Gibbs measure $\rho_\BO = \rho_\infty$ in \eqref{THM1b}.

\end{itemize}

\smallskip

When $k = 2$, 
the statements \textup{(i)} and \textup{(ii)} hold true 
without any renormalization
\textup{(}but with the Gibbs measures $ \rho_{\dl_m}$
in~\eqref{THM1d}
and $\rho_\BO$ in \eqref{Gx3}\textup{)}.

\end{theorem}

While our construction yields only weak solutions without uniqueness,
Theorem \ref{THM:6} (and Theorem \ref{THM:5})
is the first result on the construction of 
solutions with the Gibbsian initial data
for both the (generalized)
ILW equation ($k \ge 2$)  and the gBO equation ($k \ge 3$).
Furthermore, Theorem~\ref{THM:6} presents
the first convergence result for the (generalized) ILW equation
from a statistical viewpoint.
In Theorem \ref{THM:5} below, we state an analogous
result in the shallow-water regime.

In proving Theorem \ref{THM:6}, 
we employ the compactness approach 
used in \cite{BTT1, OTh1, ORT}, 
which in turn was motivated by 
the works
of Albeverio and Cruzeiro \cite{AC} and Da Prato and Debussche~\cite{DPD} in the study of fluids. 
Our strategy is to start with the frequency-truncated dynamics (say, when $k \in 2\N+1$):
\begin{align}
\dt u_{\dl, N} -
 \Gdl \partial_x^2 u_{\dl, N} 
           =  \dx \P_N \W((\P_N u_{\dl, N})^k),
\label{WILW2}
\end{align}

\noi
which preserves the truncated Gibbs measure $\rho_{\dl, N}$ in \eqref{GG6}.
By exploiting the invariance of the truncated Gibbs measures $\rho_{\dl, N}$, 
we establish  tightness (= compactness) of the pushforward measures $\nu_{\dl, N}$
(on space-time distributions)
of the truncated Gibbs measures under the truncated dynamics \eqref{WILW2},
which implies convergence in law (up to a subsequence)
of $\{ u_{\dl, N}\}_{N \in \N}$.
Then, for each fixed $\dl \gg1$, 
the Skorokhod representation  theorem (Lemma~\ref{LEM:Sk}) allows
us to prove almost sure convergence
of the solution $u_{\dl, N}$ to \eqref{WILW2} (after changes
of underlying probability spaces)
to a limit $u$, which satisfies the renormalized gILW \eqref{WILW1}
in the distributional sense.
This part follows from exactly the same argument as those in \cite{BTT1, OTh1, ORT}.
Due to the use of the compactness, 
we only obtain 
global existence of a solution $u$ to~\eqref{WILW1}
  without uniqueness.
The main ingredient in this step is the uniform bound on 
the truncated densities $\{G_{\dl, N}\}_{N \in \N}$
in \eqref{WO2}.  Here, we only need the uniformity in $N$
for each fixed $0 < \dl < \infty$, and it is with respect to the base Gaussian measure $\mu_\dl$
 in \eqref{Gibbs2}.

A new ingredient in showing
 convergence of the gILW dynamics \eqref{WILW1} to 
the gBO dynamics~\eqref{BO}
is the uniform integrability of
the truncated densities
in {\it both} $\dl \gg1 $ and $N \in \N$
established
in Theorem \ref{THM:Gibbs1}.
As mentioned above, for different values of $\dl \gg1$, 
 the base Gaussian measures $\mu_\dl$ are different
 and thus we need to work directly with  the underlying probability measure $\PP$
 on $\O$.
 This shows tightness of the probability measures $\{ \nu_{\dl, N}\}_{\dl \gg1, N \in \N}$ 
 constructed in the first step, in both $\dl \gg1$ and $N \in \N$.
Then, by the triangle inequality for 
the 
L\'evy-Prokhorov metric, which 
characterizes
weak convergence of probability measures
on a separable metric space, 
and a diagonal argument together with the
 Skorokhod representation  theorem (Lemma~\ref{LEM:Sk}), 
we extract a sequence $\{\dl_m\}_{m \in \N}$,
tending to $\infty$, 
such that the corresponding random variables $u_{\dl_m}$
converges almost surely to a limit~$u$.
In showing that $u_{\dl_m}$
indeed satisfies the renormalized gILW \eqref{WILW1}, 
we also need to apply 
the Skorokhod representation  theorem for each $m \in \N$.

\begin{remark}\rm

Our notion of solutions constructed in Theorem \ref{THM:6}
(and Theorem \ref{THM:5}) basically corresponds to that of martingale solutions studied in the field of stochastic PDEs. See, for example, \cite{DZ}.

\end{remark}

Next, we state our dynamical result in the shallow-water regime.
In this case, we study
the following renormalized scaled gILW:
\begin{equation}
\label{WILW3}
\dt v   -  \frac{3}{ \dl}  \Gdl   \dx^2 v= \dx\W(v^k) , 
\end{equation}

\noi
generated by  the renormalized Hamiltonian:
\begin{align*}
\EE_\dl(v) =  \frac{3}{2\dl} 
\int_\T v \Gdl\dx v dx + \frac{1}{k+1} \int_\T \W( v^{k+1})dx.
\end{align*}

\noi
As in the deep-water regime, 
we construct dynamics for \eqref{WILW3} as a limit of the frequency-truncated dynamics:
\begin{equation}
\label{WILW4}
\dt v_{\dl, N}   -  \frac{3}{ \dl}  \Gdl   \dx^2 v_{\dl, N}= \dx\P_N \W((\P_N v_{\dl, N})^k) .
\end{equation}

\begin{theorem}
[shallow-water regime]
\label{THM:5}

Let $k \in 2\N + 1$. Then, the following statements hold.

\smallskip
\noi
\textup{(i)}
Let $0 < \dl <  \infty$.
Then, there exists a set $\wt \Si_\dl$ of full measure with respect to 
the Gibbs measure $\wt \rho_\dl$ 
in \eqref{THM2b}
constructed in Theorem \ref{THM:Gibbs2}
such that for every $v_0 \in \wt \Si_\dl$, 
there exists a global-in-time solution 
$v \in C(\R; H^s(\T))$, $s < 0$,  to 
the renormalized scaled gILW equation~\eqref{WILW3}
with \textup{(}mean-zero\textup{)} initial data $v|_{t = 0} = v_0$.
Moreover, for any $t \in \R$, 
the law of the solution $v(t)$ at time $t$ is given by 
the Gibbs measure $\wt \rho_\dl$.

\smallskip
\noi
\textup{(ii)}
There exists a decreasing sequence $\{\dl_m\}_{m \in \N} \subset \R_+$
tending to $0$
such that the following holds.

\begin{itemize}
\item
For each $m \in \N$, 
there exists a \textup{(}random\textup{)} global-in-time solution 
$v_{\dl_m} \in C(\R; H^s(\T))$, $s < 0$,  to 
the renormalized scaled gILW equation~\eqref{WILW3},
with the depth parameter $\dl = \dl_m$, 
with the Gibbsian initial data distributed by the Gibbs measure $\wt \rho_{\dl_m}$.
Moreover, for any $t \in \R$, 
the law of the solution $v_{\dl_m}(t)$ at time $t$ is given by 
the Gibbs measure $\wt \rho_{\dl_m}$.

\smallskip

\item
As $m \to \infty$, 
$v_{\dl_m}$
converges almost surely in $C(\R; H^s(\T))$ to 
a \textup{(}random\textup{)} solution $v$ to the gKdV equation~\eqref{KDV3}.
Moreover, for any $t \in \R$, 
the law of the limiting solution $v(t)$ at time $t$ is given by 
the Gibbs measure $\rho_\KDV$ in \eqref{Gibbs5b}.

\end{itemize}

\smallskip

When $k = 2$, 
the statements \textup{(i)} and \textup{(ii)} hold true 
without any renormalization
\textup{(}but with the Gibbs measures $\wt \rho_\dl$ in \eqref{THM2d}
and $\rho_\KDV$ in \eqref{Gibbs11}
\textup{)}.

\end{theorem}

With the uniform integrability on the truncated densities
in both $0 < \dl \le 1$ and $N \in \N$ 
(established in Theorem \ref{THM:Gibbs2}), 
Theorem \ref{THM:5}
follows from 
 exactly the same argument 
in the proof of  Theorem \ref{THM:6}
and hence we omit details.

\begin{remark}\label{REM:XX1}\rm
Theorems~\ref{THM:6} and \ref{THM:5} 
yield the construction and convergence of {\it weak} solutions.
Due to the use of a compactness argument, 
we do not have any uniqueness statement.
While these solutions are distributional solutions, 
they do not satisfy the Duhamel formulation, 
which is the usual notion for strong solutions in the study of dispersive PDEs.
Furthermore, due to the lack of uniqueness,\footnote{The solution map to the frequency-truncated equation
such as \eqref{WILW2} enjoys the flow property, and thus a suitable uniqueness statement
would imply the flow property for the limiting dynamics.} 
these solutions do not enjoy the flow property~\eqref{flow} 
and thus do not satisfy the invariance property as stated in~\eqref{1inv1}.
This is the reason why we have
a weaker invariance property in 
Theorems~\ref{THM:6} and \ref{THM:5},
which is, for example, not sufficient to imply the 
Poincar\'e recurrence property.
See \cite{STz} for a further discussion.
We also expect that a suitable uniqueness statement
would allow us to show convergence of the entire family $\{u_\dl\}_{\dl \gg 1}$
 in the deep-water limit ($\dl \to \infty$)
(and $\{v_\dl\}_{0 < \dl \ll 1}$
in the shallow-water limit ($\dl\to 0$)).

Therefore, it would be of significant interest to 
study 
probabilistic construction of strong solutions to the (scaled) ILW equation
with the Gibbsian initial data.\footnote{In view of the absolute continuity
of the Gibbs measure with respect to the base Gaussian
measure, it suffices to study probabilistic local well-posedness with the Gaussian initial data $X_\dl$ in \eqref{GG3}
(or $\wt X_\dl$ in \eqref{GH3}) in the spirit of \cite{BO96, CO, R16}.}
As mentioned above, 
the $k \ge 3$ case seems to be  out of reach at this point.
Even as for  the $k = 2$ case,
the problem is  very challenging.  
For example, in studying low regularity well-posedness of the BO equation, 
the gauge transform \cite{Tao} plays a crucial role.
For the ILW equation, however,   existence of such a gauge transform is unknown;
see
\cite[p.~128]{KS2022}.

When $k = 2$, another possible approach would be to exploit the complete integrability of the ILW equation.
In the case of the BO equation, there are recent breakthrough works \cite{GKT, KLV}
on sharp global well-posedness in $H^s(\T)$, $s > -\frac 12$.
Even with the complete integrability, however, the low regularity well-posedness
of the ILW equation seems to be very challenging.
See \cite{CLOP, CFLOP} for recent developments in this direction.

Lastly, let us point out that,  
as for the gKdV equation \eqref{KDV} (and also \eqref{KDV3}), 
there is a good well-posedness theory 
with the Gibbsian initial data; see \cite{BO93, R16, CK21}.
In particular, in a recent work \cite{CK21}, 
Chapouto and Kishimoto completed the program 
initiated by Bourgain~\cite{BO94}
on the construction of invariant Gibbs dynamics
for the (defocusing) gKdV \eqref{KDV} for any $k \in 2\N + 1$.

\end{remark}

\begin{remark}\label{REM:int}\rm

When 
$k = 2$, the ILW equation is known to be completely integrable 
with an infinite sequence of conservation laws of increasing regularities.
In this work, we study the construction and convergence
of the Gibbs measures associated with the Hamiltonians
and the corresponding dynamical problem.
For the ILW equation, it is also possible to study the
construction of invariant measures associated with the  higher order conservation laws.
See
\cite{Zhi01, TV1, TV2, TV3, DTV}
for such construction of invariant measures associated
with the higher order conservation laws in the context of the KdV and BO equations.
Once such construction is done, 
it would be of strong interest to study the related convergence problem.
We plan to address this issue in a forthcoming work
\cite{CLOZ}.
These invariant measures 
will be supported on 
smooth(er) functions
and thus this problem is  of importance even from the physical point of view.

We point out that, 
after  the completion of the current paper, 
there have been 
very recent progresses on low-regularity well-posedness and convergence issues
for the ILW equation \eqref{ILW1}, 
at the $L^2$-level 
 \cite{IS, CLOP} 
 and in negative Sobolev spaces \cite{CFLOP}.

\end{remark}

\begin{remark}\rm
In this work, we focus our attention to the circle $\T$.
From the physical point of view, 
it seems natural to study the problem on the real line.
The difficulty of this problem comes from not only the roughness of the support but also the integrability
of typical functions.
See  \cite{BO00, OQS}
for the construction of invariant Gibbs dynamics 
in the context of the nonlinear Schr\"odinger equations on the real line.
See also \cite{KMV}.
In the focusing case (including the $k = 2$ case), however, 
we expect a triviality result 
(namely, a large-torus limit of the periodic Gibbs measures
is ``trivial'' such as the Dirac delta measure on the trivial  function
(= the zero function) or a Gaussian measure);
see \cite{Rider, TW}
for such triviality results and further discussions
in the context of the Gibbs measures associated with 
the focusing nonlinear Schr\"odinger equations on the real line.

\end{remark}

\begin{remark}\rm

There are recent works
 \cite{FIH, Zine1, Zine2} on convergence of stochastic dynamics
at the Gibbs equilibrium.
One key difference between our work and these works is that, 
in \cite{FIH, Zine1, Zine2}, a single Gibbs measure remains
invariant for the entire one-parameter family of dynamics, whereas, 
in our work, the Gibbs measure (and even the base Gaussian measure) varies as the depth parameter $\dl$ changes, 
requiring us to first establish the convergence at the level of the Gibbs measures.

\end{remark}

\medskip

\noi
{\bf Organization of the paper.}
In Section \ref{SEC:2}, 
after introducing some notations, 
we study basic properties of the  variance parameters
$K_\dl(n)$ in \eqref{GG4} and 
$L_\dl(n)$ in \eqref{GH4}. 
We then go over some tools from stochastic analysis
and different modes of convergence 
for probability measures and random variables.
In Section \ref{SEC:Gibbs1}, 
we study the construction
and convergence of the Gibbs measures in the deep-water regime
and present the proof of Theorem \ref{THM:Gibbs1}.
In Section \ref{SEC:4}, 
we study the corresponding problem
in the shallow-water regime (Theorem \ref{THM:Gibbs2}).
In Section \ref{SEC:5}, 
we then study the dynamical problem
and present the proof of Theorem \ref{THM:6}.

\section{Preliminaries} \label{SEC:2}

\noi
{\bf Notations.}
By $A\les B$, we mean  $A\leq CB$ 
for some  constant $C> 0$.
We use   $A \sim  B$
to mean 
 $A\les B$ and $B\les A$.
We write $A\ll B$, if there is some small $c>0$, 
such that  $A\leq cB$.  
We may use subscripts to denote dependence on external parameters; for example,
 $A\les_{\dl} B$ means $A\le C(\dl) B$.

Throughout this paper, we fix a rich enough probability 
space $(\O, \F, \mathbb{P})$, on which all the random objects
are defined.
The realization $\o\in\O$ is often omitted in the writing.
For a random variable $X$, we denote by $\L(X)$
the law  of $X$.

We set 
$e_n(x) = e^{inx}$, $n \in \Z$
and 
$\Z^\ast = \Z \setminus \{0\}$.
Given $N \in \N$, 
let $\P_N$ be the Dirichlet projection onto the frequencies $\{|n|\leq N\}$
defined by 
\begin{align*}
\P_N f(x)         = \frac{1}{2\pi}   \sum_{|n|\le N} \ft{f} (n) e_n(x)   .
\end{align*}

\noi
Let $s \in \R$ and $1 \leq p \leq \infty$.
We define the $L^2$-based Sobolev space $H^s(\T)$
by the norm:
\begin{align*}
\| f \|_{H^s} =   \| \jb{n}^s \ft f (n) \|_{\l^2_n}.
\end{align*}

\noi
We also define the $L^p$-based Sobolev space $W^{s, p}(\T)$
by the norm:
\begin{align*}
\| f \|_{W^{s, p}} =   \|\jb{\nb}^s f\|_{L^p} = \big\| \F^{-1} [\jb{n}^s \ft f(n)] \big\|_{L^p}, 
%\label{FL1}
\end{align*}

\noi
where $\F^{-1}$ denotes the inverse Fourier transform.
When $p = 2$, we have $H^s(\T) = W^{s, 2}(\T)$. 

We use short-hand notations such as
$L^q_T H^s_x$
and $L^p_\o H^s_x$
for $L^q([-T, T]; H^s(\T))$
and $L^p(\O; H^s(\T))$, respectively.

In the following, we only work with 
real-valued functions on $\T$ or on $\R \times \T$.

\subsection{On the variance parameters}
\label{SUBSEC:2.1}

In this subsection, we establish elementary lemmas
on the variance parameters $K_\dl(n)$ in \eqref{GG4}
and $L_\dl(n)$ in \eqref{GH4}
for the Gaussian Fourier series 
$X_\dl$ in \eqref{GG3}
and
$\wt X_\dl$ in \eqref{GH3}, respectively.

\begin{lemma}  \label{LEM:p2}

Let $K_\dl(n)$  be as in  \eqref{GG4}.
Then, for any $\dl > 0$, we have 
\begin{align}
 \max \Big( 0, |n|- \frac 1\dl\Big)\leq K_\dl(n)=n \coth(\dl n) -\frac{1}{\dl} \leq  |n| ,
    \label{ub5}
\end{align}

\noi
where the above inequalities are strict for $n\neq 0$.
In particular, we have
\begin{align}
K_\dl(n) \sim_\dl |n|
\label{HH1a}
\end{align}

\noi
for any $n \in \Z^*$.
 Furthermore, for each fixed $n \in \Z^*$, 
 $K_\dl(n)$ is strictly increasing in $\dl \ge 1$
 and converges to  $|n|$ as $\dl \to \infty$.

\end{lemma}

The bound  \eqref{BOp} implies that, for $\dl \ge 2$, we have 
\begin{align}
 K_\dl(n) \geq |n| - \frac{1}{2} \sim |n|
 \label{Low_Kdl}
 \end{align}

\noi
for any $n \in \Z^*$.

\begin{proof}

For  $x\in\R\setminus \{0\}$, 
define $\mathfrak{h}$
by 
\begin{align*}
\mathfrak{h}(x) = 1 - x\coth(x) + |x|
=   1 + |x| - x \frac{  e^{x} + e^{-x} }{ e^{x} - e^{-x}}
\end{align*}

\noi
such that 
\begin{align}
K_\dl(n) = |n| -  \tfrac 1{\dl} \mathfrak{h}(\dl n).
\label{HH2}
\end{align}

\noi
In view of  \eqref{tan1}, we 
set  $\mathfrak{h} (0)=0$ such that $\mathfrak{h}$ is continuous.
We claim that  
\begin{align}
0 <   \mathfrak{h}(x)< \min( 1, |x|)
\label{ub5a}
\end{align}

\noi
for any $x \in \R \setminus\{0\}$.
Indeed, 
we first  note that $\mathfrak{h}$ is an even function.
For  $x >0$, a direct computation shows
\begin{align}
\mathfrak{h}(x)
& =1 + x - x  \frac{  e^{2x} + 1 }{ e^{2x} - 1} 
= 1 - \frac{2x}{e^{2x}-1} \in (0,1), \label{HH3}\\
\mathfrak{h}(x) 
- x & 
=   1 - x \frac{  e^{x} + e^{-x} }{ e^{x} - e^{-x}} < 0, 
\notag 
\end{align}

\noi
from which the claim \eqref{ub5a} follows.
Then, the bound \eqref{ub5} follows
from \eqref{HH2} and \eqref{ub5a}.
The equivalence \eqref{HH1a} is a direct consequence of  \eqref{ub5}
and the fact that $K_\dl(n) > 0$ for $n \in \Z^*$.

Fix $n \in \N$.
By writing
$K_\dl(n) = |n| -  \tfrac{\mathfrak{h}(\dl n)}{\dl n} n$, 
the claimed strict monotonicity of $K_\dl(n)$ in $\dl \ge 1$ 
follows once we show that $\frac{\mathfrak{h}(x)}{x}$ is strictly decreasing
and its limit as $x \to \infty$ is $0$.
A direct computation shows that 
\[ \frac d{dx}\bigg( \frac{\mathfrak{h}(x)}{x}\bigg)
= - \frac{e^{4x} - 2e^{2x} - 4x^2 e^{2x} + 1}{x^2 (e^{2x} - 1)^2}
< 0
\]

\noi
for $x \ge 1$.
Namely, $K_\dl(n)$ is increasing for $\dl \ge \frac 1n$.
From \eqref{HH3}, we have
$\frac{\mathfrak{h}(x)}{x}
= \frac1x - \frac{2}{e^{2x}-1}$, 
from which we conclude 
$\lim_{x\to \infty} \frac{\mathfrak{h}(x)}{x} = 0$.
This concludes the proof of Lemma \ref{LEM:p2}.
\end{proof}

\begin{remark}\label{REM:q1}\rm
Note that we have $q_\dl(n)
= \dl^{-1} \mathfrak{h} (\dl n)$, 
where  $q_\dl(n)$ is as in \eqref{PO1a}.
Then, \eqref{ub5a} in Lemma \ref{LEM:p2} yields 
\eqref{BOp} with the right-hand side replaced by $\frac 1 \dl$.

\end{remark}

Next, we state
basic properties of $L_\dl(n)$ defined in \eqref{GH4}.
Given $\dl > 0$, 
it follows from  $L_\dl(n) = \frac 3 \dl K_\dl(n)$ and Lemma \ref{LEM:p2}
that 
\begin{align}
  L_\dl(n) \sim_\dl |n|
   \label{Ld1}
\end{align}

\noi
for any $n \in \Z^*$.

\begin{lemma} \label{LEM:p1}

The following statements hold.

\begin{itemize}
\item
[{\rm (i)}] $0 <  L_\dl(n) <  n^2$ for any $\dl > 0$ and $n\in\Z^*$.

\smallskip

\item[{\rm (ii)}] For each $n\in\Z^\ast$, 
$L_\dl(n)$ increases  to $n^2$ as $\dl\to 0$.

\smallskip

\item[{\rm (iii)}] 
We have
\begin{align*}
L_\dl(n)\ges 
\begin{cases}
n^2, & \text{if } \dl |n| \les 1, \\
 |n|, & \text{if $\dl |n| \gg 1$ and $\dl \les 1$}.
\end{cases}
\end{align*}

\noi
In particular, the following uniform bound holds\textup{:}
\begin{align}
\inf_{0 < \dl\les 1} L_\dl(n) \ges |n|
\label{Low_Ldl}
\end{align}

\noi
for any $n \in \Z^*$.

\smallskip

\item[{\rm (iv)}] Define $h(n, \dl)$ by 
\begin{align} 
L_\dl(n) = n^2 - h(n,\dl)n^2.
\label{LL0a}
\end{align}

\noi
Then, we have 
\begin{align} \label{hnd2}
\sum_{n\in\Z} h^2(n,\dl) = \infty
\end{align}
for any $\dl>0$.

\end{itemize}

\end{lemma}

\begin{proof} 
From \eqref{GH4}, we have $L_\dl(n) = \frac 3\dl K_\dl(n)$.
Hence, from Lemma \ref{LEM:p2}, we have
$L_\dl (n) > 0$ for any $n \in \Z^*$.
On the other hand, 
from  the Mittag-Leffler expansion \cite[(11) on p.\,189]{Ahl}, we have
\begin{align}
\pi  z \coth (\pi z)  = 
 \frac{\pi z}{i}\cot\Big( \frac{\pi z}{i}\Big)
 = 1  +   \sum_{k=1}^{\infty}  \frac{  2 z^2  }{ k^2 + z^2   } 
\label{LL1}
\end{align}

\noi
for $z \in \C \setminus i \Z$.
Then,  from \eqref{GH4} and \eqref{LL1}, we have
\begin{align}
\begin{split}
L_\dl(n) 
&  = 6 n^2 \sum_{k=1}^\infty \frac{1}{k^2\pi^2 + \dl^2n^2}  
 \\
&  =     \frac{6 n^2}{\pi^2}   \sum_{k=1}^\infty \frac{1}{k^2} 
-  6 n^2    \sum_{k=1}^\infty 
\bigg(\frac{1}{k^2\pi^2}  -  \frac{1}{k^2\pi^2+\dl^2n^2} \bigg) \\
 & =   n^2  -  n^2
\sum_{k=1}^\infty \frac{ 6\dl^2n^2 }{  k^2\pi^2  ( k^2\pi^2+\dl^2n^2  ) }
\end{split}
\label{LL2}
\end{align}
	
\noi
for any $\dl>0$ and $n\in\Z$.
Hence, we conclude that $L_\dl (n) <  n^2$
for any $n \in \Z^*$.
This proves the claim (i).

From \eqref{LL0a} and \eqref{LL2}, we have 
\begin{align} 
h(n, \dl) 
=   6 \dl^2 \sum_{k=1}^\infty \frac{n^2}{k^2\pi^2(k^2\pi^2+\dl^2n^2)},
\label{LL3}
\end{align}

\noi
which tends to $0$ as $\dl \to 0$.
We also note that the expression after the first equality in \eqref{LL2}
shows that $L_\dl(n)$ is monotonic in $\dl$.
This yields  the claim (ii).
%Clearly, $h(n, \dl) \les n^2 \dl^2$.
%Moreover, 
From \eqref{LL3}, 
we see that, 
as $n \to \infty$,  
$h(n, \dl) \not\to 0$ (for each fixed $\dl > 0$),
which yields \eqref{hnd2}.
This proves the claim (iv).

Lastly, we prove (iii).
Suppose $\dl |n| \les 1$.
Then, from \eqref{LL2}, we have
 \begin{align} 
L_\dl(n)
&= 6n^2 \sum_{k=1}^\infty \frac{1}{k^2\pi^2 + \dl^2 n^2}  
 \ges n^2 \sum_{k=1}^\infty \frac{1}{k^2+ 1} \ges n^2.
 \label{LL4}
\end{align}

\noi
Now, suppose 
 $\dl |n| \gg 1$ and $\dl\les 1$.
 Then, from \eqref{LL2} and a Riemann sum approximation, we have
\begin{align}
\begin{aligned}
\frac{ L_\dl(n)}{|n|} 
&\ges \sum_{k=1}^\infty \frac{|n|}{k^2\pi^2+ \dl^2 n^2} 
=  \frac{1}{\dl} \sum_{k=1}^\infty
\frac{1}{\pi^2 (\frac{k}{\dl | n|})^2+1} \frac{1}{\dl |n|}\\
& \ges \int_0^\infty \frac{dx}{\pi^2 x^2 + 1} \ges 1.
\end{aligned}
        \label{LL5}
\end{align}

\noi
Note that the implicit constants in \eqref{LL4} and \eqref{LL5}
are independent of $\dl$.
This proves the claim~(iii).
\end{proof}

\subsection{Tools from stochastic analysis}
\label{SUBSEC:2.2}

In the following, we review some basic facts on  
the Hermite polynomials 
and the Wiener chaos estimate. 
See,  for example,  \cite{Kuo, Nualart06}.

We define the $k$th Hermite polynomials 
$H_k(x; \s)$ with variance $\s$ via the following generating function:
\begin{equation}
%F(t, x; \s) : =  
e^{tx - \frac{1}{2}\s t^2} = \sum_{k = 0}^\infty \frac{t^k}{k!} H_k(x;\s) 
\label{H1b}
 \end{equation}
 
\noi
for $t, x \in \R$ and $\s > 0$.
When $\s = 1$, we set $H_k(x) = H_k(x; 1)$.
Then, we have
\begin{align}
 H_k(x ; \s) = \s^{\frac{k}{2}}H_k(\s^{-\frac{1}{2}}x ).
\label{H1a}
 \end{align}

\noi
It is well known that $\{ H_k /\sqrt{k!}\}_{k \in \N \cup \{0\}}$
form an orthonormal basis of $L^2(\R; \frac{1}{\sqrt{2\pi}} e^{-x^2/2}dx)$.
In the following, we list the first few Hermite polynomials
for readers' convenience:
\begin{align*}
& H_0(x; \s) = 1, 
\qquad 
H_1(x; \s) = x, 
\qquad
H_2(x; \s) = x^2 - \s, \\
& H_3(x; \s) = x^3 - 3\s x, 
\qquad 
H_4(x; \s) = x^4 - 6\s x^2 +3\s^2.
\end{align*}

 \noi
From \eqref{H1b}, we obtain
  the following recursion relation: 
\begin{equation*}
\partial_{x}H_{k}(x;\s)=kH_{k-1}(x;\s)
\end{equation*}

\noi	
for any $k\in\N$, and the following identity:
\begin{align*}
 H_k(x+y) 
& = \sum_{\l = 0}^k
\begin{pmatrix}
k \\ \l
\end{pmatrix}
x^{k - \l} H_\l(y), 
\end{align*}

\noi
which, together with \eqref{H1a}, yields
\begin{align}
\begin{split}
H_k(x+y; \s )
&  = \s^\frac{k}{2}\sum_{\l = 0}^k
\begin{pmatrix}
k \\ \l
\end{pmatrix}
\s^{-\frac{k-\l}{2}} x^{k - \l} H_\l(\s^{-\frac{1}{2}} y)\\
&  = 
\sum_{\l = 0}^k
\begin{pmatrix}
k \\ \l
\end{pmatrix}
 x^{k - \l} H_\l(y; \s).
 \end{split}
\label{Herm3}
\end{align}

\smallskip

Let $\{g_n\}_{n \in \Z}$
be an independent family of 
standard complex-valued Gaussian random variables 
conditioned that  $g_n = \cj{g_{-n}}$.
We first recall the following bound:
\begin{align}
\sup_{n \in \Z}     
 \jb{n}^{-\eps}|g_n| \leq C_{\eps, \o} < \infty
\label{Gb1}
\end{align}

\noi
almost surely for some random constant $C_{\eps, \o} > 0$;
see
Lemma 3.4 in \cite{CO}.
See also Appendix in~\cite{OH3}.

We define a real-valued, mean-zero Gaussian white noise $W$ on $\T$ by 
\begin{align}\label{Wnoise}
  W(x;\o) = \sum_{n\in \Z} g_n(\o) e^{- i n \cdot x}.
\end{align}

\noi
Next,  we introduce the isonormal Gaussian process 
$\big\{W_f: f\in L^2(\T)\big\}$ 
associated to the   Gaussian white noise $W$.

\begin{definition}\label{DEF:W}\rm

The isonormal Gaussian process $\big\{W_f : f\in L^2(\T)\big\}$
is a real-valued, mean-zero  Gaussian process 
indexed by the real separable Hilbert space $L^2(\T)$ 
such that 
\begin{align*}
\E\big[ W_f W_g \big] = \jb{f, g }_{L^2_x}
\end{align*}

\noi
for $f, g\in L^2(\T)$.
Moreover, we can realize $W_f$ as follows: 
\begin{align}
 f\in L^2(\T) \longmapsto
  W_f  = \jb{f, W}_{L^2_x} = \sum_{n \in \Z} \ft f(n)  g_n(\o) , 
\label{H2a}
 \end{align}

\noi
where $W$ is as in  \eqref{Wnoise}.

\end{definition}

\begin{remark}\rm
The action \eqref{H2a} on $f$ by the white noise
is referred to as the white noise functional in \cite{OTh1, ORT}.
Note that $W_f$ is basically the `periodic' Wiener integral on $\T$.
\end{remark}

In the following, we denote $L^2(\O, \sigma\{W\}, \PP)$  the space of real-valued,
square-integrable random variables that are measurable with respect to $W$.
We present below a fundamental result in Gaussian analysis,
providing 
 us an 
orthogonal decomposition of this  $L^2$-probability space.
 Let $\G_k^W$ be the $L^2(\O)$-completion 
of  the linear span of the set 
  $
 \{ H_k(W_f): f\in L^2(\T); \|f\|_{L^2} =1\}
    $.
 We call $\G_k^W$ the $k$th Wiener chaos associated to $W$. 
 The following Wiener-Ito chaos decomposition holds:
 \begin{align} 
   L^2(\O, \s\{W\}, \PP) 
        = \bigoplus_{k = 0}^\infty \G_k^W .
\label{H3}
\end{align}

\noi
The orthogonal decomposition \eqref{H3} indicates that random variables belonging to
Wiener chaoses of different orders are uncorrelated (namely, $L^2(\O)$-orthogonal).
See also the following particular case that we will often use in our computations.

\begin{lemma}\label{LEM:W1}

Let $Y_1, Y_2$ be two real-valued, mean-zero, and jointly Gaussian random variables with variances
$\s_1 = \E[ Y_1^2] > 0$ and $\s_2 = \E[ Y_2^2] > 0$.
Then, for $k, m \in \N \cup\{0\}$, we have 
\begin{align}
\E\big[ H_k(Y_1 ; \s_1) H_m(Y_2;  \s_2)\big]
=  \ind_{k=m} \cdot k!   \big(\E[ Y_1 Y_2] \big)^k.
     \label{W1}
\end{align}

 \end{lemma}
For example,  with $f, h \in L^2(\T)$, the random variables 
$Y_1= W_f$ and $Y_2= W_h$,
with $\s_1 =  \| f\|_{L^2_x}^2$  and $\s_2 = \| h\|_{L^2_x}^2$
satisfy the identity \eqref{W1}.

Next, we state the Wiener chaos estimate, 
which is a consequence of Nelson's hypercontractivity \cite{Ne65}.
See, for example,   \cite[Theorem I.22]{Simon}.
See also \cite[Proposition 2.4]{TTz}.

\begin{lemma}[Wiener chaos estimate] \label{LEM:hyp}

Let $\{ g_n\}_{n \in \Z }$ 
be an independent family of 
standard complex-valued Gaussian random variables 
conditioned that  $g_n = \cj{g_{-n}}$.
Given $k \in \N$,  let
$\{Q_j\}_{j \in \N}$ be a sequence of polynomials in 
$ \mathbf{g}=\{ g_n\}_{n \in \Z}$ of  degrees at most $k$
such that $\sum_{j \in \N} Q_j(\mathbf{g}) \in \R$, 
almost surely.
Then, for any finite $p \ge 2$, we have 
\begin{equation*}
 \bigg\|\sum_{j \in \N} Q_j(\mathbf{g}) \bigg\|_{L^p(\O)} 
 \le (p-1)^\frac{k}{2} \bigg\|\sum_{j \in \N} Q_j( \mathbf{g}) \bigg\|_{L^2(\O)}.
\end{equation*}

\end{lemma}

Lastly, 
we provide a brief discussion on 
the Wick renormalization.

\smallskip

\noi
$\bullet$ \textbf{Wick renormalization.}
Let $\{\b_k, k\in\N\}$  be independent real-valued standard Gaussian random variables, 
which can be built from the Gaussian white noise $W$ in \eqref{Wnoise}.
 Consider the polynomial $Q(x_1, ... , x_n)$ 
 with $n$ variables. We denote its degree  by ${\rm deg}(Q)$.
Then, the random variable  $Q(\b_1, ... , \b_n)$ belongs to 
the sum of the first ${\rm deg}(Q)$ Wiener chaoses, that is, 
\[
Q(\b_1, ... , \b_n)\in\bigoplus_{k\leq {\rm deg}(Q)} \G_k^W.
\]

\noi
One can find a unique polynomial $P$ with the same degree 
and the same coefficient on the leading order term  such that 
$P(\b_1, ... , \b_n) \in \G_k^W$, 
that is, $P(\b_1, ... , \b_n) $ is the projection of $Q(\b_1, ... , \b_n)$ 
onto $\G^W_{{\rm deg}(Q)}$. 
We call such a polynomial $P$ as the Wick-ordered version of $Q$,
and we write $P = \W(Q)$.

\begin{example} \rm

(i)
 Consider the   polynomial $Q(x_1, ..., x_n) =x_1^{k_1} \cdots x_n^{k_n}$.
 Then, we have 
  \[
 \W( Q)(x_1, ..., x_n) =  \prod_{j = 1}^n H_{k_j}(x_j),
 \]
 
 \noi
 where $H_k$ is the $k$th Hermite polynomial with variance $\s = 1$.

  \vspace{1mm}
 
 \noi
 (ii)
 Given $N \in \N$, 
 consider the following truncated random Fourier series $X_{\dl, N}$:
 \begin{align}
 X_{\dl, N}(x, \o) =   \frac 1{2\pi}  \sum_{n\in\Z^\ast} \frac{g_n(\o) }{| K_\dl(n)|^{ \frac{1}{2} } }e_n.
\label{H4}
 \end{align}

\noi
Note that 
$ X_{\dl, N} = \P_N X_\dl$, where $X_\dl$ is as  in  \eqref{GG3}.
For each $x \in \T$, 
$  X_{\dl, N}(x)$ is a real-valued, mean-zero Gaussian random variable
with variance $\s_{\dl, N}$ in \eqref{Wick1a}.
Then, the Wick-ordered version of $X_{\dl, N}^k$,   $k\in\N$,
 is given by $\W(X_{\dl, N}^k) = H_k(X_{\dl, N}; \sigma_{\dl, N})$.
Compare this with \eqref{WO1}. 
 
\end{example}

\subsection{Various modes of convergence for probability measures
and random variables}
\label{SUBSEC:dist}

We conclude this section by going over various modes of convergence 
for  probability measures 
and random variables.
See, for example, \cite[Chapter~3]{Pollard}
and \cite[Chapter~2]{Tsy}
for further discussions.
See also \cite{Dudley_book}.

\medskip

\noi
$\bullet$ {\bf 
Convergence in probability
and the Ky-Fan distance.}
\\
\indent
 Let $X$ and $ Y$ be two real-valued random variables
defined on a common probability space $\O$.
Then,  the {\it Ky-Fan distance} between $X$ and $Y$ 
is defined by 
\begin{align*}
d_{\rm KF}\big( X, Y\big) = \E\big[ 1 \wedge |X-Y| \big] ,
\end{align*}

\noi
where $a\wedge b :=\min(a, b)$.
It is known that the Ky-Fan distance
 characterizes  
convergence in probability.
Namely, a sequence $\{Z_n\}_{n \in \N}$
of random variables converges in probability to some limit $Z$
if and only if $d_{\rm KF}(Z_n, Z) \to 0$ as $n \to \infty$.

The usual continuous mapping theorem \cite[Problem 5.17 on p.\,83]{Billingsley2}
states that if a sequence $\{Z_n\}_{n \in \N}$
converges to a limit $Z$ in probability, 
then, given a continuous function $\phi: \R\to \R$,  $\{\phi(Z_n)\}_{n \in \N}$
converges to $\phi(Z)$ in probability.
For our purpose, we need to extend this continuous mapping
theorem
for 
{\it uniform} convergence in probability.

\begin{lemma}[uniform continuous mapping theorem]  \label{LEM:KFC}

Let $\J \subset [0, \infty]$ be an index set.
Suppose that $\{Z_{\dl, n}\}_{n\in \N}$
converges in probability to a limit $Z_\dl$
uniformly in $\dl \in \J$, as $n \to \infty$ in the following sense\textup{:}
\begin{align}
\lim_{n\to\infty} \sup_{\dl\in \J} \dkf( Z_{\dl,n} ,  Z_{\dl})  =0
\label{uni_cvg_P1}
\end{align}

\noi
or equivalently, for any $\eta>0$,

\noi
\begin{align}
\lim_{n\to\infty} \sup_{\dl\in \J}  \PP\big(  | Z_{\dl,n} - Z_{\dl} | > \eta   \big) =0.
\label{uni_cvg_P11}
\end{align}

\noi
Suppose that the family of random variables $\{Z_\dl\}_{\dl \in \J}$
is tight, meaning that for any $\eps>0$, there exists a compact set $K_\eps\subset \R$
such that 
\begin{align}
\sup_{\dl \in \J} \PP\big(   Z_\dl \in K^c_\eps \big) \leq \eps .
\label{Z_tight}
\end{align}

 \noi
 Then, 
given any continuous function $\phi: \R\to\R$, we have 
\noi
\begin{align}
\lim_{n\to\infty} \sup_{\dl\in \J} 
\dkf \big( \phi(Z_{\dl,n}), \phi( Z_{\dl}) \big)  =0 .
\label{uni_cvg_P2}
\end{align}

\end{lemma}

Note that the tightness assumption
on $\{Z_\dl\}_{\dl \in \J}$ is crucial.

\begin{proof}

Let us first show  the equivalence of 
\eqref{uni_cvg_P1} and 
\eqref{uni_cvg_P11}. 
Let $0 < \eta < 1$.
Then, 
by Markov's inequality and \eqref{uni_cvg_P1}, we have
\begin{align*}
   \PP\big(  | Z_{\dl,n} - Z_{\dl} | > \eta   \big)
    \leq \E\Big[ \ind_{\{ | Z_{\dl,n} - Z_{\dl} | > \eta   \}} 
    \tfrac{ | Z_{\dl,n} - Z_{\dl} |\wedge \eta }{ \eta  }
          \Big] 
          \leq \frac{1}{\eta} \dkf( Z_{\dl,n} ,  Z_{\dl}) 
\end{align*}

\noi
and  
\begin{align*}
\dkf( Z_{\dl,n} ,  Z_{\dl})  =  \E\big[ | Z_{\dl,n} - Z_{\dl} |\wedge 1 \big]  \leq \eta 
 +  \PP \big( | Z_{\dl,n} - Z_{\dl} | \geq \eta \big).
\end{align*}

\noi
This proves the equivalence of \eqref{uni_cvg_P1} and 
\eqref{uni_cvg_P11}.

We now prove \eqref{uni_cvg_P2}.
Fix  $\b>0$.
In view of \eqref{Z_tight},   there exists  $\eta = \eta(\be) \geq 1$  such that 
 \begin{align}
 \sup_{\dl \in \J} \PP\big( |Z_\dl| > \eta \big) \leq \be .
    \label{tight_Zd}
 \end{align}
 
 \noi
Since $\phi$ is continuous, 
 it is uniformly continuous
 on $[-2\eta, 2\eta]$. 
 In particular,  there exists small 
  $\eps= \eps(\phi, \b)>0$
  such that
 \begin{align}\label{unif_cont}
|  \phi(x) - \phi(y) | \leq \b, \quad \text{whenever $x,y\in[-2\eta, 2\eta]$ with  $|x-y| \leq \eps$}.
 \end{align}
 
 \noi
 Without loss of generality, 
 we assume that $\eps \le \eta$.
 Note that these parameters $\e, \b$, and $\eta$ do not depend on $n\in \N$.

From \eqref{uni_cvg_P11}, we have
\begin{align}
\begin{split}
& \lim_{n\to\infty} \sup_{\dl\in \J} \E\Big[  \big( | \phi(Z_{\dl,n}) -\phi( Z_{\dl}) |\wedge 1\big) 
\ind_{\{  |  Z_{\dl,n}  -  Z_{\dl}  | > \eps \}} \Big] \\
& \quad \le \lim_{n\to\infty} \sup_{\dl\in \J} 
 \PP\big(  | Z_{\dl,n} - Z_{\dl} | > \eps   \big) =0.
\end{split}
\label{KFC1}
\end{align}

\noi
On the other hand, 
from \eqref{unif_cont} and \eqref{tight_Zd}, 
we have
\begin{align}
\begin{split}
  \E & \Big[  \big( | \phi(Z_{\dl,n}) -\phi( Z_{\dl}) |\wedge 1\big) 
\ind_{\{  |  Z_{\dl,n}  -  Z_{\dl}  | \leq \eps \}} \Big]   \\
 &\leq   \E\Big[  \big( | \phi(Z_{\dl,n}) -\phi( Z_{\dl}) |\wedge 1\big) 
\ind_{\{  |  Z_{\dl,n}  -  Z_{\dl}  | \leq \eps,  |Z_\dl|\leq \eta \}} \Big] \\
&\qquad\qquad\qquad +  \E\Big[  \big( | \phi(Z_{\dl,n}) -\phi( Z_{\dl}) |\wedge 1\big) 
\ind_{\{    |Z_\dl| > \eta \}} \Big]
  \leq 2\b   .
  \end{split}
  \label{KFC2}
 \end{align}

\noi
Since $\b>0$ is arbitrary, 
\eqref{uni_cvg_P2} follows from \eqref{KFC1} and \eqref{KFC2}.
\end{proof}

\medskip

\noi
$\bullet$ {\bf Convergence in total variation and the Hellinger distance.}
\\
\indent
 Let $\mu$ and $\nu$ be two 
probability measures on a measurable space $(E, \mathcal{E})$, 
the {\it total variation distance} 
$\dtv $ of 
$\mu$ and $\nu$ is given by 
  \begin{align}
\dtv(\mu, \nu) : = \sup\big\{  \vert \mu(A) - \nu(A) \vert : A\in\mathcal{E} \big\}.
\label{KL0}
 \end{align}
This metric induces a much stronger topology than the one induced by the weak convergence.\footnote{For example, let $\mu_N$ denote the 
law of the random variable $\frac{1}{\sqrt{N}} (Y_1+ ... + Y_N)$, 
where $Y_i, i\in\N$, are i.i.d.  random variables
with $\mathbb{P}(Y_1= 1) =\mathbb{P}(Y_1= -1) = \frac{1}{2}$.
Then, the classical central limit theorem asserts that 
$\mu_N$ converges weakly to the standard Gaussian measure on $\R$,
while due to the discrete nature of $\mu_N$, its total variation distance
from the standard Gaussian measure is always one. 
}

\smallskip

Next, we recall the notion of 
the Hellinger integral \cite{DP, DZ}.
Let $\mu$ and $\nu$ be two probability measures on 
a  measurable space $(E, \EE)$. 
Note that both $\mu$ and $\nu$ are  absolutely continuous with respect to the probability measure 
$\ld = \frac 12  (\mu + \nu)$.
Then, 
the Hellinger integral of $\mu$ and $\nu$ is defined by
\begin{align}
H(\mu, \nu) = \int_E \sqrt{\frac {d\mu}{d\ld}\frac{d\nu}{d\ld}} d\ld .
\label{KL1}
\end{align}

\noi
In fact, the definition 
 \eqref{KL1} is independent of the choice of a probability measure
 $\ld$ such that $\mu, \nu \ll \ld$.
When $\mu$ and $\nu$ are equivalent
(i.e.~mutually absolutely continuous), 
we can write $H(\mu, \nu)$ as 
\begin{align}
H(\mu, \nu) = \int_E \sqrt{\frac {d\nu}{d\mu}} d\mu .
\label{KL2}
\end{align}

\noi
Note that $0 \le H(\mu, \nu)\le 1$.
The Hellinger integral provides  a
criterion for singularity (and equivalence) of two probability measures.
It is known 
\cite[Proposition 2.20]{DZ}
that $H(\mu, \nu) = 0$ if and only if $\mu$ and $\nu$
are mutually singular.
Thus, for $\mu$ and $\nu$ to be equivalent, 
we must have $H(\mu, \nu) > 0$.
In fact, when $\mu$ and $\nu$ are product measures
on $(\R^\infty, \mathcal{B}_{\R^\infty})$, 
the condition $H(\mu, \nu) > 0$ is also  sufficient
(Kakutani's theorem). See Theorem 2.7 in \cite{DP}.

\begin{lemma}\label{LEM:Kaku2}
Let $\{\mu_n\}_{n \in \N}$ and $\{\nu_n\}_{n \in \N}$
be two sequences of probability measures on $(\R, \mathcal{B}_{\R})$
such that $\mu_n$ and $\nu_n$ are equivalent for any $n \in \N$.
Let $\mu = \bigotimes_{n \in \N} \mu_n$
and  $\nu = \bigotimes_{n \in \N} \nu_n$.
Then, we have $H(\mu, \nu) = \prod_{n \in \N} H(\mu_n, \nu_n)$
and

\smallskip
\begin{itemize}
\item
$H(\mu, \nu) > 0$ if and only if $\mu$ and $\nu$ are equivalent.
In this case, 
we have 
\begin{align}
\frac{d \mu}{d\nu} = \prod_{n \in \N} \frac{d \mu_n}{d\nu_n}.
\label{KL2a}
\end{align}

\smallskip

\item $H(\mu, \nu) = 0$ if and only if $\mu$ and $\nu$ are mutually singular.

\end{itemize}

\end{lemma}

With the notations as above, 
we  introduce the {\it Hellinger distance} $\dhh$
of $\mu$ and $\nu$ by setting
\begin{align}
\begin{split}
\dhh(\mu, \nu) & = \bigg(\frac 12 \int_E \Big(\sqrt{\frac{d \mu}{d\ld}}
- \sqrt{\frac{d \nu}{d\ld}}\Big)^2d\ld\bigg)^\frac 12 \\
& = \big(1 - H(\mu, \nu)\big)^\frac 12 , 
\end{split}
\label{KL3}
\end{align}

\noi
where $H(\mu, \nu)$ is the Hellinger integral defined in \eqref{KL2}.
It is clear that $0 \le \dhh (\mu, \nu) \le 1$.
We state 
Le Cam's inequality, relating the total variation distance and Hellinger distance;
see
Lemma 2.3 in \cite{Tsy}.\footnote{Note a slightly difference multiplicative constant
in the definition of the Hellinger distance in \cite{Tsy}.}

\begin{lemma}\label{LEM:KL1}
Let $\dtv$ and $\dhh$ be as in \eqref{KL0} and \eqref{KL3}, respectively.
Then, we have
\[ \big(\dhh(\mu, \nu)\big)^2 \le \dtv(\mu, \nu) \le \sqrt 2 \cdot \dhh(\mu, \nu)
\]

\noi
for any probability measures $\mu$ and $\nu$ on a measurable space $(E, \EE)$.
In particular, a sequence $\{\mu_k\}_{k \in \N}$
of probability measures on $(E, \EE)$  converges to some limit $\mu$
in total variation if and only if it converges to the same limit in the Hellinger distance.

\end{lemma}

In the remaining part of the paper, 
we do not make use of the Hellinger distance.
We, however, decided to introduce it here due to its connection 
to the total variation distance and
also to the fact that Hellinger integral plays an important role in 
the proof of Lemma~\ref{LEM:Kaku}.
See also Remark \ref{REM:HH}\,(iii).

\medskip

\noi
$\bullet$ {\bf Kullback-Leibler divergence (= relative entropy).}
\\
\indent
We now define
 the {\it Kullback-Leibler divergence} $\dkl (\mu, \nu)$ between $\mu$ and $\nu$ by 
 setting
\begin{align}
\dkl(\mu, \nu) = 
\begin{cases}
\displaystyle 
\int_E \log \frac{d\mu}{d\nu} d\mu, & \text{if } \mu \ll \nu, \\
\infty, & \text{otherwise}, 
\end{cases}
\label{KL4}
\end{align}

\noi
which is nothing but the relative entropy of $\mu$ with respect to $\nu$. 
While the total variation distances and the Hellinger distance
are metrics, the  Kullback-Leibler divergence is not a metric.
For example, $\dkl(\cdot, \cdot)$ is not symmetric, and moreover, 
the symmetrized version $\dkl(\mu, \nu) + \dkl(\mu, \nu)$ is not a metric, 
either.
If $\mu$ and $\nu$ are product measures
of the form $\mu = \bigotimes_{n \in \N} \mu_n$
and $\nu = \bigotimes_{n \in \N} \nu_n$, then we have
\begin{align}
\dkl(\mu, \nu) 
= \sum_{n \in \N} \dkl(\mu_n, \nu_n) .
\label{KL4a}
\end{align}

\noi
 The following lemma
 shows that 
 convergence  in  the Kullback-Leibler divergence
 (or in relative entropy)
 implies convergence in total variation
 and in the Hellinger distance.
See Lemmas~2.4 and~2.5 in \cite{Tsy}
for the proof.

\begin{lemma}\label{LEM:KL2}
Let $\dtv,$ $\dhh$, and $\dkl$ be as in \eqref{KL0}, \eqref{KL3}, and \eqref{KL4}, respectively.
Then, we have
\begin{align}
\dhh(\mu, \nu) \le 
\frac {\sqrt{\dkl(\mu, \nu)}}{\sqrt2}  
\label{KL5}
\end{align}

\noi
and
\begin{align}
\dtv(\mu, \nu) \le 
\frac {\sqrt{\dkl(\mu, \nu)}}{\sqrt2}  .
\label{KL6}
\end{align}

\end{lemma}

The second inequality \eqref{KL6} is known as Pinsker's inequality
and it is slightly stronger than  
$\dtv(\mu, \nu) \le 
\sqrt{\dkl(\mu, \nu)}$,  which follows from Lemma \ref{LEM:KL1} and \eqref{KL5}.

\medskip

\noi
$\bullet$ {\bf Weak convergence the L\'evy-Prokhorov metric.}
\\
\indent
Finally, let us introduce the L\'evy-Prokhorov metric for 
probability measures on a separable metric space $(\M, d)$.
Given $\eps > 0$, 
 we define 
an $\eps$-neighborhood of
 a measurable subset $A\subset \M$ by 
\[
A^\eps := \big\{ z\in \M: \, d(z, x)<\eps
\,\, \text{for some $x\in\M$} \big\}.
\] 

\noi
Given two probability measures $\mu$ and $\nu$ on $\M$,
their L\'evy-Prokhorov distance $d_{\rm LP}(\mu, \nu)$ is defined by
\begin{align}
\begin{split}
d_{\rm LP}(\mu, \nu)  := \inf\big\{ \eps> 0: \, 
& \mu(A) \leq \nu(A^\eps) +\eps\ 
{\rm and}\ 
\nu(A) \leq \mu(A^\eps) + \eps\\
& 
\text{for all measurable $A\subset \M$}
\big\}.
\end{split}
\label{LP1}
\end{align}

\noi
Note that the  L\'evy-Prokhorov metric is indeed a metric on 
the space of probability measures on $\M$.
It is known that 
the  L\'evy-Prokhorov metric
 induces the same topology as the topology for weak convergence.
Together with this property, 
 we only need one additional property of the  L\'evy-Prokhorov metric in this paper, 
that is, the triangle inequality; see \eqref{CQQ} below.
See \cite{Billingsley, Dudley_book} and   \cite[Section 30.3]{Bass} 
for a further discussion.

\smallskip

Lastly, 
we recall the Prokhorov theorem
and the Skorokhod representation theorem.

\begin{definition}\label{DEF:tight}\rm
Let $\J$ be any nonempty index set.
A family 
 $\{ \rho_i \}_{i \in \J}$ of probability measures
on a metric space $\M$ is said to be   tight
if, for every $\eps > 0$, there exists a compact set $K_\eps\subset \M$
such that $ \sup_{i\in \J}\rho_i(K_\eps^c) \leq \eps$. 
We say  that $\{ \rho_i \}_{i \in \J}$ is relatively compact, if every sequence 
in  $\{ \rho_i \}_{i \in \J}$ contains a weakly convergent subsequence. 

\end{definition}

Note that the index set $\J$ does not need to be countable.
We now 
recall the following 
Prokhorov theorem from \cite{Billingsley, Bass}.

\begin{lemma}[Prokhorov theorem]
\label{LEM:Pro}

If a sequence of probability measures 
on a metric space $\M$ is tight, then
it is relatively compact. 
If in addition, $\M$ is separable and complete, then relative compactness is 
equivalent to tightness.

\end{lemma}

Lastly, we recall  the following Skorokhod representation 
theorem from   \cite[Chapter 31]{Bass}.

\begin{lemma}[Skorokhod representation theorem]\label{LEM:Sk}

Let $\M$ be a complete  
separable metric space \textup{(}i.e.~a Polish space\textup{)}.
Suppose that    
 probability measures 
 $\{\rho_n\}_{n\in\N}$
 on $\M$ converges  weakly   to a probability measure $\rho$
as $n \to\infty$.
Then, there exist a probability space $(\wt \O, \wt \F, \wt\PP)$,
and random variables $X_n, X:\wt \O \to \M$ 
such that 
\begin{align*}
\L( X_n) = \rho_n
\qquad \text{and}\qquad
\L(X) = \rho ,
\end{align*}

\noi
and $X_n$ converges $\wt\PP$-almost surely to $X$ as $n\to\infty$.

\end{lemma}

\section{Gibbs measures in the deep-water regime}
\label{SEC:Gibbs1}

In this section, we go over the construction of the Gibbs measures
for the gILW equation~\eqref{gILW1}, 
including the gBO case ($\dl = \infty$),  and 
prove convergence of the Gibbs measures in the deep-water limit
(as $\dl \to \infty$).
As mentioned in Section \ref{SEC:1}, 
we construct the Gibbs measure
as a weighted Gaussian measure, 
where the base Gaussian measure is given by
$\mu_\dl$ in~\eqref{Gibbs2}
with the understanding that it is given by $\mu_\infty$ in \eqref{Gbo1}
when $\dl = \infty$.
For $0 < \dl < \infty$,
let $K_\dl(n)$ be as in \eqref{GG4}.
We extend the definition of $K_\dl(n)$
to  $\dl = \infty$ by setting
\begin{align}
K_\infty(n) = |n|,
\label{GX2}
\end{align}

\noi
which is consistent with Lemma \ref{LEM:p2}.
Then, a typical element
under the Gaussian measure $\mu_\dl$ in~\eqref{Gibbs2}
(and $\mu_\infty$ in \eqref{Gbo1})
is given by $X_\dl$ in \eqref{GG3}
when $0 < \dl < \infty$
and $X_\infty : = X_\BO$ in~\eqref{Gbo2} when $\dl = \infty$.
It is easy to see that, 
given $0 < \dl \le \infty$,  $X_\dl \in H^{-\eps}(\T)\setminus L^2(\T)$ for any $\eps > 0$, almost surely. 
Indeed,  from  Lemma \ref{LEM:p2}, 
we have
$K_\dl(n)   \sim_\dl |n|$.
%with \eqref{GH4}, we have
%$K_\dl(n)  = \frac{\dl}{3} L_\dl(n)  \ges_\dl |n|$.
Hence, with $X_{\dl, N} = \P_N X_\dl$ in \eqref{H4}, it follows from Lemma \ref{LEM:hyp} that 
there exists $C_\dl > 0$ such that, for any finite $p \ge 1$, 
\begin{align}
\| X_{\dl, N} \|_{L^p_\o H^{-\eps}_x}
\le p^\frac 12 \| \jb{\nb}^{-\eps} X_{\dl, N} (x) \|_{ L^2_xL^2_\o}
%\les p^\frac 12 \| X_{\dl, N} \|_{L^2_\o H^{-\eps}_x}
\leq C_\dl  \,  p ^\frac 12 \bigg(\sum_{0 < |n| \le N} \frac 1{|n|^{1+2\eps} }\bigg)^\frac 12
 \sim C_\dl  \, p^\frac 12, 
\label{GX3}
\end{align}

\noi
uniformly in $N \in \N$, provided that $\eps > 0$.
A similar computation together with the Borel-Cantelli lemma shows that $X_{\dl, N}$ 
converges,  
in $L^p(\O)$ and almost surely,  to the limit $X_\dl$
in $H^{-\eps}(\T)$ for any $\eps > 0$.
The fact that 
 $X_\dl \notin L^2(\T)$ almost surely 
 follows from 
Lemma B.1 in \cite{BT1}.

\smallskip

In Subsection \ref{SUBSEC:equiv}, 
 we first study various properties of the base Gaussian measures $\mu_\dl$.
See Proposition \ref{PROP:equiv}.
By restricting our attention to the defocusing case ($k \in 2\N + 1$), 
we then go over the construction of the Gibbs measures in Subsection
\ref{SUBSEC:MC1}.
In Subsection \ref{SUBSEC:conv1}, 
we continue to study the defocusing case and establish
convergence in total variation
of the Gibbs measure $\rho_\dl$ to $\rho_\BO$ 
in the deep-water limit ($\dl \to \infty$).
Finally, in Subsection \ref{SUBSEC:ILW1},
we present the proof of Theorem \ref{THM:Gibbs1} 
when $k = 2$.

\subsection{Equivalence of the base Gaussian measures}
\label{SUBSEC:equiv}

\begin{proposition} \label{PROP:equiv}

{\rm (i)} 
Let $X_\dl$ and $X_\BO$ be as in \eqref{GG3} and \eqref{Gbo2}, respectively.
Then, given any $\eps > 0$
and finite $p \ge 1$, 
$X_\dl$ converges to $X_\BO$ in $L^p(\Omega ;H^{-\e}(\T))$
and in $H^{-\eps}(\T)$ almost surely, 
as $\dl    \to \infty$.
 In particular, the Gaussian measure $\mu_\dl$ in \eqref{Gibbs2}
converges weakly to the Gaussian measure $\mubo$ in \eqref{Gbo1}, as $\dl\to\infty$.

\smallskip

\noi
{\rm (ii)} 
For any $0 < \dl < \infty$, 
the Gaussian measures $\mu_\dl$ and $ \mubo$ are equivalent.

\smallskip

\noi
{\rm (iii)} 
As $\dl \to \infty$, 
the Gaussian measure $\mu_\dl$ converges to $ \mubo$
in the Kullback-Leibler divergence defined in \eqref{KL4}. 
In particular, $\mu_\dl$ converges to $ \mubo$ in total variation.

\end{proposition}

Part (iii) of Proposition \ref{PROP:equiv}
plays an essential role in establishing convergence in total variation
of the Gibbs measure $\rho_\dl$ to $\rho_\BO$ in the deep-water limit ($\dl \to \infty$).

In proving Part (ii) of Proposition \ref{PROP:equiv}, we resort to 
the following   Kakutani's theorem~\cite{Kakutani} in the Gaussian
setting
(or  the Feldman-H\'ajek theorem \cite{Feldman, Hajek};
see also \cite[Theorem 2.9]{DP}).
See, for example, \cite{BT3, OST18, OTz, GOTW},
where Kakutani's theorem was used in the study of dispersive PDEs.
In particular, see also Proposition B.1 in \cite{BT3}.

\begin{lemma}\label{LEM:Kaku}
Let $\{A_n\}_{n\in\Z^*}$ and $\{B_n\}_{n\in\Z^*}$ be two sequences of independent, 
real-valued,  mean-zero Gaussian random variables with
$\E[ A_n^2] = a_n >0$
and  $\E[ B_n^2] = b_n >0$ for all $n\in\Z^*$.
     %
     %%
     %%%
Then, the laws of the sequences    $\{A_n\}_{n\in\Z^*}$ and $\{B_n\}_{n\in\Z^*}$ are equivalent
if and only if 
\begin{align} 
\sum_{n\in\Z^*} \Big( \frac{a_n}{b_n} -1 \Big)^2 <\infty . 
\label{cond_Kaku}
\end{align}

\noi
If they are not equivalent, then they are singular.
  
\end{lemma}

We first  present a short proof of Lemma \ref{LEM:Kaku}, 
based on Lemma \ref{LEM:Kaku2}.
See also the proof of Theorem 2.9 in \cite{DP}.

\begin{proof}[Proof of Lemma \ref{LEM:Kaku}]
Given $n \in \Z^*$, let  $\mu_n$ and $\nu_n$ denote the laws
of $A_n$ and $B_n$, respectively, 
and set 
$\mu
= \bigotimes_{n \in \Z^*} \mu_n$ and $\nu
= \bigotimes_{n \in \Z^*} \nu_n$.
Namely, 
$\mu$ are $\nu$ are the laws of the sequences    $\{A_n\}_{n\in\Z^*}$ and $\{B_n\}_{n\in\Z^*}$, 
respectively.
The Hellinger integral $H(\mu, \nu)$ defined in \eqref{KL2} is given by an infinite product:
\begin{align*}
H(\mu, \nu) 
& = \prod_{n \in \Z^*} H(\mu_n, \nu_n)
= \prod_{n \in \Z^*} \int_\R \sqrt{\frac{d \mu_n }{d \nu_n}} d\nu_n\\
& = \prod_{n \in \Z^*} \int_\R \frac{1}{\sqrt{2\pi}(a_n b_n)^\frac 14} e^{-\frac 14(\frac 1{a_n} + \frac 1{b_n})x^2 }dx\\
& = \prod_{n \in \Z^*} \frac{\sqrt 2 (a_n b_n)^\frac 14}{\sqrt{a_n + b_n}}.
\end{align*}

\noi
Thus, we have
\begin{align*}
\big( H(\mu, \nu) \big)^4
& = \prod_{n \in \Z^*} \frac{4 a_n b_n}{(a_n + b_n)^2} 
= \prod_{n \in \Z^*} \bigg(1 - \frac{(a_n- b_n)^2}{(a_n+ b_n)^2}\bigg).
\end{align*}

\noi
Hence, $H(\mu, \nu) > 0$ if and only if 
\begin{align}
 \sum_{n \in \Z^*}  \frac{(a_n- b_n)^2}{(a_n+ b_n)^2} 
= \sum_{n\in\Z^*} \Big( \frac{a_n}{b_n} -1 \Big)^2 \bigg/
\Big( \frac{a_n}{b_n} +1 \Big)^2
<\infty.
\label{Kaku1}
\end{align}

\noi
Note that the condition \eqref{Kaku1}
is equivalent to the 
condition \eqref{cond_Kaku},
since if one of the sums in~\eqref{cond_Kaku}
or \eqref{Kaku1} converges, then $\frac {a_n}{ b_n}$
must tend to $1$ as $n \to \infty$,
which implies the other sum also converges.
Then, the desired conclusion follows from Lemma \ref{LEM:Kaku2}.
\end{proof}

We now present the proof of Proposition \ref{PROP:equiv}.

\begin{proof}[Proof of Proposition \ref{PROP:equiv}]
(i) 
Let $\eps > 0$ and fix finite $p \ge1 $.
Then, it follows from Lemma \ref{LEM:hyp}, \eqref{Gbo2}, \eqref{GG3}, 
and 
$\sqrt{a} -\sqrt{b} \leq \sqrt{a-b}$ for any $a\ge b \ge 0$
together with 
\eqref{ub5} in Lemma \ref{LEM:p2} that 
\begin{align}
\begin{split}
\|  X_\dl - X_\BO \|_{L^p_\o H^{-\eps}_x}  
& \les_p \|  \jb{\nb}^{-\eps}(X_\dl - X_\BO)(x)\|_{ L^2_xL^2_\o}\\
& \sim \Bigg( \sum_{n\in\Z^\ast} \frac 1{\jb{n}^{2\eps}}
 \bigg( \frac{1}{  K_\dl^{ \frac{1}{2}}(n)    } 
 - \frac{1}{|n|^{\frac{1}{2}} }  \bigg)^2  \Bigg)^{\frac 12}\\
&\le  \bigg( \sum_{n\in\Z^\ast} \frac 1{ \jb{n}^{2\eps}} 
 \frac{  |n|- K_\dl(n)  }{ |n|    K_\dl(n) }\bigg)^\frac 12\\
&  \les \bigg( \frac{1}{\dl}   \sum_{n\in\Z^\ast}    \frac{ 1  }{ \jb{n}^{2+2\eps}  }  \bigg)^\frac 12
\les      \frac{1}{\dl^\frac 12 }  \too 0 , 
\end{split}
\label{GX4}
\end{align}

\noi
as $\dl \to \infty$.
See also \eqref{Low_Kdl}
 for  the penultimate step
in \eqref{GX4}.

As for the almost sure convergence,
we  repeat a computation analogous to \eqref{GX4}
 but with~\eqref{Gb1} in place of $\E[ |g_n|^2] \sim 1$.
Then, together with Lemma \ref{LEM:p2} (for $\dl \ge 2$), 
we have 
\begin{align}
\big\|  X_\dl(\o) - X_\BO(\o) \big\|_{H^{-\e}}^2 
\le \frac{C_{\eps_0, \o}}{\dl}
\sum_{n\in\Z^\ast}    \frac{ \jb{n}^{2\eps_0}  }{ \jb{n}^{1+2\eps}  ( |n| - \frac{1}{2} )} 
\too0,  
\label{GX4a}
\end{align}

\noi
as $\dl \to \infty$, provided that $0 < \eps_0 < \eps$.
Recalling that $\mu_\dl$ and $\mubo$
are the laws of $X_\dl$ and $X_\BO$, 
we conclude weak convergence of $\mu_\dl$ to $\mubo$.
This proves (i).

 \smallskip
 
 \noi
 (ii) Rewrite $X_\dl$ in \eqref{GG3} (and in \eqref{Gbo2} when $\dl  = \infty$
 with the understanding \eqref{GX2})
 as
 \begin{align*}
 X_\dl(\o) =     \sum_{n\in \N} \bigg( \frac{\Re g_n}{\pi K_\dl^\frac 12(n)} \cos (n x)
 - \frac{\Im g_n}{\pi K_\dl^\frac 12(n)}  \sin (nx)\bigg).
 \end{align*}

 \noi
 For $n\in\Z^*$, set 
  \[
 A_{n} = \frac{\Re g_n}{\pi K_\dl^\frac 12(n)} 
 \qquad \text{and}\qquad 
 A_{-n} = - \frac{ \Im g_n}{\pi K_\dl^\frac 12(n)} , 
 \]
 
\noi
and 
  \[
 B_{n} = \frac{\Re g_n}{\pi |n|^\frac 12} 
 \qquad \text{and}\qquad 
 B_{-n} = - \frac{ \Im g_n}{\pi |n|^\frac 12} \qquad \text{for $\dl = \infty$}
 \]

 \noi
 with $a_{\pm n} = \E[ A_{\pm n}^2] = \frac{1}{\pi K_\dl(n)}$ and  $b_{\pm n} = \E[ B_{\pm n}^2] = \frac{1}{\pi |n|}$. 
Then, from 
Lemma \ref{LEM:p2}, 
 we have
\begin{align*}
\sum_{n\in\Z^*} \Big( \frac{a_n}{b_n} -1 \Big)^2 
=  \sum_{n\in\Z^\ast} \frac{ ( |n| - K_\dl(n) )^2}{K^2_\dl(n)} 
\les C_\dl  % \frac{1}{\dl^2}
 \sum_{n\in\Z^\ast} \frac{ 1}{n^2}  <\infty.
\end{align*}

\noi
Therefore, the claimed equivalence of $\mu_\dl$ and $\mu_\infty$
follows from 
 Kakutani's theorem (Lemma~\ref{LEM:Kaku}).

\smallskip

\noi
(iii) In this part, we prove that $\mu_\dl$ converges to 
$\mubo$ in  the Kullback-Leibler divergence defined in~\eqref{KL4}.
Once this is achieved, 
convergence in total variation follows from 
Pinsker's inequality (\eqref{KL6} in Lemma \ref{LEM:KL2}).

Let us first write $\mu_\dl$, $0< \dl \le \infty$,  as the product
of Gaussian measures on $\R$ (see also~\eqref{Gibbs8}):
\begin{align*}
d\mu_\dl 
& = \bigg(\bigotimes_{n \in \N} \frac{K_\dl^\frac 12(n)}{\sqrt 2\pi}
e^{-\frac 1{2\pi} K_\dl(n) (\Re \ft u(n))^2 } d\Re \ft u(n)\bigg)\\
& \quad \times  \bigg(\bigotimes_{n \in \N} \frac{K_\dl^\frac 12(n)}{\sqrt 2\pi}
e^{-\frac 1{2\pi} K_\dl(n) (\Im \ft u(n))^2 } d\Im \ft u(n)\bigg)
\end{align*}

\noi
with the identification \eqref{GX2} when $\dl = \infty$.
With $x = (x_1, x_2) \in \R^2$, we then have
\begin{align}
d\mu_\dl 
& = \bigotimes_{n \in \N} \frac{K_\dl(n)}{ 2\pi^2}
e^{-\frac 1{2\pi} K_\dl(n) |x|^2} dx
=: \bigotimes_{n \in \N} \frac{K_\dl(n)}{ 2\pi^2}d\mu_\dl^n.
\label{KD1}
\end{align}

\noi
Then,
the Radon-Nikodym derivative $\frac{d \mu_\dl^n}{d \mu_\infty^n}$
is given by 
\begin{align}
\frac{d \mu_\dl^n}{d \mu_\infty^n}
= \frac{K_\dl(n)}{n} e^{\frac 1{2\pi} (n - K_\dl(n)) |x|^2}.
\label{KD2}
\end{align}

\noi
See \eqref{KL2a}.
Then, from 
Part (ii), 
\eqref{KL4},  and \eqref{KL4a} with \eqref{KD1} and \eqref{KD2}, we have 
\begin{align}
\begin{split}
 \dkl(\mu_\dl, \mu_\infty)
&  = \sum_{n \in \N}\dkl(\mu_\dl^n, \mu_\infty^n)
\\
&  = \sum_{n \in \N}
 \int_{\R^2}\bigg( \log \frac{K_\dl(n)}{n}
+ \frac 1{2\pi} (n - K_\dl(n)) |x|^2
\bigg)
\frac{K_\dl(n)}{ 2\pi^2}
e^{-\frac 1{2\pi} K_\dl(n) |x|^2} dx\\
&  = \sum_{n \in \N} 
\bigg( \log \frac{K_\dl(n)}{n}
+ \frac 1{2\pi} (n - K_\dl(n))
 \int_{\R^2}
\frac{K_\dl(n)}{ 2\pi^2}|x|^2 
e^{-\frac 1{2\pi} K_\dl(n) |x|^2} dx\bigg)\\
&  = \sum_{n \in \N}  \phi\bigg(\frac{n}{K_\dl(n)}\bigg), 
\end{split}
\label{KD3}
\end{align}

\noi
where $\phi(t) := t - 1 - \log t$.
Note that $\phi(1) = 0$ and $\phi'(t) > 0$ for $t > 1$.
Then, it follows from Lemma \ref{LEM:p2}
that  for each fixed $n \in \N$, 
we have
\begin{align}
 \phi\bigg(\frac{n}{K_\dl(n)}\bigg)
 \text{ decreases to } \phi(1) = 0 , 
 \label{KD4}
\end{align}

\noi
as $\dl \to \infty$,
since $\frac{n}{K_\dl(n)}$ decreases to 1 as $\dl \to \infty$.
Hence, if  the right-hand side of  \eqref{KD3} is finite
for some $\dl \gg 1$, 
then the observation~\eqref{KD4} allows us to apply
the dominated convergence theorem 
and conclude 
\begin{align*}
\lim_{\dl \to \infty}
 \dkl(\mu_\dl, \mu_\infty)
= \lim_{\dl \to \infty}\sum_{n \in \N}  \phi\bigg(\frac{n}{K_\dl(n)}\bigg)
= \sum_{n \in \N} \lim_{\dl \to \infty} \phi\bigg(\frac{n}{K_\dl(n)}\bigg)
= 0,
\end{align*}

\noi
yielding the desired convergence in the Kullback-Leibler divergence.

It remains to check that the right-hand side of \eqref{KD3} is
finite for some $\dl \gg 1$.
In fact, we show that  the right-hand side of \eqref{KD3} is
finite for any $\dl >0$.
By a direct computation, we have $\phi(t) \le (t - 1)^2$ for $t \ge 1$.
Then, from Lemma \ref{LEM:p2}, we have 
\begin{align*}
\sum_{n \in \N}  \phi\bigg(\frac{n}{K_\dl(n)}\bigg)
\le
\sum_{n \in \N}  \frac{(n - K_\dl(n))^2}{K_\dl^2(n)}
\le C_\dl  \sum_{n \in \N}\frac{1}{n^2} < \infty
\end{align*}

\noi
for any $\dl > 0$.
This concludes the proof of Proposition \ref{PROP:equiv}.
\end{proof}

\begin{remark}\label{REM:HH} \rm

(i)
By using the Wiener chaos estimate (Lemma \ref{LEM:hyp}), 
Chebyshev's inequality, and the Borel-Cantelli lemma, 
one can easily upgrade the convergence of 
$X_\dl$  to $X_\BO$ to that in $L^2(\Omega ;W^{-\eps, \infty}(\T))$
and in $W^{-\eps, \infty}(\T)$ almost surely,

\smallskip

\noi
(ii)
From \eqref{GX4}, 
we see that 
 the difference $X_\dl - X_\BO$
lives in $H^{\frac{1}{2} - \eps }(\T)$,\footnote{In fact, in $W^{\frac 12 - \eps, \infty}(\T)$ 
if we use  the Wiener chaos estimate (Lemma \ref{LEM:hyp}),} although
neither 
$X_\dl$ nor $X_\BO$  belongs to  $L^2(\T)$.

\smallskip

\noi
(iii)
In order to prove convergence of $\mu_\dl$ to $\mu_\infty$
in total variation, 
it is indeed possible to 
 directly show that 
$\mu_\dl$ converges to $\mu_\infty$ in 
the Hellinger distance $\dhh$ defined in \eqref{KL3}
and invoke Lemma~\ref{LEM:KL1}.

\end{remark}

\subsection{Construction of the Gibbs measure for the defocusing gILW equation}
\label{SUBSEC:MC1}

In this subsection, we present the construction of the Gibbs measure
for the gILW equation~\eqref{gILW1}, $0 < \dl \le \infty$
with the understanding that the $\dl = \infty$ case corresponds
to the gBO~\eqref{BO}, 
in the defocusing case: $k \in 2\N + 1$.
We treat the $k = 2$ case, corresponding to the ILW equation~\eqref{ILW1}, 
in Subsection \ref{SUBSEC:ILW1}.
Our basic strategy is to follow the argument
presented in~\cite{OTh1} on the construction of the complex $\Phi^{k+1}_2$-measures,
by utilizing the Wiener chaos estimate (Lemma~\ref{LEM:hyp})
and Nelson's estimate.
In order to establish convergence of the Gibbs measures
in the deep-water limit ($\dl \to \infty$), 
however, we need to establish an $L^p(\O)$-integrability
of the (truncated) densities,
{\it uniformly in both the frequency-truncation parameter $N \in \N$
and the depth parameter $\dl \gg 1$}.
See Proposition \ref{PROP:Gibbs1}.
This uniform bound also plays a crucial role
in the dynamical part presented in Section \ref{SEC:5}.

Fix the depth parameter $0 < \dl \le \infty$.
Given $N \in \N$, let $X_{\dl, N} = \P_N X_\dl$, 
where $X_\dl$ is defined in \eqref{GG3}:
\begin{align*}
X_{\dl, N}(\o) := \P_N X_\dl (\o)
=  
\frac 1{2\pi}  \sum_{0 < |n|\le N} \frac{g_n(\o) }{ K_\dl^\frac{1}{2}(n)  }e_n
\end{align*}

\noi
with the identification \eqref{GX2} when $\dl = \infty$.
When $\dl = \infty$, 
we also set 
\[X_{\BO, N} : = X_{\infty, N} = \P_N X_\infty = \P_N X_\BO,\]

\noi 
where $X_\BO$ is as in~\eqref{Gbo2}.
Given $k \in \N$, 
let  $\W(  X_{\dl, N}^k ) = H_k( X_{\dl, N} ;  \s_{\dl, N}   )$
denotes the Wick power defined in 
\eqref{WO1}, where $\s_{\dl, N}$ is as in \eqref{Wick1a}.
Then, the truncated Gibbs measure $\rho_{\dl,  N}$ in~\eqref{GG6}
can be written as 
\begin{align}
\begin{split}
\rho_{\dl,  N} (A)
& = Z_{\dl, N}^{-1} \int_{H^{-\eps}} \ind_{\{u \in A\}}e^ {-\frac 1{k+1} \int_\T \W(u_N^{k+1})dx} 
d\mu_\dl(u)\\
& = Z_{\dl, N}^{-1} \int_\O \ind_{\{X_{\dl}(\o) \in A\}}e^ {-\frac 1{k+1} \int_\T \W(X_{\dl, N}^{k+1}(\o))dx} 
d \PP(\o)
\end{split}
\label{ZZ2}
\end{align}

\noi
for any measurable set $A \subset H^{-\eps}(\T)$ with some small $\eps > 0$.
where $u_N = \P_N u$.
In the following, we freely interchange
the representations in terms of $X_\dl$
and in terms of $u$ distributed by $\mu_\dl$,
when there is no confusion.

Let us first
  construct the limiting Wick power $\W(X^k_\dl)$
  and the related stochastic objects.

\begin{proposition}\label{PROP:FN}

Let $k \in \N$ and $0 < \dl \le \infty$.
Given $N \in \N$, let 
$\W(  X_{\dl, N}^k )$ be as in~\eqref{WO1}.
Then, given any finite $p \ge 1$, 
the sequence  $\{ \W(X_{\dl, N}^k)\}_{N \in \N}$
is Cauchy in $L^p( \O; W^{s, \infty}(\T))$, $s < 0$, 
thus converging to a limit,  denoted by
$\W(X_{\dl}^k)$.
This convergence of $\W(X_{\dl, N}^k)$
to $\W(X_{\dl}^k)$ also holds almost surely in $W^{s, \infty}(\T)$.
Furthermore, given any finite $p \ge 1$, we have 
\begin{align}
\sup_{N \in \N} \sup_{2 \le \dl\le \infty} \big\|  \|\W(  X_{\dl, N}^k )\|_{W^{s, \infty}_x} \big\|_{L^p(\O)} 
< \infty
\label{FN1a}
\end{align}

\noi
and 
\begin{align}
\sup_{2 \le \dl \le \infty} \big\|  \| \W(  X_{\dl, M}^k ) - \W(  X_{\dl, N}^k )\|_{W^{s, \infty}_x} \big\|_{L^p(\O)} 
\too 0
\label{FN1}
\end{align}

\noi
for any $M \ge N$, tending to $\infty$.
In particular, the rate of convergence is uniform in $2\le \dl \le \infty$.

\smallskip

As a corollary, the following two statements hold.

\smallskip

\noi
\textup{(i)}
 Let
$0 < \dl \le \infty$.
Given $N \in \N$, 
let $R_{\dl, N}(u) = R_{\dl, N}(u; k+1)$ denotes the truncated potential energy defined by 
\begin{align}
R_{\dl, N}(u):=
\frac 1{k+1} \int_\T \W((\P_N u)^{k+1}) dx 
=\frac 1{k+1}
\int_{\T}  H_{k+1}\big(  \P_N u ; \s_{\dl, N} \big) dx, 
\label{FN2}
\end{align}

\noi
where $\sigma_{\dl, N}$ is  as in \eqref{Wick1a}
with the identification \eqref{GX2} when $\dl = \infty$;
see $\s_{\infty, N}$ in \eqref{CC1}.
Then,  given any finite $p \ge 1$, 
the sequence
$\{R_{\dl, N}(u)\}_{N \in \N}$
converges to the limit\textup{:}
\noi
\begin{align}
 R_\dl(u) = \frac 1{k+1}\int_{\T} \W(u^{k+1})dx = \lim_{N\to\infty}  \frac 1{k+1} \int_{\T} \W( (\P_Nu)^{k+1} )dx 
 \label{Wick2}
\end{align} 

\noi
in $L^p(d\mu_\dl)$, as $N \to \infty$.
Furthermore, 
 there exists $\ta > 0$ such that  given any finite $p \ge 1$, we have
\begin{align}
\sup_{N \in \N\cup\{\infty\}} \sup_{2\le \dl\le \infty} 
  \|  R_{\dl, N}(u)\|_{L^p(d\mu_\dl)} 
< \infty, 
\label{FN3a}
\end{align}

\noi
with $R_{\dl, \infty}(u) = R_\dl(u)$, 
and 
\begin{align}
\| R_{\dl, M}(u) - R_{\dl,N}(u) \|_{L^p(d\mu_\dl)}
\leq
 \frac{C_{k, \dl}\,  p^{\frac{k+1}{2}} }{N^\ta}
\label{FN3}
\end{align}

\noi
for 
any $ M \geq N \geq 1$.
For $2 \le \dl \le \infty$, 
we can choose the constant $C_{k, \dl}$ in \eqref{FN3} to be independent of $\dl$
and hence the rate of convergence 
of $R_{\dl, N}(u)$ to the limit $R_\dl(u)$ is uniform in $2\le \dl \le \infty$.

\smallskip

\noi
\textup{(ii)}
 Let
$0 < \dl \le \infty$.
Given $N \in \N$, 
let $F_N(u) = F_N(u; k)$ be  the truncated renormalized nonlinearity 
in  \eqref{WILW2} given 
by 
\noi
\begin{align}
F_N(u) := \dx \P_N\W((\P_N u)^k) = \dx \P_N H_k( \P_N u; \sigma_{\dl, N}),
\label{FN4}
\end{align}

\noi
where $\sigma_{\dl, N}$ is  as in \eqref{Wick1a}
with the identification \eqref{GX2} when $\dl = \infty$;
see $\s_{\infty, N}$ in \eqref{CC1}.
Then,  given any finite $p \ge 1$, 
the sequence  $\{ F_N(u)\}_{ N\in\N }$
is Cauchy in $L^p(d\mu_\dl; H^s (\T)  )$, $s < -1$, 
thus converging to a limit denoted by 
 $F(u) = \dx \W(u^k)$.
Furthermore, given any finite $p \ge 1$, we have 
\begin{align}
\sup_{N \in \N\cup\{\infty\}} \sup_{2\le \dl\le \infty} \big\|  \|  F_N(u) \|_{H^s_x} \big\|_{L^p(d\mu_\dl)} 
< \infty, 
\label{FN4a}
\end{align}

\noi
with $F_\infty(u) = F(u)$, 
and 
\begin{align}\label{FN5}
\sup_{2\le \dl\le \infty} \big\|  \| F_M(u) - F_N(u) \|_{H^s_x} \big\|_{L^p(d\mu_\dl)} 
\too 0
\end{align}

\noi
for any $M \ge N$, tending to $\infty$.
In particular, the rate of convergence 
of $F_N(u)$ to the limit $F(u)$
is uniform in $2\le \dl \le \infty$.

\end{proposition}

\begin{remark}\label{REM:FN1}\rm

In the proof of Proposition \ref{PROP:FN}, 
we use  \eqref{Low_Kdl} to obtain a lower bound
on $K_\dl(n)$, uniformly in $2 \le \dl \le \infty$, for any fixed $n \in \Z^*$.
The lower bound $\dl = 2$ is by no means sharp.
For example, in view of 
the strict monotonicity of $K_\dl(n)$ in $\dl \ge 1$ (for fixed $n \in \Z^*$)
and the fact that $K_\dl(n) \ne 0$ for $n \in \Z^*$
as stated in Lemma \ref{LEM:p2}, 
 a slight modification of the proof  of  Proposition \ref{PROP:FN} yields
the uniform (in $\dl$) bounds
for $1 \le \dl \le\infty$.
Since our main interest is to take the limit $\dl \to \infty$, 
we do not attempt to optimize a lower bound for $\dl$.
The same  comment applies to 
the subsequent results presented in this section
and hence to Theorem~\ref{THM:Gibbs1}.
\end{remark}

\begin{proof} [Proof of Proposition \ref{PROP:FN}]

Given $N \in \N$ and 
 $x, y \in \T$, 
we define $\g_N = \g_N(\dl)$ by setting
\begin{align}
%\begin{split}
\g_N(x-y):  = 
\E[  X_{\dl, N}(x)  X_{\dl, N}(y)] 
= \frac 1 {2\pi} \sum_{0 < |n|\le N} \frac{e_{n}(x-y)}{K_\dl(n)}.
%\end{split}
\label{GD1}
\end{align}

\noi
Note that we have
\begin{align*}
\g_N(x-y)  = 
\E[  X_{\dl, N}(x)  X_{\dl, M}(y)] 
\end{align*}

\noi
for any  $M \ge N \ge 1$.
In the following, 
for simplicity of notation, 
we set $u_N = \P_N u$
and suppress the $\dl$-dependence
in $\g_N = \g_N(\dl)$.

Let us first make a preliminary computation.
Given  $n, m \in \Z^*$, we have
\begin{align}
\begin{split}
& \E\Big[ \F(H_{k}( X_{\dl, N} ; \s_{\dl, N}))(n)\,
\cj{\F(H_{k}( X_{\dl, N} ; \s_{\dl, N}))(m)}\Big]\\
& =  
 \iint_{\T^2}
\E\big[H_{k}( X_{\dl, N}(x) ; \s_{\dl, N})  
H_{k}( X_{\dl, N}(y) ; \s_{\dl, N})\big] e_{-n + m}(x)  e_{-m}(x-y) dy dx.
\end{split}
\label{PZ1}
\end{align}

\noi
From Lemma \ref{LEM:W1} with \eqref{GD1}
we have 
\begin{align}
\E\big[H_{k}( X_{\dl, N}(x) ; \s_{\dl, N})  
H_{k}( X_{\dl, N}(y) ; \s_{\dl, N})\big]
= k! \, \g_N^k(y-x).
\label{PZ2}
\end{align}

\noi
Then, 
from \eqref{PZ1}, \eqref{PZ2}, a change of variables $z = y-x$, 
and integrating in $x$, we have 
\begin{align}
\begin{split}
 \E & \Big[ \F(H_{k}( X_{\dl, N} ; \s_{\dl, N}))(n)\,
\cj{\F(H_{k}( X_{\dl, N} ; \s_{\dl, N}))(m)}\Big]\\
& =  
k! \int_{\T} \bigg(\int_{\T} e_{-n + m}(x)   dx\bigg)
 \g_N^k(z)
e_{m}(z)dz
 \\
 & = 2\pi k! \ind_{n = m}\cdot
 \int_{\T} \g_N^k(z)e_n(z) dz.
\end{split}
 \label{GD1b}
\end{align}

Fix small  $\eps > 0$.
Then, by Sobolev's inequality
with finite $r \gg1 $ such that $r \eps > 1$, we have
\begin{align}
 \| \W(u_N^k)  \|_{W^{s,  \infty}}
\les   \| \W(u_N^k)  \|_{W^{s+\eps, r}}.
\label{GD1c}
\end{align}

\noi
Let $p \geq r$.
Then, by  \eqref{GD1c}, Minkowski's inequality
(with $p \ge r \gg 1$),  
the Wiener chaos estimate (Lemma \ref{LEM:hyp}), 
\eqref{GD1b}, and 
the boundedness of the torus $\T$, 
we have
\begin{align}
\begin{split}
\big\|  & \| \W(u_N^k)  \|_{W^{s, \infty}_x} \big\|_{L^p(d\mu_\dl)}  
\les p ^\frac k2 \big\|  \|\jb{\nb}^{s+\eps} \W(u_N^k)  \|_{L^2(d\mu_\dl)}  \big\|_{L^r_x} \\
& =  \frac{p ^\frac k2}{2\pi} 
\bigg\| \Big\| \sum_{n \in \Z} \jb{n}^{s+\eps} \F(H_{k}( X_{\dl, N} ; \s_{\dl, N}))(n)
e_n(x) \Big\|_{L^2(\O)}  \bigg\|_{L^r_x} \\
& = C_k \, p^{\frac{k}{2}}
\bigg( 
  \sum_{n\in \Z}  \jb{n}^{2(s+\eps)}  
   \int_{\T}  \g_N^k(z) e_{n}(z) dz
\bigg)^{\frac12}.
\end{split}
\label{GD2}
\end{align}

\noi
From  \eqref{GD1} with $ \g_N^k(z) =  \g_N^k(-z)$, we have 
\begin{align}
\int_{\T}  \g_N^k(z) e_{n}(z) dz
= \frac{1}{(2\pi)^{k-1}}\sum_{\substack{ 0< |n_j| \leq N \\   j=1, \dots , k  }} 
\frac{\ind_{ n = n_1 + \cdots + n_{k} }}{\prod_{j=1}^{k} K_\dl(n_j) }  .
\label{GD3}
\end{align}  

\noi
Hence, from \eqref{GD2} and \eqref{GD3},
and Lemma \ref{LEM:p2}, we obtain
\begin{align}
\begin{split}
\big\|  & \| \W(u_N^k)  \|_{W^{s, \infty}_x} \big\|_{L^p(d\mu_\dl)}  \\
&\le  C_{k, \dl} \, p^\frac k2  
\bigg(\sum_{\substack{ 0< |n_j| \leq N \\   j=1, \dots , k  }} 
\frac{1}{\prod_{j=1}^{k} \jb{n_j} }  \jb{n_1+ \cdots + n_{k}}^{2(s+\eps)} \bigg)^\frac 12\\
&\leq C_{k, \dl} \, p^\frac k2  
\bigg(
\sum_{n_1, ... ,n_{k}\in\Z^*} 
\frac{1}{\prod_{j=1}^{k} \jb{n_j} }  \jb{n_1+ \cdots + n_{k}}^{2(s+\eps)} \bigg)^\frac 12< \infty,
\end{split}
\label{GD4}
\end{align}

\noi
uniformly in $N \in \N$, 
provided that $s +\eps< 0$.
This last condition can be guaranteed for $s< 0$ 
by taking $\eps > 0$ sufficiently small.
In view of  \eqref{Low_Kdl}, 
the bound \eqref{GD4} holds uniformly in $2\le \dl \le \infty$
(namely, the constant $C_{k, \dl}$ can be chosen to be independent
of $2\le \dl \le \infty$).
This proves \eqref{FN1a}.

Let  $M\geq N \geq 1$ and  $p\geq 2$.
Proceeding as above, we have 
\begin{align}
\begin{split}
 \big\| &   \|\W(u_M^k)  - \W(u_N^k) \|_{W^{s, \infty}_x} \big\|_{L^p(d\mu_\dl)}  \\
& \le C_k\,  p^\frac k2
\bigg( 
  \sum_{n}  \jb{n}^{2(s+\eps)}   \int_{\T}
  \big(  \g_M^k(z) -  \g_N^k(z)\big) e_{n}(z) dz
\bigg)^{\frac12}\\
&\le  C_{k, \dl} \, p^\frac k2  \bigg(
\sum_{\substack{ 0< |n_j| \leq M \\   j=1, \dots , k  }} 
\frac{1}{\prod_{j=1}^{k} \jb{n_j} }  \jb{n_1+ \cdots + n_{k}}^{2(s+\eps)} \\
& \hphantom{XXXXXXX}- \sum_{\substack{ 0< |n_j| \leq N \\   j=1, \dots , k  }} 
\frac{1}{\prod_{j=1}^{k} \jb{n_j} }  \jb{n_1+ \cdots + n_{k}}^{2(s+\eps)} \bigg)^\frac 12\\
&\le  C_{k, \dl} \, p^\frac k2 \bigg(
\sum_{\substack{  0< |n_j| \leq M \\   j=1, \dots , k  }} 
\frac{\ind_{\max_{j = 1, \dots, k}  |n_j| > N}}{\prod_{j=1}^{k} \jb{n_j} }  \jb{n_1+ \cdots + n_{k}}^{2(s+\eps)} 
\bigg)^\frac 12\\
& \le C_{k, \dl} \, p^\frac k2 N^{\max(s, -\frac 12) +2\eps}
\end{split}
\label{GD5}
\end{align}

\noi
for any $\eps > 0$, 
provided that $s < 0$.
By choosing $0 < 2\eps < \min\big(-s, \frac 12 \big)$, 
we then obtain 
\begin{align}
 \big\|    \|\W(u_M^k)  - \W(u_N^k) \|_{W^{s, \infty}_x} \big\|_{L^p(d\mu_\dl)}  
 \too 0, 
 \label{GD6}
\end{align}

\noi
as $N \to \infty$.
In view of  \eqref{Low_Kdl}, 
the bound \eqref{GD5} holds uniformly in $2\le \dl \le \infty$
and thus the convergence in \eqref{GD6}
holds uniformly in $2 \le \dl \le \infty$, 
yielding \eqref{FN1}.

By applying Chebyshev's inequality 
(see also Lemma 4.5 in \cite{Tz10}), 
to \eqref{GD5} (with $M = \infty$)
and summing over in $N \in \N$
we have
\begin{align*}
\sum_{N = 1}^\infty
\PP\bigg( \|\W(u^k)  - \W(u_N^k) \|_{W^{s, \infty}} > \frac{1}{j} \bigg)
& \les 
\sum_{N = 1}^\infty e^{-c N^{- \frac{2}{k}(\max(s, -\frac 12) +2\eps)} j^{-\frac{2}{k}}}\\
& \les e^{-c' j^{-\frac 2k}}< \infty.
\end{align*}

\noi
Therefore, we conclude from the Borel-Cantelli lemma that 
there exists $\O_j$ with $\PP(\O_j) = 1$
such that 
for each $\o \in \O_j$, 
there exists $N_j = N_j(\o) \in \N$ such that 
\[
\|\W(u^k)(\o)  - \W(u_N^k))(\o) \|_{W^{s, \infty}}  <\frac{1}{j}\]
 for any $N \geq N_j$.
By setting $\Sigma = \bigcap_{j = 1}^\infty \O_j$, 
we have $\PP(\Si) = 1$.
Hence, we conclude that $\W(u_N^k)$ converges almost surely to $\W(u^k)$ in $W^{s,\infty}(\T)$.

Let us briefly discuss how to obtain the corollaries (i) and (ii).
We only discuss the difference estimates \eqref{FN4} and \eqref{FN5}.
The first corollary  on $R_{\dl, N}(u)$ (Part (i)) 
easily follows from the discussion above (in particular \eqref{GD5} with $k$ replaced by $k+1$) by noting that 
\[  |R_{\dl, M}(u)  -R_{\dl, N}(u) | \le C_k \| \W(u_M^{k+1})  - \W(u_N^{k+1}) \|_{H^s}\]

\noi
for any $s < 0$. We can take $s = -\frac 12$ for example.

As for the second corollary on $F_N(u)$, 
we just need to note that 
\begin{align}
\begin{split}
 \big\| &   \|F_M(u)  - F_N(u) \|_{H^s} \big\|_{L^p(d\mu_\dl)}  \\
& \le  \big\|    \|(\P_M - \P_N) \W(u_M^k)  \|_{H^{s+1}} \big\|_{L^p(d\mu_\dl)}  \\
& \quad +  \big\|    \|\W(u_M^k)  - \W(u_N^k) \|_{H^{s+1}} \big\|_{L^p(d\mu_\dl)}  \\
& =: \I + \II.
\end{split}
\label{GD7}
\end{align}

\noi
For $s < -1$, 
we can estimate $\II$ in \eqref{GD7} just as in \eqref{GD5}.
As for the first term $\I$ in \eqref{GD7}, we note that 
due to the projection $\P_M - \P_N$, we have
$|n| = |n_1 + \cdots + n_k| > N $ in 
a computation analogous to \eqref{GD4}, 
which in particular implies 
$\max_{j = 1, \dots, k}  |n_j| \ges_k N$.
Hence, a slight modification of \eqref{GD6} yields the desired bound
\eqref{FN5}.
\end{proof}

%\subsection{Nelson's estimate.}

Next, we study the densities
for the truncated Gibbs measures $\rho_{\dl, N}$ in \eqref{ZZ2}.
As mentioned above, we restrict our attention to the defocusing case
in this subsection.
Namely, we fix $k \in 2\N+1$.
See Subsection \ref{SUBSEC:ILW1}
for the $k = 2$ case.
Given $0 < \dl \le \infty$
and $N \in \N$, 
let $G_{\dl,  N}(u)$ be the truncated density
defined in \eqref{WO2}.
Our main goal is to establish an $L^p$-integrability of 
the truncated density
 $G_{\dl,  N}(u)$
 for the following two purposes:
 
 \begin{itemize}
 \item
 In order to construct the limiting Gibbs measure $\rho_\dl$ 
 for each fixed $0 < \dl \le \infty$
 (Theorem \ref{THM:Gibbs1}\,(i)), 
 we establish such an
$L^p$-integrability of 
the truncated density, 
uniformly in $N \in \N$ but for each fixed $0 < \dl \le \infty$.

\item In order to prove convergence of the Gibbs measures
in the deep-water limit  (Theorem~\ref{THM:Gibbs1}\,(ii)), 
we establish
 an
$L^p$-integrability of 
the truncated density, 
uniformly in 
 both $N \in \N$ and $\dl \gg 1$. %$2 \le \dl \le \infty$.
 
 \end{itemize}
Here, we need to study the $L^p$-integrability of  $G_{\dl,  N}(u)$
  with respect to the Gaussian measure~$\mu_\dl$ in \eqref{Gibbs2},
which is different for different values of $\dl$.
In order to establish a uniform (in $\dl$) bound, 
it is therefore more convenient to work with the Gaussian process $X_\dl$
and the underlying probability measure $\PP$ on $\O$.

Given $0 < \dl \le \infty$ 
and $N \in \N$, 
we define $ G_{\dl, N}(X_\dl)=  G_{\dl, N}(X_\dl; k+1)$ by 
\begin{align*}
 G_{\dl, N}(X_\dl) = e^{- R_{\dl, N}(X_\dl)    } 
                           =  e^{-  \frac 1{k+1}\int_\T  \W( X_{\dl, N}^{k+1} ) dx   } ,
\end{align*}

\noi
where $R_{\dl, N}(X_\dl) = R_{\dl, N}(X_\dl; k+1)$
is  the truncated potential energy defined 
in \eqref{FN2}.

\begin{proposition}\label{PROP:Gibbs1}

Let $k \in 2 \N + 1$
and fix finite $p \ge 1$.
Given any $0 < \dl \le \infty$, we have
\begin{align}
&\sup_{N\in \N} \| G_{\dl, N}(X_\dl)  \|_{L^p(\O)} 
=
\sup_{N \in \N}\| G_{\dl, N}(u) \|_{L^p(d\mu_\dl)}
 \le C_{p, k, \dl} < \infty.
 \label{GN0a}
\end{align}

\noi
In addition, the following uniform bound holds for $2 \le \dl \le \infty$\textup{:}
\begin{align}
\begin{split}
 \sup_{N \in \N} \sup_{2 \le \dl \le \infty}
\| G_{\dl, N}(X_\dl)  \|_{L^p(\O)}
& = 
\sup_{N \in \N} \sup_{2 \le \dl \le \infty}
\| G_{\dl, N}(u)  \|_{L^p(d\mu_\dl)}\\
& 
 \le C_{p, k} < \infty.
\end{split}
 \label{GN0b}
\end{align}

Define $G_\dl(X_\dl) = G_{\dl, \infty}(X_\dl)$  by 
\[
G_\dl(X_\dl) = e^{- R_\dl(X_\dl)}
\]

\noi
with $R_\dl(X_\dl)$  as in \eqref{Wick2}.
Then, 
$G_{\dl, N}(X_\dl)$ converges to $G_{\dl}(X_\dl)$
in $L^p(\O)$.
Namely, 
we have
\begin{align} 
\lim_{N\to\infty} \| G_{\dl, N}(X_\dl) - G_{\dl}(X_\dl) \|_{L^p(\O)} & =0.
\label{GN0c}
\end{align}

\noi
Furthermore, 
the convergence is uniform in $2 \le \dl \le \infty$\textup{:}
\begin{align}
\lim_{N\to\infty}\sup_{2 \le \dl \le \infty} \| G_{\dl, N}(X_\dl) - G_{\dl}(X_\dl) \|_{L^p(\O)} =0.
\label{GN0d}
\end{align}

\noi
As a consequence, the uniform bounds \eqref{GN0a}
and \eqref{GN0b} hold even if we replace the supremum in $N \in \N$
by the supremum in $N \in \N \cup\{\infty\}$.

\end{proposition}

Theorem \ref{THM:Gibbs1}\,(i)
follows as a directly corollary to  
 Proposition \ref{PROP:Gibbs1},
 allowing us to define the limiting
 Gibbs measure $\rho_\dl$ in \eqref{THM1b}.
 Fix $0 < \dl \le \infty$.
 Then, \eqref{GN0c} with $p = 1$
 implies that the partition function 
 $Z_{\dl, N} = \|G_{\dl, N}(u)\|_{L^1(d\mu_\dl)}$
 of the truncated Gibbs measure $\rho_{\dl, N}$ in \eqref{GG6}
 converges to 
 the partition function 
 $Z_{\dl} = \|G_{\dl}(u)\|_{L^1(d\mu_\dl)}$
of the Gibbs measure $\rho_\dl$ in \eqref{THM1b}.
Let  $\mathcal{B}_{H^{-\eps}}$ denote
the collection of Borel sets in $H^{-\eps}(\T)$.
Then, once again from \eqref{GN0c}, 
we have
\begin{align}
\begin{split}
& \lim_{N\to \infty}
\sup_{A \in \mathcal{B}_{H^{-\eps}}}
| \rho_{\dl, N}(A) - \rho_{\dl}(A) |\\
&\quad =  \lim_{N\to \infty}
\sup_{A \in \mathcal{B}_{H^{-\eps}}}
\bigg| \frac{Z_{\dl, N}}{Z_\dl} \rho_{\dl, N}(A) - \rho_{\dl}(A) \bigg|\\
 &\quad\leq Z_\dl ^{-1}
 \lim_{N\to \infty}
\sup_{A \in \mathcal{B}_{H^{-\eps}}}
  \int_{H^{-\e} } \ind_A(u)
| G_{\dl, N}(u) -  G_\dl(u) |d \mu_\dl(u) \\
& \quad \le Z_\dl^{-1}\lim_{N\to\infty} \| G_{\dl, N}(X_\dl) - G_{\dl}(X_\dl) \|_{L^1(\O)}\\
&\quad  = 0.
\end{split}
\label{GN0}
\end{align}

\noi
This proves convergence in total variation of $\rho_{\dl, N}$
to $\rho_\dl$.
By using \eqref{GN0d} in place of \eqref{GN0c},
a slight modification of the argument above yields uniform convergence
in total variation
of $\rho_{\dl, N}$ to $\rho_\dl$ for $2 \le \dl \le \infty$.
See \eqref{CX1} below.
We omit details.

We now present the proof of
 Proposition~\ref{PROP:Gibbs1}.

\begin{proof}[Proof of Proposition~\ref{PROP:Gibbs1}] 

We break the proof into two steps.

\smallskip

\noi
$\bul$ \textbf{Step 1:} We  first prove the uniform $L^p$-bounds
  \eqref{GN0a} and   \eqref{GN0b}.
Given $k \in 2\N+1$, 
the Hermite polynomial $H_{k+1}$ has a global minimum; 
there exists finite $a_{k+1} > 0$ such that 
$H_{k+1}(x) \geq -a_{k+1}$ for any $x\in\R$.  It follows from  \eqref{H1a}
that %for $\s>0$,
\begin{align}
 H_{k+1}(x; \s ) \geq - \s^{\frac{k+1}{2}} a_{k+1}
     \label{GN1}
  \end{align}

\noi
for any $x \in \R$ and $\s > 0$.
Hence, from \eqref{FN2} and \eqref{GN1}
with \eqref{Wick1a}, we have
\begin{align}
\begin{aligned}
- R_{\dl,N}(X_\dl) 
&=   -  \frac 1{k+1}  \int_{\T} H_{k+1}( X_{\dl, N} ; \s_{\dl, N})  dx   \\
&\leq  \frac{2\pi}{k+1} \s_{\dl,N}^{\frac{k+1}{2}}  a_{k+1} 
 \leq A_{k, \dl}    (\log (N+1))^{\frac{k+1}{2}}
\end{aligned}
\label{N2}
\end{align}

\noi
for some $A_{k, \dl} > 0$,
uniformly in $N \in \N$.
The bound \eqref{N2} is exactly
where the defocusing nature of the equation plays a crucial role.

\begin{remark}\label{REM:GN1}\rm
Recall the uniform lower bound 
\eqref{Low_Kdl} for $2\le \dl \le \infty$
(with the identification~\eqref{GX2} when $\dl = \infty$).
In view of  \eqref{Wick1a}, 
we can then choose $A_{k, \dl}$ to be independent of 
$2 \le \dl \le \infty$
(and $N \in \N$)
as in the proof of Proposition \ref{PROP:FN}.
Similarly, 
by restricting our attention to $2 \le \dl \le \infty$, 
we can choose the constant
 $c_{k, \dl}$ in \eqref{N1} below to be 
independent of $2\le \dl \le \infty$  
since the constant $C_{k, \dl}$ in \eqref{FN3} is independent of
$2\le \dl \le \infty$.
As a result, the constants
in 
$B_{k, \dl, p}$
and $C_2 (k, \dl, p) $
 in~\eqref{N4} below
can be chosen to be 
independent of $2 \le \dl \le \infty$.

\end{remark}

By applying
 Proposition \ref{PROP:FN}\,(i)
and  Chebyshev's inequality
(see also Lemma 4.5 in \cite{Tz10}), 
we have, for some $C_ 1 > 0$ and $c_{k, \dl} > 0$, 
\begin{align}
\PP \Big(p | R_{\dl,M}(X_\dl) - R_{\dl,N}(X_\dl)| > \ld \Big) 
 \leq C_1   e^{-c_{k,\dl}\, p^{-\frac2{k+1}} N^{\frac{2\ta }{k+1}}   \ld^{\frac2{k+1}}} 
\label{N1}
\end{align}

\noi
for any $ M \geq N \geq 1$
and any $p, \ld > 0$.

By writing
\begin{align*}
\|G_{\dl,N}(X_\dl)\|^{p}_{L^{p}(\O)}
&= \int_0^\infty  \PP\Big ( e^{-p R_{\dl, N}(X_\dl)}> \al \Big)  d \al\\
&\leq 1 +  \int_{1}^\infty \PP \Big(  -p R_{\dl, N}(X_\dl ) > \log \al \Big)  d \al, 
\end{align*}

\noi

\noi
we see that the desired bound \eqref{GN0a} follows once we show
that there exist $C_2 = C_2(k, \dl, p) > 0$ and $ \b > 0$
such that
\begin{align}
\PP \Big(-p R_{\dl,N}(X_\dl) > \log \al \Big) 
\leq 
C_2  \al^{-(1+\b)}
\label{N11}
\end{align}

\noi
for any   $\al > 1$ and $N \in \N$.
We prove \eqref{N11} via a standard application of 
the so-called Nelson's estimate.
Namely, given $\al > 1$, 
we choose $N_0 = N_0(\al) > 0$
and establish \eqref{N11}
for $N \ge N_0$ and $N < N_0$
in two different ways.

 Given $\ld := \log \al > 0$, 
 we choose $N_0 >0$
 by setting 
\begin{align}
\ld =  2 p A_{k, \dl}    (\log (N_0+1))^{\frac{k+1}{2}}.
\label{N1a}
\end{align}

\noi
Then,
from \eqref{N2} and \eqref{N1a}, we have
\begin{align}
  -p R_{\dl,N_0}(X_\dl)
 \leq p A_{k, \dl}    (\log (N_0+1))^{\frac{k+1}{2}} = \tfrac 12 \ld.
\label{N3}
\end{align}

\noi
Hence, from \eqref{N3} and \eqref{N1}, 
we have 
\begin{align}
\begin{aligned}
\PP \Big( & -p R_{\dl,N}(X_\dl) > \ld \Big) 
 \leq 
\PP \Big( -p \big( R_{\dl,N}(X_\dl) -  R_{\dl, N_0}( X_\dl )\big) > 
\tfrac 12 \ld  \Big)\\
& \leq 
\PP \Big( p \bv R_{\dl,N}(X_\dl) -  R_{\dl, N_0}( X_\dl )\bv >    \tfrac{1}{2}\ld \Big)
\\
&
 \leq C_1   e^{-c'_{k,\dl}\, p^{-\frac2{k+1}} N_0^{\frac{2\ta }{k+1}}   \ld^{\frac2{k+1}}} \\
&
 \leq C_1   e^{-c'_{k,\dl}\, p^{-\frac2{k+1}}   \ld^{\frac2{k+1}}
(  e^{B_{k, \dl, p} \ld^\frac 2{k+1}} -1)}\\
& \le C_2 (k, \dl, p)  e^{- (1+ \b) \ld } 
\end{aligned}
\label{N4}
\end{align}
	
\noi
for any $N \geq N_0$.
On the other hand, for $N < N_0$, it follows from \eqref{N2}
and \eqref{N1a} that 
\begin{align*}
  -p R_{\dl,N}(X_\dl)
 \leq p A_{k, \dl}    (\log (N+1))^{\frac{k+1}{2}} < \tfrac 12 \ld
\end{align*}
 
\noi
and thus we have
\begin{align}
\PP \Big( & -p R_{\dl,N}(X_\dl) > \ld \Big) 
= 0.
\label{N5}
\end{align}

\noi
Putting \eqref{N4} and \eqref{N5} together, we conclude that 
\eqref{N11} holds for any $\al > 1$ and $N \in \N$.
Therefore, we obtain
\begin{align}
\|G_{\dl,N}(X_\dl)\|^{p}_{L^{p}(\O)}
\le C_3(k, \dl, p) < \infty
\label{N6}
\end{align}

\noi
for any $N \in \N$.

For $2 \le \dl \le \infty$, 
it follows from Remark \ref{REM:GN1}
that the constant 
$C_3(k, \dl, p)$ in \eqref{N6}
can be chosen to be  independent of $2 \le \dl \le \infty$, 
thus yielding \eqref{GN0b}.

\smallskip

\noi
$\bul$ \textbf{Step 2:} Next, we  show the (uniform) $L^p$-convergence
of the truncated densities.

Fix $0 < \dl \le \infty$.
The $L^p$-convergence \eqref{GN0c} of the truncated density $G_{\dl, N}(X_\dl)$
follows from the  uniform bound \eqref{GN0a} and a standard argument (see \cite[Remark 3.8]{TZ2}). More precisely, as a consequence of Proposition \ref{PROP:FN}\,(i)
and the continuous mapping theorem, 
we see that $ G_{\dl, N}(X_\dl) = e^{-R_{\dl,  N}(X_\dl)}$ converges in probability 
to the limit $ G_{\dl}(X_\dl) = e^{-R_{\dl}(X_\dl)}$.
Then, the $L^p$-convergence \eqref{GN0c}
 follows from the  uniform bound \eqref{GN0a} and this softer convergence in probability.
While we omit details of the argument in this case, 
we present details of an analogous argument in 
establishing the uniform $L^p$-convergence \eqref{GN0d} in the following.

 \smallskip
 
In the following, we present the proof of \eqref{GN0d}
and thus restrict our attention to $2 \le \dl \le \infty$. 
 Proposition \ref{PROP:FN}\,(i), 
 the continuity of the exponential function, 
and  the uniform continuous mapping theorem (Lemma \ref{LEM:KFC}),\footnote{Here, we use the tightness of $\{R_\dl(X_\dl)\}_{2\le \dl \le \infty}$, 
coming from \eqref{FN3a}, 
to verify the hypothesis \eqref{Z_tight}
in Lemma \ref{LEM:KFC}.} 
we see that 
$ G_{\dl, N}(X_\dl)$
 converges in probability to 
$G_{\dl}(X_\dl)$ as $N \to \infty$, uniformly in $2\le \dl \le \infty$.
Then, by setting 
%Given $\eps > 0$,  we define the event  
\begin{equation}
A_{\dl,N,\eps}=\big\{  |G_{\dl,N}(X_\dl)-G_\dl(X_\dl)| \leq\eps  \big\}, 
\label{N7a}
\end{equation}

\noi
we have 

\noi
\begin{align}
\sup_{2 \le\dl  \le \infty} \PP ({A^{c}_{\dl, N,\eps}})\too 0 ,
\label{N7}
\end{align}

\noi
as $N\to\infty$.  
Then, from \eqref{N7a},
Cauchy-Schwarz's inequality,  
the uniform (in $\dl$ and $N$, including $N = \infty$) bound~\eqref{GN0b}, and~\eqref{N7}, 
we obtain
\noi
\begin{align*}
\sup_{2 \le\dl  \le \infty} & \| G_\dl(X_\dl) - G_{\dl, N}(X_\dl) \|_{L^{p}(\O)}\\
&\leq \sup_{2 \le\dl  \le \infty}
 \| (G_\dl(X_\dl) -G_{\dl, N}(X_\dl) )\cdot  \ind_{A_{\dl,N,\eps}}\|_{L^{p}(\O)} \\
&\quad    
 + \sup_{2 \le\dl  \le \infty}
   \| (G_\dl(X_\dl) -G_{\dl, N}(X_\dl) )\cdot  \ind_{{A^{c}_{\dl,N,\eps}}}\,\|_{L^{p}(\O)}\\
&\leq  \eps +   \sup_{2 \le\dl  \le \infty}\| G_\dl(X_\dl) -G_{\dl, N}(X_\dl) \|_{L^{2p}(\O)}
\cdot 
\sup_{2 \le\dl  \le \infty}
 \PP  \big( A^c_{\dl,N,\eps} \big)^{\frac1{2p}} \\
&\leq 2\eps
\end{align*}

\noi
for    sufficiently large $N\gg 1$.
This proves 
the uniform (in $\dl$) $L^p$-convergence \eqref{GN0d}.
This concludes the proof of Proposition \ref{PROP:Gibbs1}.
\end{proof}

\subsection{Convergence of the Gibbs measures in the deep-water limit}
\label{SUBSEC:conv1}

In this subsection, we present the proof of Theorem \ref{THM:Gibbs1}\,(ii).
Once again, we restrict our attention to the defocusing case: $k \in 2\N + 1$.
The construction of the Gibbs measures in the previous subsection
shows that, for each
 $0 < \dl \le \infty$,  the Gibbs measure $\rho_\dl$ and the base
Gaussian measure $\mu_\dl$ are equivalent.
On the other hand, 
from Proposition \ref{PROP:equiv}, 
we know that the Gaussian measures $\mu_\dl$
are all equivalent for $0 < \dl \le \infty$.
Therefore,
we conclude that the Gibbs measure $\rho_\dl$, $0 < \dl < \infty$, 
for the defocusing gILW equation \eqref{gILW1}
and the Gibbs measure $\rho_\BO = \rho_\infty$ 
for the defocusing gBO equation \eqref{BO} are equivalent.
This proves the first claim in Theorem \ref{THM:Gibbs1}\,(ii).
Hence, it remains to show that
the Gibbs measure $\rho_\dl$ converges 
to $\rho_\BO$ in total variation, as $\dl \to \infty$.
%See Proposition~\ref{PROP:C2} below.

Before proceeding to the proof of convergence in total variation
of $\rho_\dl$ to $\rho_\BO$, 
let us first present the following 
$L^p$-convergence of the (truncated) densities.
For $0 < \dl \le \infty$, 
 let  $X_\dl$ and $X_\BO = X_\infty$
be as in~\eqref{GG3} and \eqref{Gbo2}, respectively, 
and
let $R_\dl(X_\dl)$
(and $G_\dl(X_\dl)$, respectively)
be the limit of 
$R_{\dl, N}(X_\dl)$ 
constructed in Proposition \ref{PROP:FN}
(and of $G_{\dl, N}(X_\dl)$
constructed in Proposition~\ref{PROP:Gibbs1}, respectively).

\begin{lemma}\label{LEM:C1}
Let $k \in 2\N + 1$ and $1 \le p < \infty$. Then, given $N \in \N$, 
we have
\begin{align}
\lim_{\dl \to \infty} \| G_{\dl,N}(X_\dl) - G_{\infty, N}(X_\BO)\|_{L^p(\O)} = 0.
\label{CC4}
\end{align}

\noi
As a corollary, we  have
\begin{align}
   \lim_{\dl\to\infty}
                 \| G_\dl(X_\dl) - G_\infty(X_\BO)\|_{L^p(\O)} =0 .
\label{CC4a}
\end{align}

\noi
In particular, the partition function 
$Z_{\dl}$ of the Gibbs measure $\rho_\dl$ in \eqref{THM1b}
converges to the partition function $Z_\BO = Z_\infty$
of the Gibbs measure $\rho_\BO = \rho_\infty$, as $\dl\to\infty$.
\end{lemma}

\begin{remark}\rm

In view of the argument presented in 
\eqref{GN0}, one may be tempted to 
conclude 
directly from \eqref{CC4a} in 
Lemma \ref{LEM:C1} 
that $\rho_\dl$ converges 
to $\rho_\BO = \rho_\infty$ in total variation
as $\dl \to \infty$.
However, this is not possible.
This is due to the fact that the base Gaussian measures
$\mu_\dl$ and $\mu_\infty$ are different.
If we were to mimic the argument in \eqref{GN0}, 
the integral in the third step of \eqref{GN0}
would be replaced by 
\begin{align}
\begin{split}
  \int_{\O }&  \Big(\ind_{\{X_\dl \in A\}} 
 G_{\dl, N}(X_\dl) - 
  \ind_{\{X_\BO \in A\}}  G_\infty(X_\BO) \Big)d \PP \\
& =     \int_{\O } \ind_{\{X_\dl \in A\}} 
\Big( G_{\dl, N}(X_\dl) 
 -  G_\infty(X_\BO) \Big)d \PP \\
& \quad +   \int_{\O } \Big(\ind_{\{X_\dl \in A\}} 
 -  \ind_{\{X_\BO \in A\}} \Big) G_\infty(X_\BO) d \PP .
 \end{split}
\label{CC4b}
\end{align}

\noi
While we can apply \eqref{CC4a} in Lemma \ref{LEM:C1}
to control the first term on the right-hand side of~\eqref{CC4b}, 
we can not handle the second term as it is.
Note that the difference
$\ind_{\{X_\dl \in A\}} 
 -  \ind_{\{X_\BO \in A\}}$
with respect to the $\PP$-integration
(and taking the supremum in $A \in \mathcal{B}_{H^{-\eps}}$)
 is closely related to 
the convergence in total variation of $\mu_\dl$ to $\mu_\infty$
proven in Proposition \ref{PROP:equiv}\,(iii),
which plays a crucial role in the proof of 
convergence in total variation of $\rho_\dl$ to $\rho_\BO$
presented below.

\end{remark}

\begin{proof}[Proof of Lemma \ref{LEM:C1}]

Fix $N \in \N$.
From \eqref{Wick1a}
and 
Lemma \ref{LEM:p2}, we have
\begin{align}
\s_{\dl,N}
=\frac 1{2\pi} \sum_{0< |n|\leq N}\frac{1}{K_\dl(n)}
\too
 \frac1{2\pi}  \sum_{|n|\leq N}\frac{1}{|n|} =: \s_{\infty, N},
\label{CC1}
\end{align}

 \noi
as $\dl \to \infty$.
It also follows from the definitions \eqref{Gbo2}, \eqref{GG3}, and Lemma \ref{LEM:p2}
 that, for any $x \in \T$ and $\o \in \O$,\footnote{Here, 
we use the convention that $g_n(\o) \in \C$ for every $\o \in \O$ and $n \in \Z^*$.}
%namely not up to a set of full probability.}  
$X_{\dl, N}(x)$ converges to $X_{\BO, N}(x)$ as $\dl\to\infty$.
Moreover, 
from \eqref{GG3} and Lemma \ref{LEM:p2}, 
we have
%\eqref{Gb1}, there exists a set $\Si \subset \O$ with $\PP(\Si) = 1$ such that 
\begin{align*}
| X_{\dl, N}(x; \o)| \les \sum_{0< |n| \leq N} \frac{|g_n(\o)|}{ K_\dl^\frac{1}{2}(n)}
\le C_{N, \o} < \infty
\end{align*}

\noi
for any $2 \le \dl \le \infty$, $x \in \T$,  and  $\o \in \O$.
Then, by the 
dominated convergence theorem applied to the integration in $x \in \T$,
 we have 

\noi
\begin{align}
R_{\dl, N}(X_\dl(\o)) = \frac 1{k+1} \int_{\T} H_k( X_{\dl, N}(x; \o); \s_{\dl, N} )dx 
 \too R_{\infty, N}(X_{\rm BO}(\o))
\label{CC3}
\end{align}

\noi
as $\dl \to \infty$, for any $\o \in \O$.
As a consequence, we see that 
   $G_{\dl, N}(X_\dl(\o)) $ converges  to 
 $G_{\infty, N}(X_{\rm BO}(\o))$
  as $\dl\to\infty$, 
 for any $\o \in \O$.
Moreover, from the uniform  (in $\o$) bound~\eqref{N2}, 
we conclude that 
$ G_{\dl,N}(X_\dl)$ converges to $G_{\infty, N}(X_\BO)$ in $L^p(\O)$,
as $\dl\to\infty$.
This proves \eqref{CC4}.

By the triangle inequality, we have
\begin{align}
\begin{aligned}
\|  & G_\dl  (X_\dl)  - G_\infty(X_{\rm BO})\|_{L^p(\O)}\\
&\leq 
\| G_\dl(X_\dl)  -G_{\dl, N }(X_\dl)\|_{L^p(\O)} 
 +
\| G_{\dl, N}(X_\dl)  -G_{\infty, N }(X_\BO)\|_{L^p(\O)} \\
& \quad + \|G_{\infty, N }(X_\BO) - G_\infty(X_{\rm BO}) \|_{L^p(\O)} .
\end{aligned}
\label{CC0}
\end{align}

\noi
%In  Proposition \ref{PROP:Gibbs1}, we already treated the first and third terms
%on the right-hand side of~\eqref{CC0}, 
%and thus it suffices to treat the second term on the right-hand side of \eqref{CC0}.
Then,  by first  applying \eqref{CC4} above
and then \eqref{GN0d} in Proposition \ref{PROP:Gibbs1}
 to \eqref{CC0}
 (namely, we first take 
 $\dl\to\infty$
 and then  $N\to\infty$),  we obtain
\begin{align*}
& \lim_{\dl \to \infty} \|   G_\dl  (X_\dl)  - G_\infty(X_{\rm BO})\|_{L^p(\O)}\\
& \quad \leq 
2 \lim_{N\to\infty}\Big( \sup_{2 \le \dl \le \infty} \| G_\dl(X_\dl)  -G_{\dl, N }(X_\dl)\|_{L^p(\O)} \\
& \hphantom{XXXXXXX}  +
\lim_{\dl \to \infty}
\| G_{\dl, N}(X_\dl)  -G_{\infty, N }(X_\BO)\|_{L^p(\O)} \Big)\\
& \quad = 0.
\end{align*}

\noi
This proves \eqref{CC4a}.
%This concludes the proof of Lemma \ref{LEM:C1}.
\end{proof}

We are now ready to show that the Gibbs measure $\rho_\dl$ in \eqref{THM1b}
converges to $\rho_\BO = \rho_\infty$ in total variation as $\dl \to \infty$.
By  the triangle inequality, we have
\begin{align}
\dtv(\rho_\dl, \rho_\BO)
\le \dtv(\rho_\dl, \rho_{\dl, N})
+
\dtv(\rho_{\dl, N}, \rho_{\infty, N})
+ 
\dtv(\rho_{\infty, N}, \rho_\BO)
\label{CX2}
\end{align}

\noi
for any $N \in \N$.
From Theorem \ref{THM:Gibbs1}\,(i)
(see also Proposition \ref{PROP:Gibbs1}), 
we have
\begin{align}
\lim_{N \to \infty} \sup_{2 \le \dl \le \infty} \dtv(\rho_{\dl, N}, \rho_\dl)  = 0.
\label{CX1}
\end{align}

\noi
Hence, it suffices to prove
\begin{align}	
\lim_{\dl \to \infty} \dtv(\rho_{\dl, N}, \rho_{\infty, N}) = 0
\label{CX3}
\end{align}

\noi
for any $N \in \N$.
Indeed, by applying \eqref{CX1} and \eqref{CX3} to \eqref{CX2}
(namely, by first taking $\dl \to \infty$ and then $N \to \infty$), we obtain
\begin{align*}
\lim_{\dl \to \infty} \dtv(\rho_\dl, \rho_\BO)
& \le \lim_{N \to \infty} \Big(\sup_{2 \le \dl \le \infty} \dtv(\rho_{\dl,N},  \rho_{\dl})
+
\lim_{\dl \to \infty} \dtv(\rho_{\dl, N}, \rho_{\infty, N})\Big)\\
& = 0.
\end{align*}

In the following, we prove \eqref{CX3} for any fixed $N \in \N$.
Fix $N \in \N$.
Then, Lemma \ref{LEM:C1}  with $p = 1$
 implies that the partition function 
 $Z_{\dl, N} = \|G_{\dl, N}(u)\|_{L^1(d\mu_\dl)}$
 of the truncated Gibbs measure $\rho_{\dl, N}$ in \eqref{GG6}
 converges to 
 the partition function 
 $Z_{\infty, N} = \|G_{\infty, N}(u)\|_{L^1(d\mu_\infty)}$
of the truncated Gibbs measure $\rho_{\infty, N}$ 
for the gBO equation as $\dl \to \infty$.
Then, from Proposition~\ref{PROP:equiv}\,(ii), 
we have
\begin{align}
\begin{split}
& \lim_{\dl \to \infty}
\sup_{A \in \mathcal{B}_{H^{-\eps}}}
| \rho_{\dl, N}(A) - \rho_{\infty, N}(A) |\\
& =  \lim_{\dl \to \infty}
\sup_{A \in \mathcal{B}_{H^{-\eps}}}
\bigg| \frac{Z_{\dl, N}}{Z_{\infty, N}} \rho_{\dl, N}(A) - \rho_{\infty, N}(A) \bigg|\\
 &\leq Z_{\infty, N} ^{-1}
 \lim_{\dl \to \infty}
\sup_{A \in \mathcal{B}_{H^{-\eps}}}
\bigg|  \int_{H^{-\e} } \ind_A(u)
\Big( G_{\dl, N}(u)  \frac {d\mu_{\dl}}{d \mu_\infty} (u) -  G_{\infty, N} (u) \Big)d \mu_\infty(u)\bigg| \\
 &\leq Z_{\infty, N} ^{-1}
 \lim_{\dl \to \infty}
 \int_{H^{-\e} } 
| G_{\dl, N}(u)  -  G_{\infty, N} (u) |\, d \mu_\infty(u) \\
 & \quad  +  Z_{\infty, N} ^{-1}
 \lim_{\dl \to \infty}
  \int_{H^{-\e}} 
G_{\dl, N}(u)\Big|  \frac {d\mu_{\dl}}{d \mu_\infty} (u) - 1\Big|\, d \mu_\infty(u).
%& \quad \le Z_\dl^{-1}\lim_{N\to\infty} \| G_{\dl, N}(X_\dl) - G_{\infty, N}(X_\) \|_{L^1(\O)}\\
%&=: \I + \II.
\end{split}
\label{CX4}
\end{align}

From \eqref{WO2} and 
\eqref{CC1}, we see that $G_{\dl, N}(u)$
converges to $G_{\infty, N}(u)$ $\mu_\infty$-almost surely, as $\dl \to \infty$.
Moreover, it follows from \eqref{N2}
and Remark \ref{REM:GN1}
that 
\begin{align}
 | G_{\dl, N}(u)  -  G_{\infty, N} (u) | \le C_{N} < \infty
 \label{CX4a}
\end{align}

\noi
for any $2 \le \dl \le \infty$.
Hence, by the dominated convergence theorem, we obtain
\begin{align}
 \lim_{\dl \to \infty}
 \int_{H^{-\e} } 
| G_{\dl, N}(u)  -  G_{\infty, N} (u) |\, d \mu_\infty(u) 
= 0.
\label{CX5}
\end{align}

By Scheff\'e's theorem (Lemma 2.1 in \cite{Tsy};
see also Proposition 1.2.7 in \cite{KS}), we have
\begin{align}
\dtv(\mu_\dl, \mu_\infty) = \frac 12 
  \int_{H^{-\e}} 
\Big|  \frac {d\mu_{\dl}}{d \mu_\infty} (u) - 1\Big|\, d \mu_\infty(u).
\label{CX6}
\end{align}

\noi
Then, it follows from 
the convergence in total variation of $\mu_\dl$ to $\mu_\infty$
as $\dl \to \infty$ (Proposition~\ref{PROP:equiv}\,(iii)), \eqref{CX6}, 
and the uniform (in $\dl$) bound \eqref{N2} (and Remark \ref{REM:GN1}) for $2 \le \dl \le \infty$
that
\begin{align}
\begin{split}
&  \lim_{\dl \to \infty}
  \int_{H^{-\e}} 
G_{\dl, N}(u)\Big|  \frac {d\mu_{\dl}}{d \mu_\infty} (u) - 1\Big|\, d \mu_\infty(u)\\
& \quad \le C_N 
 \lim_{\dl \to \infty}
  \int_{H^{-\e}} 
\Big|  \frac {d\mu_{\dl}}{d \mu_\infty} (u) - 1\Big|\, d \mu_\infty(u)\\
& \quad
= 2C_N \lim_{\dl \to \infty}\dtv(\mu_\dl, \mu_\infty) \\
& \quad = 0.
\end{split}
\label{CX7}
\end{align}

\noi
Therefore, from \eqref{CX4}, \eqref{CX5}, and \eqref{CX7}, 
we conclude \eqref{CX3} and hence 
convergence in total variation of $\rho_\dl$ to $\rho_\BO$ as $\dl \to \infty$.
This concludes the proof of Theorem \ref{THM:Gibbs1}
when $k \in 2\N + 1$.

\subsection{Gibbs measures for the ILW equation: variational approach}
\label{SUBSEC:ILW1}

We conclude this section 
by presenting the proof of Theorem \ref{THM:Gibbs1}
for the $k = 2$ case, corresponding to the 
ILW equation \eqref{ILW1}.
In this case, the problem is no longer defocusing
and thus 
we need to consider the (truncated) Gibbs measures 
with a Wick-ordered $L^2$-cutoff
of the form \eqref{THM1c}
and \eqref{THM1d}.
As pointed out in Remark~\ref{REM:k2}, 
there is no need for a renormalization
on the potential energy under the current (spatial) mean-zero condition.

Fix $K > 0$ in the remaining part of this section.
Given $0 < \dl \le \infty$ and $N \in \N$, 
define
the truncated density $G_{\dl, N}^K(u)$ by 
\begin{align}
\begin{split}
G_{\dl, N}^K(u)
& =  \chi_K\bigg( \int_\T \W(u_N^2) dx \bigg) e^{-\frac 1{3} \int_\T u_N^{3} dx} \\
& =  \chi_K\bigg(\int_\T H_2(u_N; \s_{\dl, N}) dx \bigg) e^{-\frac 1{3} \int_\T u_N^{3} dx} , 
\end{split}
\label{V1}
\end{align}

\noi
where $u_N = \P_N u$
and 
$\chi_K: \R \to  [0,1]$
is  a continuous function %with compact support such that 
such that $\chi_K(x) = 1$ for $|x| \le K$
and $\chi_K(x) = 0$ for $|x| \ge 2K$.

In view of the discussion 
in Subsections \ref{SUBSEC:MC1}
and \ref{SUBSEC:conv1}, 
Theorem \ref{THM:Gibbs1} for $k = 2$ follows
once we prove the following uniform bounds.

\begin{proposition}\label{PROP:V1}
Fix finite $p \ge 1$ and $K > 0$.
Then, given any $0 < \dl \le \infty$, we have
\begin{align*}
&\sup_{N\in \N} \| G_{\dl, N}^K(X_\dl)  \|_{L^p(\O)} 
=
\sup_{N \in \N}\| G_{\dl, N}^K(u) \|_{L^p(d\mu_\dl)}
 \le C_{p, \dl, K} < \infty.
\end{align*}

\noi
In addition, the following uniform bound holds for $2 \le \dl \le \infty$\textup{:}
\begin{align*}
 \sup_{N \in \N} \sup_{2 \le \dl \le \infty}
\| G_{\dl, N}^K(X_\dl)  \|_{L^p(\O)}
& = 
\sup_{N \in \N} \sup_{2 \le \dl \le \infty}
\| G_{\dl, N}^K(u)  \|_{L^p(d\mu_\dl)}\\
& 
 \le C_{p, K} < \infty.
\end{align*}

\end{proposition}

Let us first discuss how to conclude Theorem \ref{THM:Gibbs1}
when $k = 2$, by assuming 
Proposition~\ref{PROP:V1}.
Define the limiting density $G^K_\dl(u)$ by 
\begin{align}
G_{\dl}^K(u)
=  \chi_K\bigg( \int_\T \W(u^2) dx \bigg)e^{-\frac 1{3} \int_\T u^{3} dx}. 
\label{V4}
\end{align}

\noi
Note that Proposition \ref{PROP:FN}\,(i)
guarantees that 
$\int_\T \W(u^2) dx $
and $\int_\T u^{3} dx = \int_\T \W(u^{3}) dx$ in~\eqref{V4}
exist as the limits in $L^p(\mu_\dl)$ of the truncated versions.
Hence, the truncated density
$
G_{\dl, N}^K(u)$ converges in measure to 
the limiting density $G^K_\dl(u)$ in \eqref{V4}.
Hence, once we prove
Proposition \ref{PROP:V1}, 
we can repeat the argument in Step 2 in the proof of 
Proposition \ref{PROP:Gibbs1}
to show the following convergence results.

\begin{corollary}
Let $0 < \dl \le \infty$ and $1 \le p < \infty$.
Then, 
$G_{\dl, N}^K(X_\dl)$ converges to $G_{\dl}^K(X_\dl)$
in $L^p(\O)$ as $N \to \infty$.
Namely, 
we have
\begin{align*} 
\lim_{N\to\infty} \| G_{\dl, N}^K(X_\dl) - G_{\dl}^K(X_\dl) \|_{L^p(\O)}  =0.
\end{align*}

\noi
Furthermore, 
the convergence is uniform in $2 \le \dl \le \infty$\textup{:}
\begin{align*}
\lim_{N\to\infty}\sup_{2 \le \dl \le \infty} \| G_{\dl, N}^K(X_\dl) - G_{\dl}^K(X_\dl) \|_{L^p(\O)} =0.
\end{align*}

\end{corollary}

This proves an analogue of Theorem \ref{THM:Gibbs1}\,(i) when $ k = 2$.
The equivalence
of the Gibbs measure $\rho_\dl$ in \eqref{THM1d}, $0 <  \dl < \infty$, 
and $\rho_\BO$ in \eqref{Gx3}
follows from (i) the equivalence
of the Gibbs measure $\rho_\dl$ 
and the Gaussian measure with the Wick-ordered $L^2$-cutoff:
\begin{align*}
 \chi_K\bigg( \int_\T \W(u^2) dx \bigg)d \mu_\dl(u)
\end{align*}

\noi
(including $\dl = \infty$ with the understanding that $\rho_\infty = \rho_\BO$),
and (ii) the equivalence of the base Gaussian measure 
$\mu_\dl$, $0 <  \dl \le \infty$ (Proposition \ref{PROP:equiv}\,(ii)).

Finally, we discuss  convergence of the Gibbs measure $\rho_\dl$
to $\rho_\BO$
in the deep-water limit ($\dl \to \infty$).
In the defocusing case discussed in the previous subsection, 
the bound \eqref{N2} provided 
the uniform (in $2\le \dl \le \infty$ and $\o \in \O)$
bound on the truncated density $G_{\dl,  N}(u)$;
see the discussion 
 after \eqref{CC3}.
 See also 
\eqref{CX4a} and 
\eqref{CX7}.
In the current non-defocusing case, however, 
the bound  \eqref{N2} is not available to us. %does not provide any control on the truncated density.
Nonetheless, 
in view of \eqref{WO1}
and  \eqref{Wick1a} with Lemma \ref{LEM:p2}, 
the Wick-ordered $L^2$-cutoff in \eqref{V1} with \eqref{WO1}
implies 
\begin{align}
\bigg|\int_\T u_N^2 dx \bigg| \leq \s_{\dl, N} + 2K 
\le C_{N, K} < \infty, 
\label{V5}
\end{align}

\noi
for any $2 \le \dl \le \infty$ and $N \in \N$, 
where $C_{N, K}$ is independent of $ 2\le \dl \le \infty$.
Then, by Sobolev's inequality with \eqref{V5}, we have
\begin{align}
\bigg|\int_\T u_N^3 dx \bigg|
\les \| u_N \|_{H^{\frac 16}}^3
\le N^\frac 12 C_{N, K}^\frac 32, 
\label{V6}
\end{align}

\noi
which provides a bound on the truncated density
$G_{\dl,  N}^K(u)$ in \eqref{V1}, uniformly in $2 \le \dl \le \infty$.
With this bound on 
$G_{\dl,  N}^K(u)$, 
we can  repeat the argument presented in Subsection \ref{SUBSEC:conv1}
to conclude the desired convergence in total variation of $\rho_\dl$ to $\rho_\BO$ as $\dl \to \infty$.

In the remaining part of this section, we present the proof of 
 Proposition \ref{PROP:V1}.
Given $0 < \dl \le \infty$ and $N \in \N$, set 
\begin{align}
\begin{split}
\RR_{\dl, N} (u)
&=   \frac13  \int_{\T }  u_N^3    dx
+ A \, \bigg| \int_{\T } \W(  u_N^2 ) dx\bigg|^2 , 
\end{split}
\label{V7}
\end{align}

\noi
where 
$\W(  u_N^2 ) = \W_{\dl, N}(  u_N^2 ) = H_2(u_N; \s_{\dl, N})$.
Then, 
as in \cite{OST}, 
we consider the following truncated density:
\begin{align}
\begin{split}
\GG_{\dl, N}^K(u)
& = e^{-\RR_{\dl, N}(u)}= 
 e^{-\frac 1{3} \int_\T u_N^3  dx
- A| \int_\T \W(u_N^2) dx |^2}
\end{split}
\label{V8}
\end{align}

\noi
for some suitable $A > 0$.
Noting that
\begin{align}
\chi_K(x) \le \exp\big( -  A |x|^\gamma\big) \exp(A 2^\g K^\g)
\label{H6}
\end{align}

\noi
for any $K, A , \g > 0$, 
we have
\begin{align*}
G_{\dl, N}^K(u) \le C_{A, K} \cdot \GG_{\dl, N}^K(u).
\end{align*}

\noi
Hence, 
Proposition \ref{PROP:V1} follows once we prove the following 
uniform bounds
on $\GG_{\dl, N}^K(u)$.

\begin{proposition}\label{PROP:V2}

Fix finite $p \ge 1$.
Then, there exists  $A_0  = A_0(p)>0 $ such that
\begin{align}
&\sup_{N\in \N} \| \GG_{\dl, N}^K(X_\dl)  \|_{L^p(\O)} 
=
\sup_{N \in \N}\| \GG_{\dl, N}^K(u) \|_{L^p(d\mu_\dl)}
 \le C_{p, \dl, K, A} < \infty
 \label{V10}
\end{align}

\noi
for any  $0 < \dl \le \infty$, 
  $K > 0$, 
 and $A \ge A_0$.
In addition, the following uniform bound holds for $2 \le \dl \le \infty$\textup{:}
\begin{align}
\begin{split}
 \sup_{N \in \N} \sup_{2 \le \dl \le \infty}
\| \GG_{\dl, N}^K(X_\dl)  \|_{L^p(\O)}
& = 
\sup_{N \in \N} \sup_{2 \le \dl \le \infty}
\| \GG_{\dl, N}^K(u)  \|_{L^p(d\mu_\dl)}\\
& 
 \le C_{p, K, A} < \infty
\end{split}
 \label{V11}
\end{align}

\noi
for any    $K > 0$  and $A \ge A_0$.

\end{proposition}

As mentioned in the introduction, 
we employ the variational approach,  introduced by Barashkov and Gubinelli \cite{BG}, 
to prove Proposition \ref{PROP:V2}.
In particular, 
we follow closely the argument in \cite{OST}, 
where the $\dl = \infty$ case was treated via the variational approach.
See also 
\cite{GOTW, OSRW2,  OOT1, bring, OOT2}
for recent works on dispersive PDEs, where the variational approach played a crucial role.

Let us first introduce some notations.
Let $W(t)$ be a cylindrical Brownian motion in 
\begin{align*}
L^2_0(\T) = \P_{\ne 0} L^2(\T)
\end{align*}

\noi
of mean-zero functions on $\T$, 
where $\P_{\ne 0}$ denotes the projection onto the non-zero frequencies.
Namely, we have
\begin{align}
W(t) = \frac 1{2\pi}\sum_{n \in \Z^*} B_n(t) e_n,
\label{P1}
\end{align}

\noi
where  
$\{B_n\}_{n \in \Z^*}$ is a sequence of mutually independent complex-valued\footnote{By convention, we normalize $B_n$ such that $\text{Var}(B_n(t)) = 2\pi t$.}
Brownian motions such that 
$\cj{B_n}= B_{-n}$, $n \in \Z^*$. 
Then, we define a centered Gaussian process $Y_\dl (t)$
by 
\begin{align}
 Y_\dl  (t)
=  ( \Gdl \dx)^{-\frac 12}    W(t), 
\label{P2}
\end{align}

\noi
where $({\Gdl \dx})^{-\frac 12}$ is the Fourier multiplier operator
with the multiplier $(K_\dl(n))^{-\frac 12}$
with $K_\dl (n)$ as in \eqref{GG4}.
In view of \eqref{GG3}, we have 
$\L(Y_\dl(1)) = \mu_\dl$.
Given $N \in \N$, 
we set   $Y_{\dl,N} = \P_NY_\dl $.
Then, from \eqref{Wick1a}, 
we have  
\[
\E [Y_{\dl, N}^2(1)] = \s_{\dl,N} \sim_\dl \log (N+1).
\]

Next, we recall the  Bou\'e-Dupuis variational formula. 
Let $\Ha$ denote the collection of drifts, 
which are progressively measurable processes 
belonging to 
$L^2([0,1]; L^2_0(\T))$, $\PP$-almost surely. 
We now state the  Bou\'e-Dupuis variational formula \cite{BD, Ust}.
See, in  particular, Theorem 7 in~\cite{Ust}.

\begin{lemma}\label{LEM:var3}
Given $0 < \dl \le \infty$, 
let $Y_\dl$ be as in \eqref{P2}.
Fix $N \in \N$.
Suppose that  $F:C^\infty(\T) \to \R$
is measurable such that $\E\big[|F(Y_{\dl, N}(1))|^p\big] < \infty$
and $\E\big[|e^{-F(Y_{\dl, N} (1))}|^q \big] < \infty$ for some $1 < p, q < \infty$ with $\frac 1p + \frac 1q = 1$.
Then, we have
\begin{align}
- \log \E\Big[e^{-F( Y_{\dl, N}(1))}\Big]
= \inf_{\dr \in \mathbb H_a}
\E\bigg[ F( Y_{\dl,N}(1) +  \P_N I_{\dl}(\dr)(1)) + \frac{1}{2} \int_0^1 \| \dr(t) \|_{L^2_x}^2 dt \bigg], 
\label{P3}
\end{align}

\noi
where  $I_{\dl}(\dr)$ is  defined by 
\begin{align*}
 I_{\dl}  (\dr)(t) = \int_0^t (\Gdl \dx)^{-\frac 12}     \dr(t') dt'
%\label{P3a}
\end{align*}

\noi
and the expectation $\E = \E_\PP$
is an 
expectation with respect to the underlying probability measure~$\PP$.

\end{lemma}

\begin{remark}\rm
(i) 
As far as the proof of Proposition \ref{PROP:V2} is concerned, 
we only need to work with $Y_{\dl, N}$ evaluated at time $t = 1$.
As such, we could have stated Lemma \ref{LEM:var3}
with $X_{\dl, N}$ in place of $Y_{\dl, N}(1)$, thus allowing
us to avoid introducing $W(t)$ in \eqref{P1} and $Y_\dl(t)$ in \eqref{P2}.
We, however, did not do so since the natural setting
of the Bou\'e-Dupuis formula is as stated above.
For example, \eqref{P3} allows us to choose a $Y_\dl$-dependent drift $\dr$,
which is crucial in showing non-normalizability of the focusing Gibbs measures
$\rho_\dl$.  See \cite{OST}.

\smallskip

\noi
(ii) In view of the discussion above, 
in order to prove Proposition \ref{PROP:V2}, 
it is possible to work with 
a slightly different and weaker variational formula
stated in  \cite[Proposition 4.4]{GOTW}, 
where an expectation is taken 
with respect to  a shifted measure.

\end{remark}

In the following, we 
prove Proposition \ref{PROP:V2}
by applying 
 Lemma \ref{LEM:var3}
to $\GG^K_{\dl, N}(u)$ in~\eqref{V8}.
Before proceeding to the proof of Proposition \ref{PROP:V2}, 
let us state a preliminary lemma
on the  pathwise regularity bounds  of 
$Y_{\dl, N}(1)$ and $I_{\dl}(\dr)(1)$.

\begin{lemma}  \label{LEM:Dr}
	
\textup{(i)} 
Let $\eps > 0$ and fix  finite $p \ge 1$.
Then, given any $0 < \dl \le \infty$, 
we have 
\begin{align}
\begin{split}
\E 
\Big[  \|Y_{\dl, N}(1) & \|_{W^{-\eps,\infty}}^p
 + \| \W(Y_{\dl, N}^2(1))  \|_{W^{-\eps,\infty}}^p\\
& + 
\big\| \W( Y_{\dl, N}^3(1)  )  \big\|_{W^{-\eps,\infty}}^p
\Big]
\leq C_{\eps, p, \dl} <\infty,
\end{split}
\label{P4}
\end{align}

\noi
uniformly in $N \in \N$.
Furthermore, 
by restricting our attention to  $2 \le \dl \le \infty$, 
we can choose the constant $C_{\eps, p, \dl}$ in \eqref{P4}
to be independent of $\dl$.

\smallskip
	
\noi
\textup{(ii)} 
Let  $0 < \dl \le \infty$.
For any $\dr \in \Ha$, we have
\begin{align}
\| I_{\dl} (\dr)(1) \|_{H^{\frac 12}}^2 \le C_\dl \int_0^1 \| \dr(t) \|_{L^2_x}^2dt, 
\label{CM}
\end{align}

\noi
where the constant $C_\dl > 0$ can be chosen
to be independent of $2 \le \dl \le \infty$.

\end{lemma}

\begin{proof}
By noting that $\L(Y_{\dl, N}(1)) = \L(X_{\dl, N})$, 
we see that 
Part (i) follows from 
Proposition~\ref{PROP:FN}.
As for the bound \eqref{CM}, 
it follows from  Minkowski's and Cauchy-Schwarz's inequalities and the lower bound 
\eqref{HH1a}
of $K_{\dl}(n)$  that 
\begin{align*}
\| I_{\dl} (\dr)(1) \|_{H^{\frac 12}}
&=\bigg\| \jb{\nb}^{\frac12} \int_0^1 (\Gdl \dx)^{-\frac 12}     \dr(t') dt'  \bigg\|_{L^{2}_0}\\
%&\les \bigg\|   \bigg(\frac{ |n| }{ |n-\frac12| }  \bigg)^{\frac12}  \int_0^t    \ft  \dr(t', n)    dt'  \bigg\|_{\l^{2}_n(\Z^*)}\\
& \le C_\dl    \int_0^1 \| \dr(t) \|_{L^2}dt 
 \leq   C_\dl \bigg( \int_0^1 \| \dr(t) \|_{L^2}^2dt   \bigg)^{\frac12}.
\end{align*}

\noi
When $ 2 \le \dl \le \infty$, 
the lower bound 
\eqref{Low_Kdl} allows us to choose the constant $C_\dl$
to be independent of $2 \le \dl \le \infty$.
\end{proof}

Fix  $0 < \dl \le \infty$ and finite $p \ge 1$.
We
 first prove the bound \eqref{V10}.
In view of the Bou\'e-Dupuis formula (Lemma \ref{LEM:var3}), 
it suffices to  establish a  lower bound on 
\begin{equation}
\M_{\dl,N}(\dr) = \E
\bigg[p \RR_{\dl, N}(Y_{\dl} (1) + I_{\dl }  (\dr)(1)) + \frac{1}{2} \int_0^1 \| \dr(t) \|_{L^2_x}^2 dt \bigg], 
\label{P5}
\end{equation}

\noi 
uniformly in $N \in \N$ and  $\dr \in \Ha$.
We  set 
\[
Y_{\dl, N} = \P_N Y_{\dl} = \P_N Y_{\dl}(1)
\qquad
\textrm{and}
\qquad
\Dr_{\dl,N} = \P_N  \Dr_{\dl} = \P_N I_{\dl} (\dr)(1).
\]

\noi
From \eqref{V7} and \eqref{Herm3}, we have 
\begin{align}
\begin{split}
\RR_{\dl, N} (Y_{\dl} + \Dr_{\dl})  & = 
\frac 13 \int_{\T}  \W ( Y_{\dl, N}^3 )  dx
+ \int_{\T}  \W ( Y_{\dl, N}^2 )  \Dr_{\dl, N} dx+ \int_{\T}   Y_{\dl, N}   \Dr_{\dl, N}^2 dx
\\
&\hphantom{X}
+ \frac 13 \int_{\T} \Dr_{\dl, N}^3 dx
+  A \bigg\{ \int_{\T} \Big( \W ( Y_{\dl, N}^2 ) + 2 Y_{\dl, N} \Dr_{\dl, N} + \Dr_{\dl, N}^2 \Big) dx \bigg\}^2, 
\end{split}
\label{P5a}
\end{align}

\noi
where the first term on the right-hand side vanishes under the expectation.
Hence, from~\eqref{P5} and \eqref{P5a}, we have
\begin{align}
\begin{split}
\M_{\dl,N}(\dr)
&=\E
\Bigg[\, 
p \int_{\T}  \W ( Y_{\dl, N}^2 )  \Dr_{\dl, N} dx
+ p \int_{\T}   Y_{\dl, N}   \Dr_{\dl, N}^2 dx + \frac p3 \int_{\T} \Dr_{\dl, N}^3 dx
\\
&\hphantom{XXX}
+  A p  \bigg\{ \int_{\T} \Big( \W ( Y_{\dl, N}^2 ) + 2 Y_{\dl, N} \Dr_{\dl, N} + \Dr_{\dl, N}^2 \Big) dx \bigg\}^2\\
&\hphantom{XXX}
+ \frac{1}{2} \int_0^1 \| \dr(t) \|_{L^2_x}^2 dt 
\Bigg].
\end{split}
\label{P5b}
\end{align}

We now recall the following lemma from \cite[Lemma 4.1]{OST}, 
where the $p = 1$ case was treated.
See also Lemma 5.8 in \cite{OOT1}.

\begin{lemma} \label{LEM:Dr2}
\textup{(i)}
There exist small $\eps>0$ and  a constant  $c  = c(p) >0$ 
and $C_0 > 0$ such that
\begin{align*}
 p \bigg| \int_{\T}  \W ( Y_{\dl, N}^2 ) \Dr_{\dl, N} dx \bigg|
&\le c \| \W ( Y_{\dl, N}^2 ) \|_{W^{-\eps,\infty}}^2  
+ \frac 1{100} 
\| \Dr_{\dl, N} \|_{H^\frac 12}^2, 
\\
p \bigg| \int_{\T}  Y_{\dl, N} \Dr_{\dl, N}^2  dx \bigg|
&\le c 
\| Y_{\dl, N} \|_{W^{-\eps,\infty}}^{6} + \frac 1{100} \Big(
\| \Dr_{\dl, N} \|_{H^\frac 12}^2 +  \| \Dr_{\dl, N} \|_{L^2}^4 \Big),
\\
\frac p3 \bigg| \int_{\T}  \Dr_{\dl, N}^3  dx \bigg|
&\le     \frac 1{100} \| \Dr_{\dl, N} \|_{H^\frac 12}^2
+ C_0 p^2 \| \Dr_{\dl, N} \|_{L^2}^{4 }, 
\end{align*}

\noi
uniformly in $N \in \N$ and $0 < \dl \le \infty$.

\smallskip

\noi
\textup{(ii)}	
Let $A> 0$. Given any small $\eps > 0$, 
there exists $c = c(\eps, p, A)>0$ such that
\begin{align}
\begin{split}
Ap\bigg\{ & \int_{\T}  \Big( \W ( Y_{\dl, N}^2 ) + 2 Y_{\dl, N} \Dr_{\dl, N} + \Dr_{\dl, N}^2 \Big) dx \bigg\}^2 \\
&\ge \frac {Ap}4 \| \Dr_{\dl, N} \|_{L^2}^4 - \frac 1{100} \| \Dr_{\dl, N} \|_{H^\frac 12}^2 
- c\bigg\{ \| Y_{\dl, N} \|_{W^{-\eps,\infty}}^c 
+ \bigg( \int_{\T} \W ( Y_{\dl, N}^2 )  dx \bigg)^2 \bigg\}, 
\end{split}
\label{P6}
\end{align}

\noi
uniformly in $N \in \N$
and $0 < \dl \le \infty$.

\end{lemma}

As in \cite{OST}, we establish a pathwise lower bound on $\M_{\dl,N}(\dr)$ in~\eqref{P5b}, 
uniformly in $N \in \N$ and $\dr \in \Ha$, 
by making use of the 
 positive terms:
\begin{equation}
\U_{\dl, N}(\dr) =
\E \bigg[\frac {Ap} 4\| \Dr_{\dl, N}\|_{L^2}^4 + \frac{1}{2} \int_0^1 \| \dr(t) \|_{L^2_x}^2 dt\bigg].
\label{P7}
\end{equation}

\noi
coming from \eqref{P5b} and \eqref{P6}.
From \eqref{P5b} and \eqref{P7} together with Lemmas  \ref{LEM:Dr2} and \ref{LEM:Dr}, we obtain
\begin{align}
\inf_{N \in \mathbb{N}} \inf_{\dr \in \Ha} \M_{\dl,N}(\dr) 
\geq 
\inf_{N \in \mathbb{N}} \inf_{\dr \in \Ha}
\Big\{ -C_{p, \dl, A} + \frac{1}{10}\U_{\dl, N}(\dr)\Big\}
\geq - C_{p, \dl, A} >-\infty, 
\label{P8}
\end{align}

\noi
provided that $A = A(p) \gg1$ is sufficiently large.
Hence, the uniform (in $N$) bound \eqref{V10}
follows from 
Lemma \ref{LEM:var3} with \eqref{V8} and \eqref{P8}.

Next, we restrict our attention to  $2 \le \dl \le \infty$.
In this case, 
the constant $C_{\eps, p, \dl}$ in \eqref{P4} of Lemma \ref{LEM:Dr}
is independent of $\dl$
and, as a result, we see that the constant $C_{p, \dl, A}$ in~\eqref{P8}
is also independent of $2 \le \dl \le \infty$.
Therefore, the second bound \eqref{V11} follows from 
Lemma~\ref{LEM:var3} with \eqref{V8} and \eqref{P8}.
This concludes the proof of Proposition \ref{PROP:V2}
and hence of Theorem~\ref{THM:Gibbs1} when $k = 2$.

\section{Gibbs measures in the shallow-water regime}
\label{SEC:4}

In this section, we present the proof of Theorem \ref{THM:Gibbs2}.
Namely, we go over the construction and convergence in the shallow-water limit ($\dl \to 0$)
of the Gibbs measure $\wt \rho_\dl$ associated
with the scaled gILW equation \eqref{gILW3}.
For each fixed $0 < \dl < \infty$, 
the scaling transformation \eqref{trans}
simply introduces a constant factor, depending on $\dl$.
Hence, the regularity properties
of the support of the base Gaussian measures
$\mu_\dl$ in \eqref{Gibbs2}
for the unscaled problem
and
$\wt \mu_\dl$ in \eqref{Gibbs4}
for the scaled problem 
are the same
for each fixed $0 < \dl < \infty$,
and thus
we can repeat the argument in Section \ref{SEC:Gibbs1}
to construct the Gibbs measure $\wt \rho_\dl$ supported on $H^{-\eps}(\T) \setminus L^2(\T)$, 
$\eps > 0$, 
yielding Theorem \ref{THM:Gibbs2}\,(i)
for {\it each fixed} $0 < \dl < \infty$.
The main difference in this shallow-water regime
appears in 
establishing uniform (in $\dl$) bounds
and convergence as $\dl \to 0$.
This is 
due to the singularity of the base Gaussian measures
$\wt \mu_\dl$, $0 < \dl < \infty$, supported on $H^{-\eps}(\T) \setminus L^2(\T)$, 
and $\wt \mu_0$ in \eqref{Gibbs6}
supported on $H^{\frac 12-\eps}(\T) \setminus H^{\frac 12}(\T)$;
see Proposition \ref{PROP:equiv2}.

In Subsection \ref{SUBSEC:41}, 
we first study the singularity and convergence properties
of the base Gaussian measures.
Then, we briefly go over the construction
and convergence of the Gibbs measure $\wt \rho_\dl$
for the defocusing case ($2 \N + 1$)
in  Subsections \ref{SUBSEC:42}
and 
\ref{SUBSEC:conv2}.
In Subsection \ref{SUBSEC:ILW2}, 
we discuss the variational approach to treat the $k = 2$ case.

\subsection{Singularity of the base Gaussian measures}
\label{SUBSEC:41}

Given $0 < \dl < \infty$, 
let 
$\wt \mu_\dl$ be as in~\eqref{Gibbs4}
and let 
$\wt \mu_0$ be as in \eqref{Gibbs6}. 
Then, a typical element
under $\wt \mu_\dl$
(and under $\wt \mu_0$, respectively)
is given by 
the Gaussian Fourier series
$\wt X_\dl$ in \eqref{GH3}
(and 
by $X_\KDV$ in \eqref{Gibbs7}, respectively).
Given $N \in \N$, 
set
\begin{align}
\wt X_{\dl, N} = \P_N \wt X_{\dl, N}
\qquad \text{and}\qquad  
X_{\KDV, N} = \P_N X_\KDV.
\label{KK1}
\end{align}

\noi
Then,  in view of \eqref{GH4a}, 
we see that, for each $0 < \dl < \infty$,   $\wt X_{\dl, N}$
converges in $L^p(\O)$ for any finite $p \ge 1$ and almost surely to 
the limit $\wt X_{\dl}$ in $H^{-\eps}(\T) \setminus L^2(\T)$, $\eps > 0$, 
as $N \to \infty$.
On the other hand, it is well known \cite{BO94, ORT} that $X_{\KDV, N}$
converges,  in $L^p(\O)$ and almost surely,  to 
the limit $X_\KDV$ in $H^{\frac 12-\eps}(\T) \setminus H^\frac 12(\T)$, $\eps > 0$, 
as $N \to \infty$.

\begin{proposition}
\label{PROP:equiv2}
	
{\rm (i)} 
Given any  $\eps >0$
and finite $p \ge 1$, 
   $\wt X_\dl$ converges to $X_\KDV$ in $L^p(\Omega ;H^{-\e}(\T))$
and almost surely in $H^{-\eps}(\T)$, 
as $\dl\to  0$.  In particular, the Gaussian measure~$\wt \mu_\dl$
converges weakly to the Gaussian measure $\mukdv$, as $\dl\to 0$.
	
\smallskip
	
\noi
{\rm (ii)} 
Let $\eps > 0$. Then,
for any $0 < \dl < \infty$,  the Gaussian measures $\wt \mu_\dl$ and $ \mukdv$ are singular
as probability measures $H^{-\e}(\T)$.
\end{proposition}

In Section \ref{SEC:Gibbs1}, 
the convergence in total variation of $\mu_\dl$ to $\mu_\infty$ played
an essential role in establishing 
the convergence in total variation of $\rho_\dl$ to $\rho_\BO$.
Proposition \ref{PROP:equiv2}
only provides weak convergence of the 
 base Gaussian measures $\wt \mu_\dl$
to  $\wt \mu_0$, 
and 
the singularity between the base Gaussian measures suggests that 
we do not expect any stronger mode of convergence
(such as convergence in total variation).
As a result, 
we only expect weak convergence of the associated Gibbs measures
$\wt \rho_\dl$ to $\rho_\KDV$ in \eqref{Gibbs5b} 
in  the shallow-water limit ($\dl \to 0$).

\begin{proof}[Proof of Proposition \ref{PROP:equiv2}]
Let $\eps > 0$.
From 
\eqref{Gibbs7}, 
\eqref{GH3}, and Lemma \ref{LEM:hyp}, we have
\begin{align}
\begin{split}
\|  \wt X_\dl - X_\KDV \|_{L^p_\o H_x^{-\eps}}  \
& \les_p \| \jb{\nb}^{-\eps} (\wt  X_\dl - X_\KDV)(x) \|_{L^2_x L^2_\o }\\
& \sim \Bigg( \sum_{n\in\Z^\ast} \frac 1{\jb{n}^{2\eps}}
 \bigg( \frac{1}{  L_\dl^ \frac{1}{2} (n)   } 
 - \frac{1}{|n| }  \bigg)^2  \Bigg)^{\frac 12}.
\end{split}
\label{KK2}
\end{align}

\noi
It follows from \eqref{Low_Ldl} in 
Lemma \ref{LEM:p1}
that 
  the summand is bounded by $\jb{n}^{-1 -2\eps}$
uniformly in $0 < \dl \les 1$, 
which is summable in $n \in \Z^*$.
Moreover, 
from Lemma \ref{LEM:p1}\,(ii), 
we see that, for each $n \in \Z^*$,  the summand tends to 0 
as $\dl \to 0$.
Hence, by the dominated convergence theorem, 
we conclude that $\wt X_\dl $
converges to $X_\KDV$ in $L^p(\O; H^{-\eps}(\T))$.
As for almost sure convergence,
we  repeat a computation analogous to \eqref{KK2}
 but with~\eqref{Gb1} in place of $\E[ |g_n|^2] \sim 1$.
We omit details.
See
\eqref{GX4a} for an analogous argument in the unscaled case.

 Next we prove part (ii) by using Lemma \ref{LEM:Kaku}.
From \eqref{GH3} and \eqref{Gibbs7}, we have
 \begin{align*}
\wt  X_\dl(\o) =     \sum_{n\in \N} \bigg( \frac{\Re g_n}{\pi L_\dl^\frac 12(n)} \cos (n x)
 - \frac{\Im g_n}{\pi L_\dl^\frac 12(n)}  \sin (nx)\bigg)
 \end{align*}

\noi
for $0 \le \dl < \infty$
with the understanding that $\wt X_0 = X_\KDV$ and $L_0 (n) = n^2$.
 For $n\in\Z^*$, set 
  \[
 A_{n} = \frac{\Re g_n}{\pi L_\dl^\frac 12(n)} 
 \qquad \text{and}\qquad 
 A_{-n} = - \frac{\Im g_n}{\pi L_\dl^\frac 12(n)} , 
 \]
 
\noi
and 
  \[
 B_{n} = \frac{\Re g_n}{\pi |n|} 
 \qquad \text{and}\qquad 
 B_{-n} =-  \frac{ \Im g_n}{\pi |n|} . 
 \]

 \noi
 with $a_{\pm n} = \E[ A_{\pm n}^2] = \frac{1}{\pi L_\dl(n)}$ and  $b_{\pm n} = \E[ B_{\pm n}^2] = \frac{1}{\pi n^2}$. 
Then, from 
Lemma \ref{LEM:p1}\,(iv),  we have
\begin{align*}
\sum_{n\in\Z^*} \Big( \frac{b_n}{a_n} -1 \Big)^2 
=  \sum_{n\in\Z^\ast} \frac{ ( n^2 - L_\dl(n) )^2}{n^4} 
=  \sum_{n\in\Z^\ast} h^2(n, \dl) = \infty
\end{align*}

\noi
for any $\dl > 0$, 
where $h(n, \dl)$ is as in \eqref{LL0a}.
Therefore, 
we conclude from  Kakutani's theorem (Lemma~\ref{LEM:Kaku})
that, for any $0 < \dl < \infty$, 
the Gaussian measures $\wt \mu_\dl$
 and $\wt \mu_0$ are mutually singular.
 \end{proof}

\subsection{Construction of the  Gibbs measures for the defocusing scaled gILW equation}
\label{SUBSEC:42}

In this subsection, we briefly go over the construction 
of the Gibbs measure $\wt \rho_\dl$, $0 < \dl < \infty$, 
for the scaled gILW equation \eqref{gILW3}
in the defocusing case $k \in 2\N+1$.
(Theorem \ref{THM:Gibbs2}\,(i)).
We treat the $k = 2$ case
in Subsection \ref{SUBSEC:ILW2}.

Fix the depth parameter $0 < \dl <  \infty$.
Given $N \in \N$, let $\wt X_{\dl, N} = \P_N \wt X_\dl$, 
where $\wt X_\dl$ is defined in \eqref{GH3}.
Given $k \in \N$, 
let  
\begin{align}
\W(  \wt X_{\dl, N}^k ) = H_k( \wt X_{\dl, N} ;  \wt \s_{\dl, N}   )
\label{HN0a}
\end{align}
denote the Wick power defined in~\eqref{GH5a},\footnote{As in Section \ref{SEC:Gibbs1},
we freely interchange
the representations in terms of $\wt X_\dl$
and in terms of $v$ distributed by $\wt \mu_\dl$,
when there is no confusion.
}
 where $\wt \s_{\dl, N}$ is as in \eqref{Wick11a}.
Then, the truncated Gibbs measure $\wt \rho_{\dl, N}$ in~\eqref{GH6}
can be written as 
\begin{align*}
\wt \rho_{\dl,  N} (A)
& = Z_{\dl, N}^{-1} \int_{H^{-\eps}} \ind_{\{v \in A\}}e^ {-\frac 1{k+1} \int_\T \W(v_N^{k+1})dx} 
d\wt \mu_\dl(v)\\
& = Z_{\dl, N}^{-1} \int_\O \ind_{\{\wt X_{\dl}(\o) \in A\}}e^ {-\frac 1{k+1} \int_\T \W(\wt X_{\dl, N}^{k+1}(\o))dx} 
d \PP(\o), 
\end{align*}

\noi
where $v_N = \P_N v$.
By repeating the proof of Proposition \ref{PROP:FN}
in the unscaled setting, 
we obtain the following result.

\begin{proposition}\label{PROP:HN}

Let $k \in \N$ and $0 < \dl <  \infty$.
Given $N \in \N$, let 
$\W( \wt  X_{\dl, N}^k )$ be as in~\eqref{HN0a}.
Then, given any finite $p \ge 1$, 
the sequence  $\{ \W(\wt X_{\dl, N}^k)\}_{N \in \N}$
is Cauchy in $L^p( \O; W^{s, \infty}(\T))$, $s < 0$, 
thus converging to a limit denoted by
$\W(\wt X_{\dl}^k)$.
This convergence of $\W(\wt X_{\dl, N}^k)$
to $\W(\wt X_{\dl}^k)$ also holds almost surely in $W^{s, \infty}(\T)$.
Furthermore, given any finite $p \ge 1$, we have 
\begin{align*}
\sup_{N \in \N} \sup_{0 < \dl\le1} \big\|  \|\W(  \wt X_{\dl, N}^k )\|_{W^{s, \infty}_x} \big\|_{L^p(\O)} 
< \infty
\end{align*}

\noi
and 
\begin{align*}
\sup_{0 < \dl \le1} \big\|  \| \W(  \wt X_{\dl, M}^k ) - \W( \wt  X_{\dl, N}^k )\|_{W^{s, \infty}_x} \big\|_{L^p(\O)} 
\too 0
\end{align*}

\noi
for any $M \ge N$, tending to $\infty$.
In particular, the rate of convergence is uniform in $0 < \dl \le 1$.

\smallskip

As a corollary, the following two statements hold.

\smallskip

\noi
\textup{(i)}
 Let
$0 < \dl <  \infty$.
Given $N \in \N$, 
let $\wt R_{\dl, N}(v) = \wt R_{\dl, N}(v; k+1)$ denote the truncated potential energy defined by 
\begin{align}
\wt  R_{\dl, N}(v):=
\frac 1{k+1}\int_\T \W((\P_N v)^{k+1}) dx 
=
\frac 1{k+1}\int_{\T}  H_{k+1}\big(  \P_N v ; \wt  \s_{\dl, N} \big) dx, 
\label{HN2}
\end{align}

\noi
where $\wt \sigma_{\dl, N}$ is  as in \eqref{Wick11a}.
Then,  given any finite $p \ge 1$, 
the sequence
$\{\wt R_{\dl, N}(v)\}_{N \in \N}$
converges to the limit\textup{:}
\begin{align}
\wt  R_\dl(v) =\frac 1{k+1} \int_{\T} \W(v^{k+1})dx = \lim_{N\to\infty}  \frac 1{k+1}\int_{\T} \W( (\P_N v)^{k+1} )dx 
 \label{HN3a}
\end{align} 

\noi
in $L^p(d\mu_\dl)$, as $N \to \infty$.
Furthermore, 
 there exists $\ta > 0$ such that  given any finite $p \ge 1$, we have
\begin{align}
\sup_{N \in \N\cup\{\infty\}} \sup_{0 <  \dl\le 1} 
  \| \wt  R_{\dl, N}(v)\|_{L^p(d\wt \mu_\dl)} 
< \infty, 
\label{HN3x}
\end{align}

\noi
with $\wt R_{\dl, \infty}(v) = \wt R_\dl(v)$, 
and \begin{align}
\|\wt  R_{\dl, M}(v) - \wt  R_{\dl,N}(v) \|_{L^p(d\wt  \mu_\dl)}
\leq
 \frac{C_{k, \dl}\,  p^{\frac{k+1}{2}} }{N^\ta}
\label{HN3}
\end{align}

\noi
for 
any $ M \geq N \geq 1$.
For $0 < \dl \le 1$, 
we can choose the constant $C_{k, \dl}$ in \eqref{HN3} to be independent of $\dl$
and hence the rate of convergence 
of $ \wt R_{\dl, N}(v)$ to the limit $\wt  R_\dl(v)$ is uniform in $0 < \dl \le1$.

\smallskip

\noi
\textup{(ii)}
 Let
$0 < \dl <  \infty$.
Given $N \in \N$, 
let $\wt F_N(u) = \wt F_N(u; k)$ be  the truncated renormalized nonlinearity 
in  \eqref{WILW4} given 
by 
\noi
\begin{align*}
\wt  F_N(v) := \dx \P_N\W((\P_N v)^k) = \dx \P_N H_k( \P_N v; \wt \sigma_{\dl, N}),
\end{align*}

\noi
where $\wt  \sigma_{\dl, N}$ is  as in \eqref{Wick11a}.
Then,  given any finite $p \ge 1$, 
the sequence  $\{ \wt  F_N(v)\}_{ N\in\N }$
is Cauchy in $L^p(d\wt  \mu_\dl; H^s (\T)  )$, $s < -1$, 
thus converging to a limit denoted by 
 $\wt  F(v) = \dx \W(v^k)$.
Furthermore, given any finite $p \ge 1$, we have 
\begin{align*}
\sup_{N \in \N} \sup_{0 <  \dl\le 1} \big\|  \| \wt  F_N(v) \|_{H^s_x} \big\|_{L^p(d\wt  \mu_\dl)} 
< \infty
\end{align*}

\noi
and 
\begin{align*}
\sup_{0<  \dl\le 1} \big\|  \| \wt  F_M(v) - \wt  F_N(v) \|_{H^s_x} \big\|_{L^p(d\wt  \mu_\dl)} 
\too 0
\end{align*}

\noi
for any $M \ge N$, tending to $\infty$.
In particular, the rate of convergence 
of $\wt F_N(v)$ to the limit $\wt  F(v)$
is uniform in $0<  \dl \le 1$.

\end{proposition}

\begin{proof}
Proposition \ref{PROP:HN}
follows from a straightforward modification 
of the proof of Proposition~\ref{PROP:FN}.
The only notable difference is that instead of using the bounds
\eqref{HH1a} and  \eqref{Low_Kdl} 
for $K_\dl(n)$, we need to use the bounds \eqref{Ld1} and  \eqref{Low_Ldl} for $L_\dl(n)$.
We omit details.
\end{proof}

Given $0 < \dl <  \infty$ 
and $N \in \N$, 
we define $\wt  G_{\dl, N}( \wt X_\dl)= \wt  G_{\dl, N}(\wt X_\dl; k+1)$ by 
\begin{align*}
\wt  G_{\dl, N}(\wt X_\dl) = e^{- \wt R_{\dl, N}(\wt X_\dl)    } 
                           =  e^{-\frac{1}{k+1}  \int_\T  \W( \wt X_{\dl, N}^{k+1} ) dx   } ,
\end{align*}

\noi
where $\wt R_{\dl, N}(\wt X_\dl) = \wt R_{\dl, N}(\wt X_\dl; k+1)$
is  the truncated potential energy defined 
in \eqref{HN2}.
Then, a slight modification of the proof of Proposition \ref{PROP:Gibbs1}
yields the following proposition.

\begin{proposition}\label{PROP:Gibbs2}

Let $k \in 2 \N + 1$
and fix finite $p \ge 1$.
Given any $0 < \dl <  \infty$, we have
\begin{align*}
&\sup_{N\in \N} \| \wt G_{\dl, N}(\wt X_\dl)  \|_{L^p(\O)} 
=
\sup_{N \in \N}\| \wt  G_{\dl, N}(v) \|_{L^p(d\wt \mu_\dl)}
 \le C_{p, k, \dl} < \infty.
\end{align*}

\noi
In addition, the following uniform bound holds for $0 <  \dl \le 1$\textup{:}
\begin{align}
\begin{split}
 \sup_{N \in \N} \sup_{0<  \dl \le 1}
\| \wt G_{\dl, N}(\wt X_\dl)  \|_{L^p(\O)}
& = 
\sup_{N \in \N} \sup_{0 <  \dl \le 1}
\|\wt  G_{\dl, N}(v)  \|_{L^p(d\wt \mu_\dl)}\\
& 
 \le C_{p, k} < \infty.
 \end{split}
 \label{HHN1a}
\end{align}

Define $\wt G_\dl(\wt X_\dl)$  by 
\[
\wt G_\dl(\wt X_\dl) = e^{- \wt R_\dl(\wt X_\dl)}
\]

\noi
with $\wt R_\dl(\wt X_\dl)$   as in \eqref{HN3a}.
Then, 
$\wt G_{\dl, N}(\wt X_\dl)$ converges to $\wt G_{\dl}(\wt X_\dl)$
in $L^p(\O)$.
Namely, 
we have
\begin{align*} 
\lim_{N\to\infty} \| \wt G_{\dl, N}(\wt X_\dl) - \wt G_{\dl}(\wt X_\dl) \|_{L^p(\O)} & =0.
\end{align*}

\noi
Furthermore, 
the convergence is uniform in $0<  \dl \le 1$\textup{:}
\begin{align}
\lim_{N\to\infty}\sup_{0<  \dl \le1} \| \wt G_{\dl, N}(\wt X_\dl) - \wt G_{\dl}(\wt X_\dl) \|_{L^p(\O)} =0.
 \label{HHN1b}
\end{align}

\noi
As a consequence, the uniform bounds \eqref{HHN1a}
and \eqref{HHN1b} hold even if we replace the supremum in $N \in \N$
by the supremum in $N \in \N \cup\{\infty\}$.

\end{proposition}

Theorem \ref{THM:Gibbs2}\,(i)
follows as a direct corollary to  
 Proposition \ref{PROP:Gibbs2},
 allowing us to define the limiting
 Gibbs measure $\wt \rho_\dl$ in \eqref{THM2b}.
See the discussion right after Proposition \ref{PROP:Gibbs1}. 

For $0 < \dl < \infty$, 
the Gibbs measure $\wt \rho_\dl$ is equivalent to the base Gaussian measure $\wt \mu_\dl$.
Similarly, the Gibbs measure $\rho_\KDV$ in \eqref{Gibbs5b}
equivalent to the base Gaussian measure $\wt \mu_0$.
Recalling from 
Proposition \ref{PROP:equiv2} that the base Gaussian measures $\wt \mu_\dl$, $0 < \dl < \infty$, 
and $\wt \mu_0$ are mutually singular, 
we conclude that 
the Gibbs measures $\wt \rho_\dl$ in \eqref{THM2b}
and  $\rho_\KDV$ in \eqref{Gibbs5b}
are mutually singular.
This proves the first claim in 
Theorem \ref{THM:Gibbs2}\,(ii).

\begin{proof}[Proof of Proposition \ref{PROP:Gibbs2}]
From 
\eqref{GN1} with \eqref{Wick11a}, we have
\begin{align}
\begin{aligned}
- \wt R_{\dl,N}(\wt X_\dl) 
&=   -    \frac 1{k+1}\int_{\T} H_{k+1}(\wt  X_{\dl, N} ; \wt \s_{\dl, N})  dx   \\
&\leq  \frac{2\pi}{k+1} \wt  \s_{\dl,N}^{\frac{k+1}{2}}  a_{k+1} 
 \leq \wt  A_{k, \dl}    (\log (N+1))^{\frac{k+1}{2}}
\end{aligned}
\label{HHN2}
\end{align}

\noi
for some $\wt A_{k, \dl} > 0$,
uniformly in $N \in \N$.
Then, we can simply repeat the proof of Proposition~\ref{PROP:Gibbs1}, 
using Proposition~\ref{PROP:HN} in place of Proposition \ref{PROP:FN}.

For $0 < \dl \le 1$, it follows from~\eqref{Wick11a}
and Lemma \ref{LEM:p1} that the constant $\wt A_{k, \dl}$
in \eqref{HHN2} can be chosen to be independent of $0 < \dl \le 1$. 
Similarly, 
by restricting our attention to $0 < \dl \le 1$, 
we can choose the constant
 $c_{k, \dl}$ in an analogue of \eqref{N1}
 in the current setting to be 
independent of $0 <  \dl \le 1$
since the constant $C_{k, \dl}$ in \eqref{HN3} is independent of
$0 <  \dl \le 1$.
Moreover, in applying Lemma \ref{LEM:KFC}
in Step 2 of the proof of Proposition \ref{PROP:Gibbs1}, 
we need the uniform bound~\eqref{HN3x}, replacing
\eqref{FN3a}.
This observation yields the uniform bounds \eqref{HHN1a}
and~\eqref{HHN1b}.
 \end{proof}

\begin{remark}\label{REM:K}\rm

Given $N$,  define $\s_{\KDV, N}$ by 
\begin{align}
\s_{\KDV, N} = \E\big[ X_{\KDV, N}^2(x)\big] = \frac{1}{4\pi^2} \sum_{0 < |n|\le N} \frac{2\pi}{n^2}, 
\label{HX1}
\end{align}

\noi
which is uniformly bounded in $N \in \N$.
Here, $X_{\KDV, N}$ is as in \eqref{KK1}.
We then extend the definition of $L_\dl(n)$ and $\wt G_{\dl, N}$
to the $\dl = 0$ case by setting $L_0(n) = n^2$ and 
\begin{align}
\wt G_{0, N}(X_\KDV)
=  e^{-\frac 1{k+1} \int_\T \W(X_{\KDV, N}^{k+1}) dx}, 
\label{HX2}
\end{align}

\noi
where 
$\W(X_{\KDV, N}^{k+1}) = H_{k+1}(X_{\KDV, N}; \s_{\KDV, N})$.
We also set
\begin{align}
\wt G_{0}(X_\KDV)
=  e^{-\frac 1{k+1} \int_\T \W(X_\KDV^{k+1}) dx}, 
\label{HX3}
\end{align}

\noi
where 
$\W(X_\KDV^{k+1}) = H_{k+1}( X_\KDV; \s_{\KDV})$
as in \eqref{s2}.
Then, by setting 
$\wt X_0 = X_{\KDV}$, 
Proposition~\ref{PROP:Gibbs2}
extends to $\dl = 0$.
In particular, the uniform bounds \eqref{HHN1a}
and \eqref{HHN1b} hold for $0 \le \dl \le 1$.

\end{remark}

\subsection{Convergence of the Gibbs measures in the shallow-water limit}
\label{SUBSEC:conv2}

It remains to prove that 
the Gibbs measure $ \wt  \rho_\dl $ converges weakly to $\rho_\KDV$
as $\dl\to 0$. We first state an analogue of Lemma \ref{LEM:C1}.

\begin{lemma}\label{LEM:C2}
Let $k \in 2\N + 1$ and $1 \le p < \infty$. Then, given $N \in \N$, 
we have
\begin{align*}
\lim_{\dl \to 0} \| \wt G_{\dl,N}(\wt X_\dl) - \wt G_{0, N}(X_\KDV)\|_{L^p(\O)} = 0.
\end{align*}

\noi
As a corollary, we  have
\begin{align*}
   \lim_{\dl\to\infty}
                 \| \wt G_\dl(\wt X_\dl) - \wt G_0(X_\KDV)\|_{L^p(\O)} =0 .
\end{align*}

\noi
In particular, the partition function 
$Z_{\dl}$ of the Gibbs measure $\wt \rho_\dl$ in \eqref{THM2b}
converges to the partition function $Z_\KDV = Z_0$
of the Gibbs measure $\rho_\KDV = \wt \rho_0$
in \eqref{Gibbs5b}, as $\dl\to 0$.
\end{lemma}

\begin{proof}
From Lemma \ref{LEM:p1}, 
we see  that $\wt \s_{\dl, N}$ in \eqref{Wick11a}
converges to $\s_{\KDV, N}$ in \eqref{HX1} as $\dl \to 0$.
With this observation, we can simply repeat the proof of Lemma \ref{LEM:C1}.
We omit details.
\end{proof}

We are now ready to prove weak convergence of $\wt \rho_\dl$ to $\rho_\KDV$ 
in the shallow-water limit ($\dl \to 0$).
Fix small $\eps > 0$.
 Let $A$ be any Borel subset of $H^{-\e}(\T)$ with
 $\wt \mu_0(\dd A) = 0$,
 where $\partial A$ denotes the boundary of the set $A$.
 Our goal is to show that 
\begin{align}
 \wt \rho_\dl(A) - \rho_\KDV(A) 
& \too 0
\label{CD4b}
\end{align}

\noi
as $\dl \to 0$, 
which, together with the portmanteau theorem, yields the desired weak convergence. 

By the triangle inequality, we have
\begin{align}
\begin{split}
| \wt \rho_\dl(A) - \rho_\KDV(A) |
& \le | \wt \rho_\dl(A) - \wt \rho_{\dl, N}(A) |\\
& \quad + | \wt \rho_{\dl, N}(A) - \rho_{\KDV, N}(A) |
+ |  \rho_{\KDV, N}(A) - \rho_\KDV(A) |, 
\end{split}
\label{CD5}
\end{align}

\noi
where $\rho_{\KDV, N}$ denotes the truncated Gibbs measure for $\dl = 0$
given by 
\begin{align*}
\rho_{\KDV, N} (A)
& = Z_{0, N}^{-1} \int_{H^{-\eps}} \ind_{\{v \in A\}} e^{-\frac 1{k+1} \int_\T \W(v_N^{k+1}) dx} d\wt \mu_0(v)\\
& = Z_{0, N}^{-1} \int_\O \ind_{\{X_\KDV(\o) \in A\}} \wt G_{0, N} (X_\KDV)  d\PP(\o)
\end{align*}

\noi
for any measurable set $A \subset H^{-\eps}(\T)$.
From Proposition \ref{PROP:Gibbs2}
and Remark \ref{REM:K}, we have
\begin{align}
\begin{split}
& \lim_{N \to \infty} \sup_{0 \le \dl \le 1}| \wt \rho_\dl(A) - \wt \rho_{\dl, N}(A) |\\
& \quad = \lim_{N \to \infty} \sup_{0 \le \dl \le 1}
\| \wt G_{\dl, N}(\wt X_\dl) - \wt G_{\dl}(\wt X_\dl)\|_{L^1(\O)}\\
& \quad = 0, 
\end{split}
\label{CD6}
\end{align}

\noi
with the identification $\wt X_0 = X_\KDV$, 
$\wt \rho_0 = \rho_\KDV$,
and $\wt \rho_{0, N} = \rho_{\KDV, N}$, 
where $\wt G_{0, N}(\wt X_0)$ 
and $\wt G_{0}(\wt X_0)$ are as in \eqref{HX2} and \eqref{HX3}, respectively.
Hence, in view of \eqref{CD4b}, \eqref{CD5}, 
and \eqref{CD6}, it suffices to prove
\begin{align}
\begin{split}
& \lim_{\dl \to 0}
| \wt \rho_{\dl, N}(A) - \rho_{\KDV, N}(A) |\\
&\quad  = 
\lim_{\dl \to 0 }\Big|
Z_{\dl, N}^{-1} \E[ \wt G_{\dl, N}(\wt X_\dl) \ind_A(\wt X_\dl) ]
- Z_{0, N}^{-1} \E[ \wt G_{0, N}( X_\KDV) \ind_A(X_\KDV) ]
\Big|\\
& \quad = 0
\end{split}
\label{CD7}
\end{align}

\noi
for some $N \in \N$.

First, 
note that it suffices to show that
 \begin{align}
 \E[ \wt G_{\dl, N}(\wt X_\dl) \ind_A(\wt X_\dl) ]
-  \E[ \wt G_{0, N}( X_\KDV) \ind_A(X_\KDV) ]
\too 0
 \label{CD8}
\end{align}

\noi
as $\dl \to 0$
since, by taking $A = H^{-\eps}(\T)$, 
\eqref{CD8} implies  $Z_{\dl, N} \to Z_{0, N}$ as $\dl\to0$.

By the triangle inequality, we have
\begin{align}
\begin{split}
\big| & \E[ \wt G_{\dl, N}(\wt X_\dl) \ind_A(\wt X_\dl) ]
-  \E[ \wt G_{0, N}( X_\KDV) \ind_A(X_\KDV) ]\big|\\
 & 
   \leq    \E\big[ |\wt G_{\dl, N}(\wt X_\dl) - \wt G_{0, N}( X_\KDV)|\big]\\
& \quad    
+  \E\big[ \wt G_{0, N}(X_\KDV)  |\ind_A(\wt X_\dl) -\ind_A(X_\KDV) | \big].
\end{split}
\label{CD9}
  \end{align}

\noi
From  Lemma \ref{LEM:C2}, 
we have
\begin{align}
   \E\big[ |\wt G_{\dl, N}(\wt X_\dl) - \wt G_{0, N}( X_\KDV)|\big]
   \too 0, 
\label{CD9b}
  \end{align}

\noi
as $\dl \to \infty$.
As for the second term 
on the right-hand side of \eqref{CD9}, 
we first note that 
$\s_{\KDV, N}$ defined in~\eqref{HX1} is uniformly bounded in $N \in \N$.
Then, together with  \eqref{GN1} and \eqref{HX2}, 
we conclude that 
\begin{align}
0 < \wt G_{0, N}(X_\KDV(\o)) \les 1, 
\label{CD9a}
\end{align}

\noi
uniformly in  $\o \in \O$ and $N \in \N$.
Hence, from \eqref{CD9a} and $\wt \mu_0(\dd A) = 0$
(which implies $\E[ \ind_{\dd A}(X_\KDV)] = 0$), we have
\begin{align}
\begin{split}
 \E & \big[ \wt G_{0, N}(X_\KDV)  |\ind_A(\wt X_\dl) -\ind_A(X_\KDV) | \big]\\
& \les 
 \E\big[  |\ind_A(\wt X_\dl) -\ind_A(X_\KDV) | \big]\\
& =  
 \E\big[  \ind_{\text{int} A} (X_\KDV) \cdot |\ind_A(\wt X_\dl) -\ind_A(X_\KDV) | \big]\\
& \quad +  \E\big[  \ind_{\text{int} A^c} (X_\KDV) \cdot |\ind_A(\wt X_\dl) -\ind_A(X_\KDV) | \big], 
\end{split}
\label{CD10}
\end{align}

\noi
where $\text{int} A$ denotes the interior of $A$
given by $\text{int} A = A \setminus \dd A$.
From 
Proposition \ref{PROP:equiv2}\,(i) 
and  the openness of $\text{int} A$ and $\text{int} A^c$, 
the integrands of the terms on the right-hand side
of \eqref{CD10} tend to $0$ as $\dl \to 0$.
Hence, by the bounded convergence theorem, 
we conclude that 
\begin{align}
 \E  \big[ \wt G_{0, N}(X_\KDV)  |\ind_A(\wt X_\dl) -\ind_A(X_\KDV) | \big]
\too 0, 
\label{CD11}
\end{align}

\noi
as $\dl \to 0$.
Therefore, 
putting \eqref{CD9}, \eqref{CD9b}, 
and \eqref{CD11} together, 
we conclude \eqref{CD8}, 
which in turn implies \eqref{CD7}.
Finally, from \eqref{CD5}, \eqref{CD6}, 
and \eqref{CD7}, 
we conclude \eqref{CD4b}, 
namely, weak convergence of $\wt \rho_\dl$ to $\rho_\KDV$
as $\dl \to 0$.
This concludes the proof of Theorem \ref{THM:Gibbs2}
when $ k \in 2\N+1$.

\subsection{Gibbs measures for the scaled ILW equation: variational approach}
\label{SUBSEC:ILW2}

We conclude this section by briefly going over the proof of Theorem \ref{THM:Gibbs2}
when $k = 2$, based on the variational approach as in Subsection \ref{SUBSEC:ILW1}.
The major part of the argument follows
exactly as  in Subsection \ref{SUBSEC:ILW1}
and thus we only describe necessary definitions and steps.

Fix $K > 0$ in the remaining part of this section.
Given $0 \le  \dl <  \infty$ and $N \in \N$, 
define
the truncated density $\wt G_{\dl, N}^K(v)$ by 
\begin{align*}
\wt G_{\dl, N}^K(v)
& =  
\chi_K\bigg( \int_\T \W(v_N^2) dx \bigg)e^{-\frac 1{3} \int_\T v_N^{3} dx} \\
& = 
\chi_K\bigg( \int_\T H_2(v_N; \wt \s_{\dl, N}) dx \bigg)e^{-\frac 1{3} \int_\T v_N^{3} dx} , 
\end{align*}

\noi
where $v_N = \P_N v$, 
$\wt \s_{\dl, N}$ is as in \eqref{Wick11a} when $0 < \dl < \infty$, 
and $\wt \s_{0, N} = \s_{\KDV, N}$.
As in the unscaled case discussed in Subsection \ref{SUBSEC:ILW1}, 
Theorem \ref{THM:Gibbs2} for $k = 2$ follows
once we prove the following uniform bounds.

\begin{proposition}\label{PROP:VX1}
Fix finite $p \ge 1$ and $K > 0$.
Then, given any $0 \le  \dl <  \infty$, we have
\begin{align*}
&\sup_{N\in \N} \|\wt  G_{\dl, N}^K(\wt X_\dl)  \|_{L^p(\O)} 
=
\sup_{N \in \N}\| \wt G_{\dl, N}^K(v) \|_{L^p(d\wt \mu_\dl)}
 \le C_{p, \dl, K} < \infty.
\end{align*}

\noi
In addition, the following uniform bound holds for $0  \le \dl \le 1$\textup{:}
\begin{align*}
 \sup_{N \in \N} \sup_{0 \le \dl \le 1}
\| \wt G_{\dl, N}^K(\wt X_\dl)  \|_{L^p(\O)}
& = 
\sup_{N \in \N} \sup_{0 \le \dl \le1}
\| \wt G_{\dl, N}^K(v)  \|_{L^p(d\wt \mu_\dl)}\\
& 
 \le C_{p, K} < \infty.
\end{align*}

\end{proposition}

Once we have Proposition \ref{PROP:VX1}, 
we can argue exactly  as in Subsection \ref{SUBSEC:ILW1}
to conclude Theorem \ref{THM:Gibbs2}.
In particular,  \eqref{V5} and \eqref{V6}
provide a bound on the truncated density~$\wt G_{\dl, N}^K$, uniformly in $0\le \dl \le 1$, 
replacing the defocusing bound \eqref{HHN2}.
We omit details.

In order to prove Proposition \ref{PROP:VX1}, 
we consider the truncated density 
with a taming by a power of the Wick-ordered $L^2$-norm
as in Subsection \ref{SUBSEC:ILW1}.
Given $0 \le \dl <  \infty$ and $N \in \N$, set 
\begin{align*}
\wt \RR_{\dl, N} (v)
&=   \frac13  \int_{\T }  v_N^3    dx
+ A \, \bigg| \int_{\T } \W(  v_N^2 ) dx\bigg|^2 , 
\end{align*}

\noi
where 
$\W(  v_N^2 ) = \W_{\dl, N}(  v_N^2 ) = H_2(v_N; \wt \s_{\dl, N})$.
Then, we also define the truncated density
with a taming by a power of the Wick-ordered $L^2$-norm:
\begin{align*}
\wt \GG_{\dl, N}^K(v)
& = e^{-\wt \RR_{\dl, N}(v)}= 
 e^{-\frac 1{3} \int_\T v_N^3  dx
- A| \int_\T \W(v_N^2) dx |^2}
\end{align*}

\noi
for some suitable $A > 0$.
Then, from \eqref{H6}, we have
\begin{align*}
\wt G_{\dl, N}^K(v) \le C_{A, K} \cdot \wt \GG_{\dl, N}^K(v)
\end{align*}

\noi
and, hence, 
Proposition \ref{PROP:VX1} follows once we prove the following 
uniform bounds.

\begin{proposition}\label{PROP:VX2}

Fix finite $p \ge 1$.
Then, there exists  $A_0  = A_0(p)>0 $ such that
\begin{align*}
&\sup_{N\in \N} \|\wt  \GG_{\dl, N}^K(\wt X_\dl)  \|_{L^p(\O)} 
=
\sup_{N \in \N}\| \wt \GG_{\dl, N}^K(v) \|_{L^p(d\wt \mu_\dl)}
 \le C_{p, \dl, K, A} < \infty
\end{align*}

\noi
for any  $0 \le  \dl < \infty$,  $K > 0$, 
 and $A \ge A_0$.
In addition, the following uniform bound holds for $0 \le \dl \le 1$\textup{:}
\begin{align*}
 \sup_{N \in \N} \sup_{0 \le \dl \le 1}
\| \wt \GG_{\dl, N}^K(\wt X_\dl)  \|_{L^p(\O)}
& = 
\sup_{N \in \N} \sup_{0 \le \dl \le 1}
\| \wt \GG_{\dl, N}^K(u)  \|_{L^p(d\wt \mu_\dl)}\\
& 
 \le C_{p, K, A} < \infty
\end{align*}

\noi
for any    $K > 0$  and $A \ge A_0$.

\end{proposition}

In order to set up the variational formulation, 
let us introduce some notations.
Define  $\wt Y_\dl (t)$
by 
\begin{align}
\wt Y_\dl  (t)
=  \Big(\frac 3\dl \Gdl \dx\Big)^{-\frac 12}    W(t), 
\label{PX2}
\end{align}

\noi
where $W(t)$ is as in \eqref{P1}
and $\big( \frac 3\dl \Gdl \dx\big)^{-\frac 12}$ is the Fourier multiplier operator
with the multiplier $(L_\dl(n))^{-\frac 12}$
with $L_\dl (n)$ as in \eqref{GH4}.
In view of \eqref{GH3}, we have 
$\L(\wt Y_\dl(1)) = \wt \mu_\dl$.
Given $N \in \N$, 
we set   $\wt Y_{\dl,N} = \P_N \wt Y_\dl $.
The variational formulation in the current problem is given by the following lemma.

\begin{lemma}\label{LEM:var3x}
Given $0 \le  \dl < \infty$, 
let $\wt Y_\dl$ be as in \eqref{PX2}.
Fix $N \in \N$.
Suppose that  $F:C^\infty(\T) \to \R$
is measurable such that $\E\big[|F(\wt Y_{\dl, N}(1))|^p\big] < \infty$
and $\E\big[|e^{-F(\wt Y_{\dl, N} (1))}|^q \big] < \infty$ for some $1 < p, q < \infty$ with $\frac 1p + \frac 1q = 1$.
Then, we have
\begin{align*}
- \log \E\Big[e^{-F( \wt Y_{\dl, N}(1))}\Big]
= \inf_{\dr \in \mathbb H_a}
\E\bigg[ F( \wt Y_{\dl,N}(1) +  \P_N \wt I_\dl(\dr)(1)) + \frac{1}{2} \int_0^1 \| \dr(t) \|_{L^2_x}^2 dt \bigg], 
\end{align*}

\noi
where  $\wt I_{\dl}(\dr)$ is  defined by 
\begin{align*}
\wt  I_{\dl}  (\dr)(t) = \int_0^t \big(\tfrac 3\dl \Gdl \dx\big)^{-\frac 12}     \dr(t') dt'.
%\label{P3a}
\end{align*}

\end{lemma}

With Lemma \ref{LEM:var3x} in hand, 
we can proceed as in Subsection \ref{SUBSEC:ILW1}
to prove 
Proposition \ref{PROP:VX2}
by using 
Lemma \ref{LEM:Dr2}
and the following lemma.

\begin{lemma}  \label{LEM:DrX}
	
\textup{(i)} 
Let $\eps > 0$ and fix  finite $p \ge 1$.
Then, given any $0 \le  \dl <  \infty$, 
we have 
\begin{align}
\begin{split}
\E 
\Big[  \|\wt Y_{\dl, N}(1) & \|_{W^{-\eps,\infty}}^p
 + \| \W(\wt Y_{\dl, N}^2(1))  \|_{W^{-\eps,\infty}}^p\\
& + 
\big\| \W( \wt Y_{\dl, N}^3(1)  )  \big\|_{W^{-\eps,\infty}}^p
\Big]
\leq C_{\eps, p, \dl} <\infty,
\end{split}
\label{PX4}
\end{align}

\noi
uniformly in $N \in \N$.
Furthermore, 
by restricting our attention to  $0 \le \dl \le 1$, 
we can choose the constant $C_{\eps, p, \dl}$ in \eqref{PX4}
to be independent of $\dl$.

\smallskip
	
\noi
\textup{(ii)} Let $0 \le \dl < \infty$. For any $\dr \in \Ha$, we have
\begin{align*}
\| \wt I_{\dl} (\dr)(1) \|_{H^{\frac 12}}^2 \les \int_0^1 \| \dr(t) \|_{L^2_x}^2dt, 
\end{align*}

\noi
where $\Ha$  denotes the collection of drifts, 
which are progressively measurable processes 
belonging to 
$L^2([0,1]; L^2_0(\T))$, $\PP$-almost surely, 
as in Subsection \ref{SUBSEC:ILW1}.

\end{lemma}

The proof of Lemma \ref{LEM:DrX}
follows exactly as in the proof of Lemma \ref{LEM:Dr}, 
using the lower bounds
\eqref{Ld1}
and \eqref{Low_Ldl} of $L_\dl(n)$
(in place of \eqref{HH1a} and \eqref{Low_Kdl}).
We omit details.

We conclude this section
by recalling
Proposition \ref{PROP:VX2} implies 
Proposition \ref{PROP:VX1},
which in turn implies Theorem \ref{THM:Gibbs2}
for $k = 2$.

\section{Dynamical problem}
\label{SEC:5}

In this section, we study the dynamical problem
associated with the Gibbs measures constructions in the previous sections.
In the following, we only
consider the deep-water regime $0 < \dl \le \infty$ (namely, we work on the unscaled problem 
\eqref{ILW1})
and 
present the proof of Theorem~\ref{THM:6}
since Theorem~\ref{THM:5} in the shallow-water regime ($0 \le \dl < \infty$) follows
from a similar argument.
Our main strategy is to use a compactness
argument as in \cite{BTT1, OTh1, ORT}.
In fact, 
as mentioned in Section~\ref{SEC:1},
the proof of  Theorem~\ref{THM:6}\,(i)
follows from exactly the same argument as that presented in 
\cite[Section 5]{OTh1}.
As for  the dynamical convergence result in  Theorem~\ref{THM:6}\,(ii), 
we can repeat the same argument
but with one key additional ingredient: 
the uniform (in $\dl$ and $N$) 
integrability of the (truncated) densities
(Proposition~\ref{PROP:Gibbs1}).
For conciseness of the 
presentation, 
 we restrict our attention to $2 \le \dl \le \infty$ in the following
 and discuss the proof of Theorem \ref{THM:6}.
For each fixed $0 < \dl < 2$, 
the same argument (without uniformity in $\dl$) applies to yield Theorem \ref{THM:6}\,(i).

In the remaining part of this section, 
fix   $k \in 2\N+1$
and  $s < 0$.
The $k = 2$ case follows from exactly the same 
argument by replacing 
the truncated Gibbs measure $\rho_{\dl, N}$ in \eqref{GG6}
and the Gibbs measure $\rho_\dl$ in \eqref{THM1b}
by $\rho_{\dl, N}$ in \eqref{THM1c}
and $\rho_\dl$ in \eqref{THM1d}, respectively, 
and thus we omit details.
In Subsection~\ref{SUBSEC:5.1}, 
we first study the truncated gILW equation
\eqref{WILW2}
and 
construct global-in-time invariant Gibbs dynamics 
associated with the truncated Gibbs measure $\rho_{\dl, N}$ in \eqref{GG6}
for each $N \in \N$ and $2\le  \dl \le \infty$;
see Lemma \ref{LEM:global} below.
This allows us to construct  a probability measure
$\nu_{\dl, N} = \rho_{\dl, N} \circ\Phi_{\dl, N}^{-1}$
 on space-time functions
 as the pushforward of the truncated Gibbs measure $\rho_{\dl, N}$
 under the solution map $\Phi_{\dl, N}$ for the truncated gILW equation~\eqref{WILW2}.
Then, by using the uniform (in $\dl$ and $N$) bound 
on the (truncated) densities
(Proposition~\ref{PROP:Gibbs1}), 
we prove that 
$\{\nu_{\dl, N}\}_{2\le \dl \le \infty, N \in \N}$
is tight (Proposition~\ref{PROP:tight}).
The main new point in this work is that we prove tightness
{\it not only in the frequency cutoff parameter $N \in \N$
but also in the depth parameter $2 \le \dl \le \infty$.}
In Subsection~\ref{SUBSEC:5.3}, 
we then present the proof of Theorem \ref{THM:6}
by constructing the limiting dynamics.
For each fixed $2\le \dl \le \infty$, 
we can simply repeat the argument in \cite{BTT1, OTh1, ORT}, 
based on the Skorokhod representation theorem (Lemma \ref{LEM:Sk}), 
and construct the limiting invariant Gibbs dynamics (without uniqueness) as $N \to \infty$,
yielding Theorem \ref{THM:6}\.(i).
As for proving Theorem \ref{THM:6}\.(ii), 
by exploiting the tightness
of $\{\nu_{\dl, N}\}_{  2\le \dl \le \infty ,  N \in \N}$, 
we use a diagonal argument
together with the triangle inequality
for 
the L\'evy-Prokhorov metric, characterizing weak convergence, 
to show that there exists a sequence $\{\dl_m\}_{m \in \N}$, 
tending to $\infty$, 
such that $u_{\dl_m}$ converges almost surely to some limit $u$ in $C(\R; H^s(\T))$.
Here, in order to have the claimed almost sure convergence
of $u_{\dl_m}$ to $u$, 
we apply the Skorokhod representation theorem (Lemma \ref{LEM:Sk}).
Furthermore, in order to show that $u_{\dl_m}$, $m \in \N$, 
satisfies the renormalized gILW equation \eqref{WILW1}, 
we need to apply 
the Skorokhod representation theorem (Lemma \ref{LEM:Sk})
infinitely many times (i.e.~once for each $m \in \N$).

\subsection{Pushforward of the truncated Gibbs measure}\label{SUBSEC:5.1}

Given $ 2\le \dl \le \infty$ and $N \in \N$,
consider the truncated  gILW equation \eqref{WILW2}:
\begin{align}
\begin{split}
\dt u_{\dl, N} -
 \Gdl \partial_x^2 u_{\dl, N} 
 & = F_{N}(u_{\dl, N})\\
& =  \dx \P_N \W((\P_N u_{\dl, N})^k)\\
&= \dx \P_N   H_k( \P_N u_{\dl, N} ; \s_{\dl, N} ),
\end{split}
\label{AQ1}
\end{align}

\noi
where $\sigma_{\dl, N}$ is as in \eqref{Wick1a}
and $F_N$ is as in \eqref{FN4}.
We first prove global well-posedness of~\eqref{AQ1}
and invariance of  the truncated Gibbs measure $\rho_{\dl, N}$
defined in ~\eqref{GG6}.

\begin{lemma}\label{LEM:global}
Let $ 2\le \dl \le \infty$, $N \in \N$, and $s < 0$.
Then, 
the truncated gILW equation~\eqref{AQ1}
is globally well-posed in $H^s(\T)$.
Moreover, the truncated Gibbs measure $\rho_{\dl, N}$  
is invariant under the dynamics of 
\eqref{AQ1}.

\end{lemma}

\begin{proof}
The proof of this lemma  follows from that of Lemma 5.1 in \cite{OTh1}
and thus we will be brief here.
We first decompose~\eqref{AQ1} into two parts:
\begin{align}
u_{\dl, N} = u_{\dl, N}^\low 
 + u_{\dl, N}^\high
= \P_N u_{\dl, N} +  \P_N^\perp u_{\dl, N}, 
\label{AQ2a}
\end{align}

\noi
where $\P_N^\perp = \Id - \P_N$.
Then,  $u_{\dl, N}^\low $ and $u_{\dl, N}^\high $
satisfy the following equations:

\begin{itemize}

\item[(i)] nonlinear dynamics on  the low-frequency part $\{ 0< |n| \leq N\}$:
\begin{align}
\dt u_{\dl, N}^\low  - \Gdl \dx^2   u_{\dl, N}^\low
= \dx \P_N H_k(   u_{\dl, N}^\low ; \sigma_{\dl,N} ).
\label{AQ2}
\end{align}

\noi

\smallskip

\item[(ii)]
linear dynamics
on the high frequency part $\{|n| > N\}$:
\begin{align}
 \dt u_{\dl, N}^\high   -  \Gdl \dx^2  u_{\dl, N}^\high=0.
\label{AQ3}
\end{align}

\end{itemize}

\noi
We now view the equations \eqref{AQ2} and \eqref{AQ3} 
on the Fourier side.
As a decoupled system of linear equation 
(for each frequency $|n| > N$), 
\eqref{AQ3} is globally well-posed.
As for \eqref{AQ2}, 
it is a system of finitely many ODEs
with a Lipschitz vector field
and thus by the Cauchy-Lipschitz theorem, it is locally well-posed.
Furthermore, a direct computation shows
that the $L^2$-norm of $u_{\dl, N}^\low$
is conserved under the flow of \eqref{AQ2}, 
which yields global well-posedness of~\eqref{AQ2}.
Putting together, we conclude that \eqref{AQ1} is globally well-posed.

Next, we prove  invariance of the truncated Gibbs measure $\rho_{\dl, N}$.
We first write 
 $\rho_{\dl, N}$ in~\eqref{GG6} as 
\begin{align}
\rho_{\dl, N} = \rho_{\dl, N}^\low \otimes \rho^\high_{\dl, N}, 
\label{AQ3a}
\end{align} 

\noi
where $\rho_{\dl, N}^\low$ and $\rho^\high_{\dl, N}$ are given as follows:

\begin{itemize}
\item[(i)]
the low-frequency component
$\rho_{\dl, N}^\low$ is the finite-dimensional Gibbs measure on $\P_N H^s(\T)$, 
defined by 
\begin{align*}
d \rho^\low_{\dl, N} (u)
=  Z_{\dl, N}^{-1} e^{-\frac 1{k+1} \int_{\T} H_{k+1}( u; \s_{\dl, N} )  dx} 
d \mu_{\dl, N}^\low(u) ,
\end{align*}

 \noi
 where 
$ \mu_{\dl, N}^\low =  (\P_N)_*\mu_\dl$
is the pushforward image measure under $\P_N$
of the 
base Gaussian measure $\mu_\dl$ in \eqref{Gibbs2}.
Namely, 
$ \mu_{\dl, N}^\low$ is the induced probability measure
under the map $\o \in  \O \mapsto X_{\dl, N}(\o) = \P_N X_\dl(\o)$, 
where $X_\dl$ is as in \eqref{GG3}.

\medskip

\item[(ii)] the high-frequency component
$\rho^\high_{\dl, N}$ is nothing but the Gaussian measure
$ (\P_N^\perp)_*\mu_\dl$ given as the (infinite) product 
of Gaussian measures at each frequency $|n| > N$:
\begin{align}
  (Z_{\dl, N}^\perp)^{-1} \bigotimes_{|n| > N} e^{-\frac 1{2\pi} K_\dl(n) |\ft u(n)|^2} d \ft u(n).
\label{AQ4}
\end{align}

\end{itemize}

\noi
By the classical Liouville theorem 
and the conservation of the (truncated) Hamiltonian
for~\eqref{AQ2}, 
we see that the Gibbs measure 
$\rho_{\dl, N}^\low$ is invariant under the flow of \eqref{AQ2}.
On the other hand, 
the linear dynamics \eqref{AQ3} acts as a rotation
on the Fourier coefficient at each frequency $|n| > N$, 
preserving the Gaussian measure 
at each frequency $|n| > N$ in \eqref{AQ4}.
As a result, the Gaussian measure 
$\rho^\high_{\dl, N} =  (\P_N^\perp)_*\mu_\dl$ 
is invariant under the linear dynamics \eqref{AQ3}.
In view of \eqref{AQ2a} and \eqref{AQ3a}, 
we conclude invariance of the truncated Gibbs measure $\rho_{\dl, N}$
under the flow of  the truncated gILW equation \eqref{AQ1}.
\end{proof}

As a consequence of  Lemma~\ref{LEM:global}, 
we can define the solution map $\Phi_{\dl, N}: H^s(\T)\to C(\R; H^s(\T))$ 
associated to~\eqref{AQ1}. 
More precisely, 
for $t \in \R$, 
we define  $\Phi_{\dl, N}(t): H^s(\T) \to H^s(\T)$
by 
\begin{align}
\phi \in H^s(\T) \longmapsto   \Phi_{\dl, N}(t)(\phi) =  u_{\dl, N}(t),
\label{CQ333}
\end{align}

 \noi
 where $u_{\dl, N}$ is the global-in-time solution to 
 the truncated gILW equation \eqref{AQ1}
 with initial data $u_{\dl, N}(0) = \phi$.

Next,  we introduce the pushforward image measure
$\nu_{\dl, N}$ of the truncated Gibbs measure $\rho_{\dl, N}$
under the solution map $\Phi_{\dl, N}$:
\begin{align}
\label{AQ5}
\nu_{\dl, N} = \rho_{\dl, N} \circ \Phi_{\dl, N}^{-1}.
\end{align}

\noi
Here, we view 
$\nu_{\dl, N}$ as a probability measure
on $ C(\R; H^s(\T))$ endowed with the compact-open topology, 
induced by   the following metric:
\[ 
\texttt{dist} (u, v) = \sum_{j = 1}^\infty 2^{-j}
\frac{\|u-v\|_{ C([-j, j]; H^s)}}{1+\|u-v\|_{ C([-j, j]; H^s)}}.
\]

\noi
Recall that, under this topology, a sequence $\{u_n\}_{n \in \N}\subset  C(\R; H^s(\T))$ 
converges
if and only if 
it converges uniformly on $[-K, K]$ for each finite $K >0$.
We also recall that the metric  space $\big( C(\R; H^s(\T)), \texttt{dist} \big)$ 
 is complete and separable.\footnote{Recall that 
 the space of continuous functions from 
a  separable metric space %locally compact space 
 $X$ to another separable metric space $Y$ 
 with the compact-open  topology is separable; see \cite{Mi}.
 See also the paper \cite[Corollary 3.3]{Khan86}. }
Then, it follows from  the local Lipschitz continuity of $\Phi_{\dl, N}$
that $\Phi_{\dl, N}$ is continuous
from $H^s (\T)$ into $C(\R; H^s(\T))  $,
which shows that $\nu_{\dl, N}$
is a well-defined probability measure 
on  $C(\R; H^s(\T))  $ endowed with the compact-open topology.
Note that we have
\begin{align}
 \int_{C(\R; H^s)}  F(u) d \nu_{\dl,N} (u) = \int_{H^s} F(\Phi_{\dl, N}(\phi))  d\rho_{\dl, N}(\phi)
\label{AQ6}
\end{align}

\noi
for any bounded  measurable function $F :C(\R; H^s(\T))\to  \R$.

\smallskip

Our main goal in this subsection is to 
prove the following tightness result
on $\{\nu_{\dl, N}\}_{2\le \dl \le \infty, N\in \N}$.
We point out that tightness holds not only over $N \in \N$
but also over $2\le \dl \le \infty$, 
which is the key new feature of this proposition.

\begin{proposition} \label{PROP:tight}

Let $s< 0$.
Then, 
the family $\{ \nu_{\dl, N}\}_{ 2\le \dl \le \infty, N \in \N}$
of   probability measures on 
$C(\R; H^s(\T))$
is tight, and hence is relatively compact.

\end{proposition}

Before proceeding to the proof of Proposition \ref{PROP:tight}, 
we state two auxiliary lemmas.
The first lemma 
establishes
uniform (in $\dl$ and $N$) space-time bounds
on the solutions to the truncated gILW equation \eqref{AQ1}.
We postpone its proof to the end of this subsection.

Given $1\le p \le \infty$ and $s \in \R$, 
we define the space $W^{1, p}_T H^{s}_x
= W^{1, p}([-T, T];  H^{s}(\T)) $ by the norm:
\[\| u \|_{W^{1, p}_T H^{s}_x} = \| u \|_{L^{ p}_T H^{s}_x}
+ \| \dt u \|_{L^{ p}_T H^{s}_x}. \]

\begin{lemma}\label{LEM:bound1}
Let $ s< 0$,
and fix finite $p \geq 1$.
Then, there exists $C_p > 0$ such that 
\begin{align}
\sup_{N \in \N} \sup_{2 \le \dl \le \infty}
\big\| \| u\|_{L^p_T H^s_x} \big\|_{L^p(d\nu_{\dl, N})}  & \leq C_p T^\frac{1}{p}, 
\label{BQ1}\\
\sup_{N \in \N} \sup_{2 \le \dl \le \infty}
\big\| \| u\|_{ W^{1, p}_T H^{s-2}_x} \big\|_{L^p(d\nu_{\dl, N})} &  \leq C_pT^\frac{1}{p},
\label{BQ2}
\end{align}

\end{lemma}

The following interpolation lemma
 allows us to control the H\"older regularity (in time)
 by the two quantities controlled in Lemma \ref{LEM:bound1} above.
   For $\al \in (0, 1)$ and $s \in \R$, define  the   space
 $\CC^\al_TH^{s}_x = \CC^\al([-T, T]; H^{s}(\T))$ by   the norm
\begin{align}
 \| u \|_{\CC^\al_T H^{s}_x} = \sup_{\substack{t_1, t_2 \in [-T, T]\\t_1 \ne t_2}}
\frac{\| u(t_1) - u(t_2) \|_{H^{s}}}{|t_1 - t_2|^\al} + \|u \|_{L^\infty_T H^{s}_x}.
\label{BTT0}
\end{align}

\begin{lemma}[{\cite[Lemma 3.3]{BTT1}}]\label{LEM:BTT1}
Let $T > 0$ and $1\leq p \leq \infty$.
 Suppose that 
 $u \in L^p_T H^{s_1}_x$ and $\dt u \in L^p_T H^{s_2}_x$
for some $s_2 \leq s_1$.
Then, for $ \dl > p^{-1}(s_1 - s_2)$, we have 
\begin{align} 
\| u \|_{L^\infty_TH^{s_1 - \dl}_x} \les \| u \|_{L^p_T H^{s_1}}^{1-\frac 1p}
\| u \|_{W^{1, p}_T H^{s_2}_x}^{ \frac 1p}.
\label{BTT1}
\end{align}
	
\noi
Moreover, there exist $\al > 0$ and $\theta \in [0, 1]$
such that for all $t_1, t_2 \in [-T, T]$, we have

\noi
\begin{align}
 \| u(t_2) - u(t_1)  \|_{H^{s_1 - 2\dl}} \les |t_2 - t_1|^\al  \| u \|_{L^p_T H^{s_1}_x}^{1-\theta}
\| u \|_{W^{1, p}_T H^{s_2}_x}^{ \theta}.
\label{BTT2}
\end{align}

\noi
As a consequence, we have
\begin{align}
 \| u \|_{\CC^\al_T H^{s_1 - 2\dl }_x} 
 \les
   \| u \|_{L^p_T H^{s_1}_x} + 
\| u \|_{W^{1, p}_T H^{s_2}_x}.
\label{BTT3}
\end{align}

\end{lemma}

\begin{proof}
As for \eqref{BTT1} and \eqref{BTT2}, 
see the proof of Lemma 3.3 in \cite{BTT1}.
The bound \eqref{BTT3} follows from \eqref{BTT0}, \eqref{BTT1}, 
and \eqref{BTT2} with Young's inequality.
\end{proof}

We now 
 present the proof of Proposition~\ref{PROP:tight}.

\begin{proof}[Proof of Proposition~\ref{PROP:tight}]
Let $s < s_1 < s_2 < 0$
and  $\al \in (0, 1)$.
By  the Arzel\`a-Ascoli theorem, 
 the embedding 
$\CC^\al([-T, T]; H^{s_1}(\T)) \subset 
C([-T, T];  H^{s}(\T))$ is compact
for each $T>0$.
From 
Lemma~\ref{LEM:BTT1} 
(with large $p\gg1$)
and 
Lemma~\ref{LEM:bound1}, we have
\begin{align}
\begin{split}
& \sup_{N \in \N} \sup_{2 \le \dl \le \infty}
 \big\| \| u \|_{C^\al_T H^{s_1}_x}\big\|_{L^p(d\nu_{\dl, N} )}\\
& \quad \les \sup_{N \in \N} \sup_{2 \le \dl \le \infty}  \big\| \| u \|_{L^p_T H^{s_2}_x}\big\|_{L^p(d\nu_{\dl, N} )} + 
\sup_{N \in \N} \sup_{2 \le \dl \le \infty} \big\|\| u \|_{W^{1, p}_T H^{s_2 - 2}_x}\big\|_{L^p(d\nu_{\dl, N} )}\\
& \quad \leq C_p T^\frac{1}{p}.
\end{split}
\label{BQ3}
\end{align}

Given  $j \in \N$ and $\eps \in(0,1)$,
define 
 $K_\eps$ by setting
\begin{align}
K_\eps 
:= \big\{ u \in C(\R; H^s(\T)):\, \| u \|_{C^\al_{T_j} H^{s_1}_x} 
\leq C_0 \eps^{-\frac 1p } T_j^{1+ \frac{1}{p}} 
\hspace{2mm} \text{ for all } j \in \N \big\}, 
\label{BQ4}
\end{align}

\noi
where  $T_j = 2^j$.
Then, by Chebyshev's inequality 
and  \eqref{BQ3}, we have
\begin{align*} 
\sup_{N \in \N} \sup_{2 \le \dl \le \infty}\nu_{\dl, N} (K_\eps^c) 
&\leq
\sum_{j=1}^\infty \nu_{\dl, N} \Big( \| u \|_{C^\al_{T_j} H^{s_1}_x} 
> C_0 \eps^{-\frac 1p } T_j^{1+ \frac{1}{p}}  \Big) \\
&\leq C_0^{-p}   \eps  \sum_{j=1}^\infty T_{j}^{-p -1  }   
\big\|\| u \|_{C^\al_{T_j} H^{s_1}_x}\big\|_{L^p(d\nu_{\dl, N} )}^p
\\
&\leq \bigg( C_0^{-p} C_p^p     \sum_{j=1}^\infty T_{j}^{-p  }   \bigg)\eps
< \eps,
\end{align*}

\noi
where the last step follows from choosing 
 $C_0 > 0$ sufficiently large
 in the definition \eqref{BQ4} of~$K_\eps$.

It remains to show that $K_\eps$ is compact in 
$\big( C(\R; H^s(\T)), \texttt{dist} \big)$, 
%$ C(\R; H^s(\T))$, 
namely,  endowed with the compact-open topology.
While the proof of this fact was presented in 
the proof of Proposition~5.4 in \cite{OTh1}, 
we present the argument for readers' convenience.
Let $\{u_n \}_{n \in \N} \subset K_\eps$.
It follows from  \eqref{BQ4}
that 
$\{u_n \}_{n \in \N}$ is bounded in 
$\CC^\al([-T_j, T_j]; H^{s_1}(\T))$ for each $j \in \N$
and hence is compact 
in $C([-T_j, T_j]; H^{s}(\T))$ for each $j \in \N$.
Then, by a diagonal argument, 
we can extract a subsequence $\{u_{n_\l} \}_{\l \in \N}$
that is  convergent in $C([-T_j, T_j];  H^s(\T))$ for each $j \in \N$.
Hence,  $\{u_{n_\l} \}_{\l \in \N}$ 
is convergent in$\big( C(\R; H^s(\T)), \texttt{dist} \big)$.
This proves that $K_\eps$ is  relatively 
compact in $\big( C(\R; H^s(\T)), \texttt{dist} \big)$.
It is clear that $K_\eps$ is closed as well, 
and hence we  conclude the proof.
\end{proof}

We conclude this subsection by presenting the proof of Lemma \ref{LEM:bound1}.

\begin{proof}[Proof of Lemma \ref{LEM:bound1}]

The proof   essentially follows  the same lines  in the proof of Lemma~5.5
in~\cite{OTh1}.
From  \eqref{AQ6},   the invariance of $\rho_{\dl, N}$ under the
truncated gILW dynamics~\eqref{AQ1}, 
Cauchy-Schwarz's inequality, 
 Proposition \ref{PROP:FN} (see \eqref{FN1a} with $k = 1$), 
 and 
 Proposition \ref{PROP:Gibbs1} (see \eqref{GN0b}), 
 we have 
 \begin{align}
 \begin{split}
 \big\| \| u\|_{L^p_T H^s_x} \big\|_{L^p(d\nu_{\dl, N})}
 &=  \big\| \| \Phi_{\dl, N}(t) \phi\|_{L^p_T H^s} \big\|_{L^p(d\rho_{\dl, N})} \\
 &= \big\| \| \Phi_{\dl, N}(t) \phi\|_{L^p(d\rho_{\dl, N}) H^s_x} \big\|_{L^p_T } \\
 &\les T^{\frac{1}{p}}  \|  \phi\|_{L^p(d\rho_{\dl, N}) H^s_x}   \\
 &\les T^{\frac{1}{p}}  \big\| \|  u \|_{H^s_x} \big\|_{L^{2p}(d\mu_{\dl,  N})}
 Z_{\dl, N}^{-\frac 1p } \big\| G_{\dl, N}(u) \big\|_{L^{2p}(d\mu_{\dl, N}) } \\
&  \les   T^{\frac{1}{p}}, 
\end{split}
\label{BQ5a}
 \end{align}

\noi
uniformly in $N \in \N$ and $2 \le \dl \le \infty$.
This proves \eqref{BQ1}.

Next, we prove the second bound \eqref{BQ2}.
By writing $\Gdl\dx^2 = (\Gdl\dx )\dx $, 
it follows from~\eqref{GG4} and Lemma \ref{LEM:p2}
that 
\begin{align}
\sup_{2 \le \dl \le \infty} \| \Gdl\dx^2 f \|_{H^{s-2}}
\le \|  f \|_{H^{s}}.
\label{BQ5}
\end{align}

\noi
Then, 
from  \eqref{AQ1}
and   \eqref{BQ5}, 
we have
\begin{align*}
\big\| \| &  u\|_{W^{1,p}_T H^{s-2}_x} \big\|_{L^p(d\nu_{\dl, N})} 
=  \big\| \| \dt  u\|_{L^p_T H^{s-2}_x} \big\|_{L^p(d\nu_{\dl, N})} \\
&\leq  \big\| \| \Gdl\dx^2  u\|_{L^p_T H^{s-2}_x} \big\|_{L^p(d\nu_{\dl, N})}  
+  \big\| \| F_N(u)    \|_{L^p_T H^{s-2}_x} \big\|_{L^p(d\nu_{\dl, N})}
\\
&\le  
\big\| \|    u\|_{L^p_T H^{s}_x} \big\|_{L^p(d\nu_{\dl, N})}  
+ \big\| \| F_N(u)   \|_{L^p_T H^{s-1}_x} \big\|_{L^p(d\nu_{\dl, N})}, 
\end{align*}

\noi
uniformly in $2 \le \dl \le \infty$
and $N \in \N$, 
where $F_N(u)$ is as in~\eqref{FN4}.
Then, the rest follows as in~\eqref{BQ5a} from Cauchy-Schwarz's inequality, 
 Proposition \ref{PROP:Gibbs1}, 
 and 
 Proposition \ref{PROP:FN} (see \eqref{FN1a} and~\eqref{FN4a}). 
\end{proof}

\subsection{Proof of Theorem~\ref{THM:6}}
\label{SUBSEC:5.3}

In this subsection, we present the proof of Theorem \ref{THM:6}.
We first work with fixed $2 \le \dl \le \infty$
and construct invariant Gibbs dynamics
to the renormalized gILW equation 
\eqref{WILW1}:
\begin{align}
\begin{split}
\dt u_\dl -
 \Gdl \partial_x^2 u_\dl 
& =  F(u_\dl) \\
& = \dx \W(u_\dl^k)    
\end{split}
\label{CQ0}
\end{align}

\noi
with the understanding that it corresponds to the renormalized
gBO equation \eqref{BO1} when $\dl = \infty$, 
where $F(u)$ is the limit of $F_N(u)$ in \eqref{FN4}
constructed in 
Proposition \ref{PROP:FN}\,(ii).
In view of Proposition \ref{PROP:tight},
 the family $\{ \nu_{\dl, N}\}_{N\in\N}$
is tight.
Hence, by the Prokhorov theorem (Lemma~\ref{LEM:Pro}), 
there exists a 
subsequence
$\{\nu_{\dl, N_j}\}_{j \in \N}$
converging 
weakly to some limit,\footnote{The space $\M = C(\R; H^s(\T))$ endowed with the compact-open
topology is complete and separable, 
and thus $\mathcal{P}(M) =$ the set of all the
probability measures on $\M$ is complete;
see, for example,  \cite[Theorem 6.8 on p.\,73]{Billingsley}.}
 denoted by $\nu_\dl$.
Namely, we have
\begin{align}
d_{\rm LP}\big( \nu_{\dl, N_j},  \nu_\dl \big) 
\too 0
\label{CQ0a}
\end{align}

\noi
as $j \to \infty$, 
where $d_{\rm LP}$ denotes the   L\'evy-Prokhorov metric
defined in \eqref{LP1}.

By the Skorokhod representation theorem
(Lemma~\ref{LEM:Sk}), 
there exist  some probability space $(\wt \O_\dl, \wt \F_\dl, \wt\PP_\dl)$
and $C(\R; H^s(\T))$-valued random variables $u_{\dl, N_j}$ and $u_\dl$,
such that 
\begin{align}
\L(u_{\dl, N_j}) = \nu_{\dl, N_j}
\qquad \text{and}\qquad 
 \L(u_{\dl}) = \nu_{\dl},  
\label{CQ1}
\end{align}

\noi
 and $u_{\dl, N_j}$ converges $\wt\PP_\dl$-almost surely to $u_\dl$ 
in  $C(\R; H^s(\T))$ as $j \to \infty$.
By repeating the argument in \cite{BTT1, OTh1, ORT}
(see,  in particular, Subsection 5.3 in \cite{OTh1}), 
we obtain the following global existence result
for the  gILW equation \eqref{CQ0}
with the Gibbsian initial data
(Theorem~\ref{THM:6}\,(i)).

\begin{proposition}\label{PROP:end}

Let $  u_{\dl, N_j}$, $j \in \N$,  and $u_\dl$ be  as above.
Then, 
$u_{\dl, N_j}$ and $u_\dl$
are global-in-time distributional 
solutions to 
the truncated gILW equation~\eqref{AQ1}
and  the renormalized gILW equation~\eqref{CQ0}, respectively.
Moreover,  we have
\begin{align}
\L\big( u_{\dl, N_j}(t) \big) = \rho_{\dl, N_j}
\qquad 
\text{and}
 \qquad
\L\big( u_\dl(t) \big) = \Pk 
\label{CQ2}
\end{align}

\noi
for any $t \in \R$.

\end{proposition}

\begin{proof}
While the proof of Proposition \ref{PROP:end} follows
exactly 
the same lines in Subsection 5.3 of~\cite{OTh1}, 
we present details (with some modifications from \cite{OTh1}) for readers' convenience.
We also point out that Proposition \ref{PROP:end}
will be applied iteratively in the proof of Theorem \ref{THM:6}\,(ii)
presented below.

Fix $t \in \R$.
Let $R_t:C(\R; H^s(\T)) \to H^s(\T)$ be the evaluation map defined by $R_t(v) = v(t)$. 
Note that 
$R_t$ is a continuous function.
Then, from 
\eqref{AQ5} and the invariance of the truncated Gibbs measure $\rho_{\dl, N}$
(Lemma \ref{LEM:global}), we have
\begin{align}
\begin{aligned}
\nu_{\dl, N}\circ R_t^{-1} 
&= \rho_{\dl, N}\circ \Phi_{\dl, N}^{-1} \circ R_t^{-1}
 = \rho_{\dl, N}\circ  \big( R_t\circ \Phi_{\dl, N})^{-1} \\
& =    \big( R_t\circ \Phi_{\dl, N})_* \, \rho_{\dl, N} 
 =   \big( \Phi_{\dl, N}(t)\big)_* \, \rho_{\dl, N} \\ 
& =\rho_{\dl, N}. 
\end{aligned}
\label{CQ3a}
\end{align}

\noi
 Then, 
it follows from~\eqref{CQ1} and~\eqref{CQ3a} that 
\begin{align}
\L\big( u_{\dl, N_j}(t) \big) = \nu_{\dl, N_j}\circ R_t^{-1} = \rho_{\dl, N_j} .
\label{CQ3}
\end{align}

\noi
By the construction, 
 $ u_{\dl, N_j}$ converges to $u_\dl$ in $C(\R; H^s(\T))$ almost surely
with respect to $\wt\PP_\dl$.
Thus, we have
\[
 u_{\dl, N_j}(t) = R_t\big(  u_{\dl, N_j} \big) 
 \too
u_\dl(t) =  R_t(u_\dl)
\]

\noi
almost surely as $j \to \infty$,
which in particular implies
$ u_{\dl, N_j}(t)$ converges in law to 
$u_\dl(t)$ as $j \to \infty$.
Namely, 
$\L\big( u_{\dl, N_j}(t) \big)$
converges weakly to 
$\L\big( u_{\dl}(t) \big)$
as $j \to \infty$.
On the other hand, recall from 
Theorem \ref{THM:Gibbs1}\,(i)
that 
$\rho_{\dl, N_j}$ converges to $\rho_\dl$
in total variation
as $j\to\infty$,
which in particular implies 
that   $\rho_{\dl, N_j}$ converges weakly to $\rho_\dl$.
Hence, in view of \eqref{CQ3}
and the uniqueness of the limit, 
we conclude 
 $\L \big( u_\dl(t) \big) = \rho_\dl$.
This proves \eqref{CQ2}.

 Next, we show that the random variable $u_{\dl, N_j}$ is indeed a 
 global-in-time distributional solution to \eqref{AQ1}.
Given a test function $\varphi \in \D (\R \times \T) = C^\infty_c(\R \times \T)$, 
define $V_{\varphi, j}:C(\R; H^s(\R))\to \R$ by 
\begin{align} V_{\varphi, j}(u) = \big|\jbb{\varphi, \dt u -
 \Gdl \partial_x^2 u
-  F_{N_j}(u)}\big|, 
\label{CQ33}
\end{align}

\noi
where $\jb{\cdot, \cdot}$ denotes the $\D_{t, x}$-$\D'_{t, x}$
pairing.
It is easy to see that $V_{\varphi, j}$ is continuous.
In view of the separability of $\D (\R \times \T)$, 
let $\{\varphi_m\}_{m \in  \N}$ be a countable dense
subset of $\D (\R \times \T)$.
Then, in view of 
\eqref{AQ6}, \eqref{CQ33},  and the definition \eqref{CQ333} of $\Phi_{\dl, N_j}$, we have
\begin{align}
\| V_{\varphi_m, j}\|_{L^1(d \nu_{\dl, N_j})} 
= 
 \int_{H^s} |V_{\varphi_m, j}(\Phi_{\dl, N_j}(\phi))| d \rho_{\dl, N_j}(\phi)
 = 0
 \label{CQ4}
\end{align}

\noi
for any $m \in \N$.
Namely, there exists a set $\Si_m \subset 
 C(\R; H^s(\T))$ such that 
 $\nu_{\dl, N_j}(\Si_m) = 1$
and $V_{\varphi_m, j} (u) = 0$ for any $u \in \Si_m$.
Now, set $\Si = \bigcap_{m \in \N} \Si_m$.
Then, we have 
 $\nu_{\dl, N_j}(\Si) = 1$
 and, moreover, $V_{\varphi, j}(u) = 0$
for any $u \in \Si$
and $\varphi \in \D(\R\times \T)$, 
where the latter claim follows from~\eqref{CQ4} and the density of $\{\varphi_m\}_{m \in \N}$.

Finally, we prove  that the random variable $u_\dl$ is   a 
 global-in-time distributional solution to~\eqref{CQ0}.
 It follows from 
 the almost sure convergence
of   $u_{\dl, N_j}$   to $u_\dl$ in $C(\R; H^s(\T))$
that 
\begin{align}
\dt  u_{\dl, N_j}  - \Gdl \dx^2   u_{\dl, N_j} 
\too  
   \dt u_\dl - \Gdl \dx^2 u_\dl
\label{CQ4b}
\end{align}

\noi
in $\D'(\R\times \T)$, 
$\wt \PP_\dl$-almost surely, as $j \to \infty$.

Next, we show   almost sure convergence  
of the truncated nonlinearity 
$F_{N_j}(u_{\dl, N_j})$
to $F(u_\dl) = \dx \W(u_\dl)$.
Given $M \in \N$, write 
\begin{align}
\begin{aligned}
F_{N_j}(u_{\dl, N_j}) - F(u_\dl)
& = \big(F_{N_j}(u_{\dl, N_j}) - F_M(u_{\dl, N_j})\big)
+  \big( F_M (u_{\dl, N_j }) - F_M(u_\dl ) \big) \\
&\quad +  \big( F_M (u_\dl ) -F(u_\dl) \big).
\end{aligned}
\label{CQ5}
\end{align}

\noi
Noting that 
$u\in C(\R; H^s(\T)) \mapsto 
F_M(u)  \in C(\R; H^{s-1}(\T))$
 is continuous, 
it follows from 
 the almost sure convergence
of   $u_{\dl, N_j}$   to $u_\dl$ in $C(\R; H^s(\T))$
that 
\[
F_M( u_{\dl, N_j}) \too F_M( u_\dl)
\]

\noi
in $C(\R; H^{s-1}(\T))$, 
$\wt{\mathbb{P}}_\dl$-almost surely, as $j\to\infty$.
As for
the first term on the right-hand side of~\eqref{CQ5}, 
 for fixed $T > 0$, 
it follows from  
 \eqref{AQ6}, 
the invariance of the truncated Gibbs measure $\rho_{\dl, N}$,
and 
 Proposition~\ref{PROP:Gibbs1}
 that 
\begin{align*}
\big\|& 
\| F_{N_j}(u_{\dl, N_j})   - F_M(u_{\dl, N_j}) \|_{L^2_T H^{s-1}_x}
\big\|_{L^2(\wt \O_\dl )} \\
&
= \big\|
\| F_{N_j}(u)   - F_M(u) \|_{L^2_T H^{s-1}_x}
\big\|_{L^2(d\nu_{\dl, N_j})} 
\\
& = 
\big\|
\| F_{N_j}(\phi)    - F_M(\phi )\|_{L^2(d\rho_{\dl, N_j}) H^{s-1}_x}
\big\|_{L^2_T}  
\\
&
\les T^\frac{1}{2} Z_{\dl, N_j}^{-\frac 12 } \|G_{\dl, N_j}\|_{L^4(\O)}
\| F_{N_j}(\phi)    - F_M(\phi ) 
\|_{L^4(d\mu_\dl) H^{s-1}_x}\\
&
\les T^\frac{1}{2} 
\| F_{N_j}(\phi)    - F_M(\phi ) 
\|_{L^4(d\mu_\dl) H^{s-1}_x}, 
\end{align*}
 
\noi
where the implicit constants are independent of $N_j$.
By applying
Proposition \ref{PROP:FN}\,(ii), 
we conclude that 
the first term on the right-hand side of~\eqref{CQ5} converges to $0$
 in $L^2(\wt \O_\dl; L^2([-T, T]; H^{s-1}(\T)))$
 as $j, M \to \infty$.
Hence, by extracting a subsequence, 
the first term on the right-hand side of~\eqref{CQ5} converges to $0$
 in $L^2([-T, T]; H^{s-1}(\T))$, 
 $\wt \PP_\dl$-almost surely, as $j, M \to \infty$.
A similar argument shows that, by extracting a subsequence,  
the third term on the right-hand side of~\eqref{CQ5} converges to $0$
 in $L^2([-T, T]; H^{s-1}(\T))$, 
 $\wt \PP_\dl$-almost surely, as $M \to \infty$.

Putting all together with \eqref{CQ5}, 
we conclude that, up to a subsequence,
$F_{N_j}(u_{\dl, N_j})$ 
converges to $F(u_\dl)$
 in $L^2([-T, T]; H^{s-1}(\T))$, 
 $\wt \PP_\dl$-almost surely, as $j \to \infty$.
Since the choice of $T>0$ was arbitrary, 
we  can apply this argument  for $T_m = 2^m$, $m\in \N$.
Thus, with $m = 1$, 
there exists a subsequence
$F_{N_{j_1}}(u_{\dl, N_{j_1}})$ 
and a set $\Si_1$ of full $\wt \PP_\dl$-probability
such that 
$F_{N_{j_1}}(u_{\dl, N_{j_1}})(\o)$ 
converges to $F(u_\dl)(\o)$
 in $L^2([-T_1, T_1]; H^{s-1}(\T))$
for each $\o \in \Si_1$ as $j_1 \to \infty$.
For each $m \geq 2$, 
we can extract a further subsequence
$F_{N_{j_m}}(u_{\dl, N_{j_m}})$ 
of $F_{N_{j_{m-1}}}(u_{\dl, N_{j_{m-1}}})$ 
and a subset $\Si_m \subset \Si_{m-1}$
of full $\wt \PP_\dl$-probability
such that 
$F_{N_{j_m}}(u_{\dl, N_{j_m}})(\o)$ 
converges to $F(u_\dl)(\o)$
 in $L^2([-T_m, T_m]; H^{s-1}(\T))$
for each $\o \in \Si_m$ as $j_m \to \infty$.
By a diagonal argument, 
we conclude that, 
passing to a  subsequence, we have
$F_{N_{j}}(u_{\dl, N_{j}} )$
converges to $F(u_\dl)$
in $L^2_{t, \text{loc}}H^{s-1}(\T)$, 
$\wt \PP_\dl$-almost surely, % with respect to $\wt{\mathbb{P}}_\dl$, 
which in particular implies
that this subsequence converges
to $F(u_\dl)$
in $\mathcal D'(\R\times \T)$, 
$\wt \PP_\dl$-almost surely.
Therefore, 
together with \eqref{CQ4b}, 
we conclude that $u_\dl$ is a global-in-time distributional solution to~\eqref{CQ0}.
\end{proof}

Finally, we present the proof of Theorem \ref{THM:6}\,(ii).
In the discussion at the beginning of this subsection, 
we used Proposition \ref{PROP:tight}
and 
 the Prokhorov theorem (Lemma~\ref{LEM:Pro})
 to conclude that, 
for each fixed $2 \le \dl \le \infty$, 
 there exists a sequence $N_j  \to \infty$
such that
\eqref{CQ0a} holds. 
In the following,  we iteratively apply this argument
for integers $\dl \ge 2$
and apply  a diagonal argument.

\smallskip

\begin{itemize}

\item[(i)]
Let $\dl  = 2$.
Then, 
it follows from  Proposition \ref{PROP:tight} that 
 the family $\{ \nu_{2, N}\}_{N\in\N}$
is tight.
Hence, by the Prokhorov theorem (Lemma~\ref{LEM:Pro}), 
there exists
a weakly convergent subsequence 
 $\{ \nu_{2, N_j^{(2)}}\}_{j \in \N}$.
Namely, 
there exists 
a probability measure $\nu_2$ on $C(\R; H^s(\T))$ such that 
$\dlp(\nu_{2, N_j^{(2)}}, \nu_2 )\to 0$
as $j \to \infty$.

\smallskip

\item[(ii)]

For  $\dl  = 3$, 
we apply the same argument to 
$\{ \nu_{3, N_j^{(2)}}\}_{j \in\N}$
to conclude that 
there exists a weakly convergent subsequence 
 $\{ \nu_{3, N_j^{(3)}}\}_{j \in \N}$ with $\{ N_j^{(3)}\}_{j \in \N} \subset \{ N_j^{(2)}\}_{j \in \N}$.
Namely, 
there exists 
a probability measure $\nu_3$ on $C(\R; H^s(\T))$ such that 
$\dlp(\nu_{3, N_j^{(3)}}, \nu_3 )\to 0$
as $j \to \infty$.

\smallskip

\item[(iii)] 
We iterate this procedure for each integer $\dl \ge 4$
and construct a 
weakly convergent subsequence 
 $\{ \nu_{\dl , N_j^{(\dl)}}\}_{j \in \N}$ with $\{ N_j^{(\dl)}\}_{j \in \N} \subset \{ N_j^{(\dl-1)}\}_{j \in \N}$.
Namely, 
there exists 
a probability measure $\nu_\dl $ on $C(\R; H^s(\T))$ such that 
\begin{align}
\dlp(\nu_{\dl, N_j^{(\dl)}}, \nu_\dl )\too 0, 
\label{CQ7}
\end{align}
as $j \to \infty$.

\smallskip

\item[(iv)] 

Let  $\N_{\ge 2} = \N \cap [2, \infty)$. 
We  take a diagonal sequence $\{\nu_{\dl, N^{\dl}_{j(\dl)}}  \}_{\dl  \in \N_{\ge 2}}$, 
where $j(\dl)$ is chosen such that $j(\dl)$ is increasing in $\dl$
and 
\begin{align}
\dlp( \nu_{\dl, N^{\dl}_{j(\dl)}}, \nu_\dl) \leq \frac{1}{\dl}.
\label{CQ7a}
\end{align}

\end{itemize}

By Proposition \ref{PROP:tight} and the Prokhorov theorem (Lemma \ref{LEM:Pro}),
the family  $\big\{ \nu_{\dl , N_{j(\dl)}^{(\dl)}}\big\}_{\dl \in \N_{\ge2}}$
is tight and thus admits a weakly convergent subsequence
$\big\{ \nu_{\dl_m, N_{j(\dl_m)}^{(\dl_m)}}\big\}_{m \in \N}$
to some limit, which we denote by $\nu_\infty$.
Namely, we have 
\begin{align}
\dlp
(\nu_{\dl_m, N_{j(\dl_m)}^{(\dl_m)}}, \nu_\infty) \too 0, 
\label{CQ8}
\end{align}

\noi
as $m \to \infty$.
By the
triangle inequality for the L\'evy-Prokhorov metric $\dlp$
with \eqref{CQ7a} and~\eqref{CQ8}, 
we have
\begin{align}
\begin{split}
\dlp( \nu_{\dl_m}, \nu_{\infty} ) 
&\leq \dlp(\nu_{\dl_m}, \nu_{\dl_m, N_{j(\dl_m)}^{(\dl_m)}})
 + \dlp(\nu_{\dl_m, N_{j(\dl_m)}^{(\dl_m)}},  \nu_\infty  )\\
&\leq \frac 1{\dl_m}
 + \dlp(\nu_{\dl_m, N_{j(\dl_m)}^{(\dl_m)}},  \nu_\infty  )\\
 & \too 0, 
\end{split}
\label{CQQ}
\end{align}

\noi
as $m \to \infty$ (and hence $\dl _m\to \infty$).
Hence, 
$ \nu_{\dl_m}$  converges weakly 
 to $\nu_\infty$ as $m \to \infty$.

By  the 
Skorokhod representation theorem (Lemma \ref{LEM:Sk}), 
there exist  a probability space $(\wt \O, \wt \F, \wt\PP)$
and  $C(\R; H^s(\T))$-valued random variables $u_{\dl_m}$ and 
$u$ such that
\begin{align}
\L(u_{\dl_m} ) = \nu_{\dl_m}\qquad \text{and}
\qquad \L(u) = \nu_\infty
\label{CQ9}
\end{align}

\noi
and $u_{\dl_m}$
 converges $\wt \PP$-almost surely to  
$u$ in $C(\R; H^s(\T))$  as $m \to \infty$.

Next, we show that 
$u_{\dl_m}$ is a global-in-time distributional solution to
the renormalized gILW equation~\eqref{CQ0} (with $\dl = \dl_m$).
It follows from   \eqref{CQ7}
and 
 the 
Skorokhod representation theorem (Lemma \ref{LEM:Sk})
that 
there exist  a probability space $(\wt \O_m, \wt \F_m, \wt\PP_m)$
and  $C(\R; H^s(\T))$-valued random variables $\wt u_{\dl_m, N_j^{(\dl_m)}}$ and 
$\wt u_{\dl_m}$ such that
\begin{align}
\L(\wt u_{\dl_m, N_j^{(\dl_m)}} ) = \nu_{\dl_m, N_j^{(\dl_m)}}\qquad \text{and}
\qquad \L(\wt u_{\dl_m}) = \nu_{\dl_m}
\label{CQ10}
\end{align}

\noi
and $\wt u_{\dl_m, N_j^{(\dl_m)}}$
 converges $\wt \PP_m$-almost surely to  
$\wt u_{\dl_m}$ in $C(\R; H^s(\T))$  as $j \to \infty$.
Arguing as in the proof of Proposition \ref{PROP:end}, 
we see that 
$\wt u_{\dl_m}$
is a  global-in-time distributional 
solution to 
  the renormalized gILW equation~\eqref{CQ0}.
Hence, from \eqref{CQ9}
and \eqref{CQ10}, 
we conclude that 
$ u_{\dl_m}$
is a  global-in-time distributional 
solution to 
  the renormalized gILW equation~\eqref{CQ0}.

It remains  to  show that $ u$ satisfies 
the renormalized gBO equation \eqref{BO1} in the
distributional sense.
The almost sure convergence of $u_{\dl_m}$
to $u$ implies that 
\begin{align}
 \dt {u_{\dl_m} } - \Gdl \dx^2   {u_{\dl_m}} 
\too   \dt u- \mathcal H \dx^2 u
\label{CQ10a}
\end{align}

\noi
in $\mathcal D'(\R\times \T)$ as $m  \to \infty$.
Next, we discuss convergence of the nonlinearity.
Let $F(u_{\dl})=\dx \W (u_{\dl  }^{k})$
be as in Proposition \ref{PROP:FN}\,(ii).
Given $M \in \N$, write 
\begin{align}
\begin{aligned}
F(u_{\dl_m}) - F(u) 
& =
\big(F(u_{\dl_m}) - F_M(u_{\dl_m})  \big)
 + \big( F_M(u_{\dl_m}) - F_M(u)  \big) \\
&\quad + \big(F_M(u) - F(u)   \big), 
\end{aligned}
\label{CQ11}
\end{align}

\noi

\noi
From 
the continuity of $F_M$
and
the almost sure convergence of $u_{\dl_m}$
to $u$, 
we see that the second term on the right-hand side of \eqref{CQ11}
tends to $0$ 
in $C(\R; H^{s-1}(\T))$, 
$\wt \PP$-almost surely, 
as $m \to \infty$.
As for the first and third terms on the right-hand side of~\eqref{CQ5}, 
we need to exploit the uniform (in $\dl$ and $N$) bounds, 
which is the main difference from the proof of Proposition~\ref{PROP:end}
presented above.
Let $T > 0$.
Then, from 
 \eqref{AQ6}
and the invariance of the truncated Gibbs measure $\rho_{\dl, N}$,
we have 
 \begin{align}
\begin{aligned}
\big\|& 
\| F(u_{\dl_m})   - F_M(u_{\dl_m}) \|_{L^2_T H^{s-1}_x}
\big\|_{L^2(\wt \O)} \\
&
= \big\|
\| F(u)   - F_M(u) \|_{L^2_T H^{s-1}_x}
\big\|_{L^2(d\nu_{\dl_m})} 
\\
& = 
\big\|
\| F(\phi)    - F_M(\phi )\|_{L^2(d\rho_{\dl_m})H^{s-1}_x}
\big\|_{L^2_T}  
\\
&
\les T^\frac{1}{2} Z_{\dl_m}^{-\frac 12 } \|G_{\dl_m}\|_{L^4(\O)}
\| F(\phi)    - F_M(\phi ) 
\|_{L^4(d\mu_{\dl_m}) H^{s-1}_x}
\end{aligned}
\label{CQ12}
\end{align}

\noi
with the understanding that
$u_{\dl_\infty}  = u$
 when $m = \infty$.
From  Proposition~\ref{PROP:Gibbs1}, we have
\begin{align}
\begin{split}
\sup_{m \in \N \cup\{\infty\}} Z_{\dl_m}^{-\frac 12 } +
\sup_{m \in \N \cup\{\infty\}}\|G_{\dl_m}\|_{L^4(\O)}& \les 1.
\end{split}
\label{CQ13}
\end{align}

\noi
Then, from \eqref{CQ12}, \eqref{CQ13}, and Proposition \ref{PROP:FN}\,(ii)
(see \eqref{FN5} with $(M, N) = (\infty, N)$), 
we conclude that 
the first and third terms on the right-hand side of~\eqref{CQ11} converge to $0$
 in $L^2(\wt \O_\dl; L^2([-T, T]; H^s(\T)))$
 as $M \to \infty$.
Then, by first taking $m \to \infty$ and then $M \to \infty$ in \eqref{CQ11}, 
we conclude that, by extracting a subsequence,  
$F(u_{\dl_m})$ converges to $F(u)$ 
in $L^2([-T, T]; H^{s-1}(\T))$, $\wt \PP$-almost surely, 
as $m \to \infty$.
By repeating the argument at the end of the proof of Proposition 
\ref{PROP:end}, 
we see that 
up to a further subsequence, 
 $F(u_{\dl_m})$ converges
to $F(u)$
in $\mathcal D'(\R\times \T)$, 
$\wt \PP$-almost surely, as $m \to \infty$.
Therefore, 
together with \eqref{CQ10a}, 
we conclude that $u$ is a global-in-time distributional solution to
the renormalized gBO equation~\eqref{CQ0}.
This concludes the proof of Theorem \ref{THM:6}\,(ii).

\begin{remark}\rm \label{REM:END}
In this paper, 
we considered probability measures on $H^{-\eps}(\T)$ for fixed small $\eps >0$. 
In the following, we briefly explain how to remove the dependence on $\eps$.
First, we set
\[
H^{0-}(\T) : = \bigcap_{s > 0}  H^{-s}(\T) =  \bigcap_{j\in\N }  H^{-s_j}(\T),
\]
with $s_j = \frac{1}{j}$.
Then,  we equip $H^{0-}(\T)$ with the following distance:
\[
\textbf{d}(f, g) 
= \sum_{j=1}^\infty 2^{-j}
 \frac{  \| f - g\|_{H^{-s_j}}    }{1 +  \| f - g\|_{  H^{-s_j} }  }.
\]

\noi
By definition, we have $\textbf{d}(f_n, f)\to 0$ if and only if
$f_n$ converges to $f$ in $H^{-s_j}(\T)$ for each $j\in\N$.
Let $\mathbf{D}$ be the set of smooth functions $Q\in C^\infty(\T)$ of the form
\[
Q(x) =\sum_{|n|\leq N} q_n e_n(x),  
\]

\noi
with $q_n\in\mathbb{Q}$ and $N\in\N$.
Then,  $\mathbf{D}$ is a countable
dense subset of $H^{-s_j}(\T)$ for any $j\in\N$.
Let $f\in H^{0-}(\T)$.
Then, for each $j \in \N$, 
there exists $Q_{j, N}\in\mathbf{D}$
such that 
\[
\| Q_{j,N} - f \|_{H^{-s_j}} \leq 2^{-N}.
\]

\noi
Now, set 
 $Q_N = Q_{N,N}\in\mathbf{D}$, $N \in \N$.
Then,   given $\eps > 0$, by choosing $N \ge \frac{1}{\eps}$,
we have
\[
\| Q_{N} - f \|_{H^{-\eps}} \leq \| Q_{N} - f \|_{H^{-\frac{1}{N}}} 
=  \| Q_{N, N} - f \|_{H^{-\frac{1}{N}}} \leq  2^{-N}.
\]

\noi
Hence, we have
\begin{align*}
\textbf{d}(Q_N, f) 
& \le  \sum_{j=1}^N 2^{-j}
 \| Q_N - f\|_{H^{-\frac 1j}} 
 + 
  \sum_{j=N+1}^\infty 2^{-j}\\
& \le  2^{-N}
+  2^{-N}
\too 0, 
\end{align*}

\noi
as $N \to \infty$.
In other words, we just proved that
$\mathbf{D}$ is also a countable dense subset of $H^{0-}(\T)$ with respect to the
metric $\textbf{d}$.  
Hence, from \cite{Khan86}, we see  that $C(\R; H^{0-}(\T))$ is separable.\footnote{Note that we have 
\[
C(\R; H^{0-}(\T)) = \bigcap_{j=1}^\infty C\big(\R; H^{-\frac{1}{j}}(\T) \big).
\]}
This 
 allows us to repeat the entire paper by replacing 
 $C(\R; H^{-\eps}(\T))$ with 
$C(\R; H^{0-}(\T))$.

\end{remark}

\begin{ackno}\rm
 The authors would like to thank an anonymous referee for the helpful comments.
G.L., T.O., and~G.Z.~were supported by the European Research Council
(grant no.~864138 ``SingStochDispDyn'').
G.L.~was also supported 
 by the Maxwell Institute Graduate School in Analysis and its Applications, a Centre for Doctoral Training funded by the UK Engineering and Physical Sciences Research Council (Grant EP/L016508/01), the Scottish Funding Council, Heriot-Watt University 
and the University of Edinburgh, 
and
by the EPSRC New Investigator Award (grant no.~EP/S033157/1).

\end{ackno}


\begin{thebibliography}{99}




\bibitem{ABFS}
L.~Abdelouhab, J.L.~Bona, M.~Felland, J.-C.~Saut, 
{\it Nonlocal models for nonlinear, dispersive waves}, 
Phys. D 40 (1989), no. 3, 360--392.


\bibitem{AS}
M.J.~Ablowitz, H.~Segur, 
{\it  Solitons and the inverse scattering transform},
SIAM Studies in Applied Mathematics, 4. Society for Industrial and Applied Mathematics (SIAM), Philadelphia, Pa., 1981. x+425 pp.



\bibitem{Ahl}
L.~Ahlfors, 
{\it Complex analysis. An introduction to the theory of analytic functions of one complex variable,} Third edition. International Series in Pure and Applied Mathematics. McGraw-Hill Book Co., New York, 1978. xi+331 pp. 




\bibitem{AC}
S.~Albeverio, A.~Cruzeiro,
{\it  Global flows with invariant (Gibbs) measures for Euler and Navier-Stokes two dimensional fluids,} 
Comm. Math. Phys. 129 (1990), no. 3, 431--444.
 

 

\bibitem{BG} 
N. Barashkov, M. Gubinelli,
 {\it A variational method for $\Phi^4_3$}, 
 Duke Math. J. 169 (2020), no. 17, 3339--3415.

 
% \bibitem{BL}
%N.~Barashkov, P.~Laarne,
%{\it 
% Invariance of $\phi^4$ measure under nonlinear wave and Schr\"odinger equations on the plane},
% arXiv:2211.16111 [math.AP].

 
\bibitem{Bass}
R.~Bass,
{\it Stochastic processes},
Cambridge Series in Statistical and Probabilistic Mathematics, 33. 
Cambridge University Press, Cambridge, 2011. xvi+390 pp.




\bibitem{Benyi}
\'A.~B\'enyi, T.~Oh, 
{\it Modulation spaces, Wiener amalgam spaces, and Brownian motions}, 
Adv. Math. 228 (2011), no. 5, 2943--2981.





\bibitem{BOP4}
\'A.~B\'enyi, T.~Oh, O.~Pocovnicu.
{\it On the probabilistic Cauchy theory for nonlinear dispersive PDEs}, 
Landscapes of time-frequency analysis, 1--32,
 Appl. Numer. Harmon. Anal., Birkh\"auser/Springer, Cham, 2019. 
 
 




\bibitem{Billingsley2}
P.~Billingsley,
{\it Probability and measure}, 
Third edition. Wiley Series in Probability and Mathematical Statistics.
A Wiley-Interscience Publication. John Wiley \& Sons, Inc., New York, 1995. xiv+593 pp.



\bibitem{Billingsley}
P.~Billingsley,
{\it  Convergence of probability measures,}
Second edition. 
Wiley Series in Probability and Statistics: Probability and Statistics. 
A Wiley-Interscience Publication. 
John Wiley \& Sons, Inc., New York, 1999. x+277 pp.







\bibitem{BD}
M.~Bou\'e, P.~Dupuis,
{\it A variational representation for certain functionals of Brownian motion},
Ann. Probab. 26 (1998), no. 4, 1641--1659.

%
\bibitem{BO93}
J.~Bourgain, 
{\it Fourier transform restriction phenomena for certain lattice subsets and applications to nonlinear evolution equations, 
 II. The KdV-equation,} 
 Geom. Funct. Anal. 3 (1993), no. 3, 209--262.
 
 
 

\bibitem{BO94}
J.~Bourgain, 
{\it Periodic nonlinear Schr\"odinger equation and invariant measures}, 
Comm. Math. Phys. 166 (1994), no. 1, 1--26.




\bibitem{BO96}
J.~Bourgain, 
{\it Invariant measures for the 2D-defocusing nonlinear Schr\"odinger equation}, 
Comm. Math. Phys. 176 (1996), no. 2, 421--445. 






\bibitem{BO00}
J.~Bourgain, 
{\it Invariant measures for NLS in infinite volume},
 Comm. Math. Phys. 210 (2000), no. 3, 605--620.
 



\bibitem{bring}
B.~Bringmann, 
{\it Invariant Gibbs measures for the three-dimensional wave equation with a Hartree nonlinearity I: measures},  
Stoch. Partial Differ. Equ. Anal. Comput. 10 (2022), no. 1, 1--89.



\bibitem{BS}
D.~Brydges, G.~Slade, 
{\it Statistical mechanics of the 2-dimensional focusing nonlinear Schr\"odinger equation,}
 Comm. Math. Phys. 182 (1996), no. 2, 485--504.












\bibitem{BTT1}
N.~Burq, L.~Thomann, N.~Tzvetkov, 
{\it Remarks on the Gibbs measures for nonlinear dispersive equations}, 
Ann. Fac. Sci. Toulouse Math. 27 (2018), no. 3, 527--597.




\bibitem{BT1}
N.~Burq, N.~Tzvetkov, 
{\it Random data Cauchy theory for supercritical wave equations. I. Local theory,}
 Invent. Math. 173 (2008), no. 3, 449--475. 



\bibitem{BT3}
N.~Burq, N.~Tzvetkov, 
{\it Probabilistic well-posedness for the cubic wave equation,}
 J. Eur. Math. Soc. (JEMS) 16 (2014), no. 1, 1--30. 






\bibitem{CK21} 
A.~Chapouto, N.~Kishimoto,
{\it Invariance of the Gibbs measures for periodic generalized Korteweg-de Vries equations,}
Trans. Amer. Math. Soc.
 375 (2022), no. 12, 8483--8528. 

%arXiv:2104.07382v2 [math.AP].



\bibitem{CFLOP}
A.~Chapouto, J.~Forlano, G.~Li, T.~Oh, D.~Pilod,
{\it  Intermediate long wave equation in negative Sobolev spaces},
Proc. Amer. Math. Soc. Ser. B 11 (2024), 452--468. 




\bibitem{CLOZ}
A.~Chapouto, G.~Li, T.~Oh, 
{\it
Deep-water and shallow-water limits
of statistical equilibria
for the intermediate long wave equation}, 
arXiv:2409.06905 [math.AP].




\bibitem{CLOP}
A.~Chapouto, G.~Li, T.~Oh, D.~Pilod,
{\it Deep-water limit of the intermediate long wave equation in $L^2$},
Math. Res. Lett. 31 (2024), no. 6, 1655--1692. 




\bibitem{CO}
J.~Colliander, T.~Oh, 
{\it  Almost sure well-posedness of the cubic nonlinear Schr\"odinger equation below $L^2(\T)$}, 
Duke Math. J. 161 (2012), no. 3, 367--414. 




\bibitem{DP}
G.~Da Prato, 
{\it An introduction to infinite-dimensional analysis}, 
Revised and extended from the 2001 original by Da Prato. Universitext. Springer-Verlag, Berlin, 2006. x+209 pp.




\bibitem{DPD}
G.~Da Prato, A.~Debussche, 
{\it Two-dimensional Navier-Stokes equations driven by a space-time white noise}, 
J. Funct. Anal. 196 (2002), no. 1, 180--210.




\bibitem{DPT1}
G.~Da Prato, L.~Tubaro, 
{\it Wick powers in stochastic PDEs: an introduction,} 
Quantum and stochastic mathematical physics, 1--15,
Springer Proc. Math. Stat., 377, Springer, Cham, [2023], \copyright2023. 






\bibitem{DZ}
G.~Da Prato, J.~Zabczyk, 
{\it Stochastic equations in infinite dimensions,} 
Second edition. Encyclopedia of Mathematics and its Applications, 152. Cambridge University Press, Cambridge, 2014. xviii+493 pp. 





 \bibitem{De15} 
Y.~Deng,
 {\it Invariance of the Gibbs measure for the Benjamin-Ono equation},
J. Eur. Math. Soc. (JEMS) 17 (2015), no. 5, 1107--1198.




\bibitem{DTV}
Y.~Deng, N.~Tzvetkov, N.~Visciglia, 
{\it Invariant measures and long time behaviour for the Benjamin-Ono equation III},
 Comm. Math. Phys. 339 (2015), no. 3, 815--857. 





\bibitem{Dudley_book}
R.M.~Dudley, 
{\it Real analysis and probability,}
Revised reprint of the 1989 original. Cambridge Studies in Advanced Mathematics, 74.
Cambridge University Press, Cambridge, 2002. x+555 pp.



\bibitem{Feldman}
J.~Feldman,
{\it Equivalence and perpendicularity of Gaussian processes},
Pacific J. Math. 8 (1958), 699--708. 



\bibitem{Fried}
L.~Friedlander, 
{\it An invariant measure for the equation $u_{tt} - u_{xx} +u^{3}=0$}, 
Comm. Math. Phys. 98 (1985), 1--16.






\bibitem{FKSS}
J.~Fr\"ohlich, A.~Knowles, B.~Schlein, V.~Sohinger, 
{\it The mean-field limit of quantum Bose gases at positive temperature},
 J. Amer. Math. Soc. 35 (2022), no. 4, 955--1030.
 
 
 


\bibitem{FIH}
R.~Fukuizumi, M.~Hoshino, T.~Inui, 
{\it Non relativistic and ultra relativistic limits in 2D stochastic nonlinear damped Klein-Gordon equation}, 
Nonlinearity 35 (2022), no. 6, 2878--2919. 







\bibitem{GKT}
P.~G\'erard, T.~Kappeler, P.~Topalov, 
{\it Sharp well-posedness results of the Benjamin-Ono equation in $H^s(\T, \R)$ and qualitative properties of its solution},
Acta Math.
231 
(2023), no. 1, 
 31--88.






\bibitem{GJ}
J.~Glimm, A.~Jaffe, 
{\it Quantum physics. A functional integral point of view},
 Second edition. Springer-Verlag, New York, 1987. xxii+535 pp.



\bibitem{GOTW}
T.~Gunaratnam, T.~Oh, N.~Tzvetkov, H.~Weber, 
{\it Quasi-invariant Gaussian measures for the nonlinear wave equation in three dimensions,}
 Probab. Math. Phys. 3 (2022), no. 2, 343--379. 





\bibitem{GLM}
Z.~Guo, Y.~Lin, L.~Molinet, 
{\it Well-posedness in energy space for the periodic modified Banjamin-Ono equation,} 
J. Differential Equations 256 (2014), no. 8, 2778--2806.




\bibitem{GW}
Z.~Guo, B.~Wang,
 {\it Global well-posedness and limit behavior for the modified finite-depth-fluid  equation,}
arXiv:0809.2318 [math.AP].






\bibitem{Hajek}
J.~H\'ajek, %\textcolor{red}{(Gaek, Yaroslav)}
{\it On a property of normal distribution of any stochastic process},  
Czechoslovak Math. J. 8(83) (1958), 610--618. 









\bibitem{HW}
L.~Han, B.~Wang, 
 {\it Global wellposedness and limit behavior for the generalized finite-depth-fluid equation with small data in critical Besov spaces $\dot B^{s}_{2,1}$},
J. Differential Equations 245 (2008), no. 8, 2103--2144.



\bibitem{IS}
M.~Ifrim, J.-C.~Saut, 
{\it The lifespan of small data solutions for intermediate long wave equation (ILW)}, 
 Comm. Partial Differential Equations 50 (2025), no. 3, 258--300. 




\bibitem{Kakutani}
S.~Kakutani,
{\it On equivalence of infinite product measures},
Ann. of Math.  49 (1948), 214--224.






\bibitem{Khan86}
L.A.~Khan,
{\it Separability in function spaces,}
J. Math. Anal. Appl. 113 (1986), no. 1, 88--92.



\bibitem{KLV}
R.~Killip, T.~Laurens, M.~Vi\c{s}an, 
{\it Sharp well-posedness for the Benjamin--Ono equation}, 
Invent. Math. 236 (2024), no. 3, 999--1054.


\bibitem{KMV}
R.~Killip, J,~Murphy, M.~Vi\c{s}an, 
{\it Invariance of white noise for KdV on the line,} 
Invent. Math. 222 (2020), no. 1, 203--282. 




\bibitem{KS2022}
C.~Klein, J.-C.~Saut,
{\it  Nonlinear dispersive equations--inverse scattering and PDE methods},
Applied Mathematical Sciences, 209. Springer, Cham, [2021], \textcircled{c}2021. xx+580 pp.






\bibitem{KKD}
T.~Kubota, D.R.S.~Ko, L.D.~Dobbs, 
 {\it  Weakly-nonlinear, long internal gravity waves in stratified fluids of finite depth},
J.~Hydronautics 12 (1978), no. 4, 157--165.


\bibitem{KS}
S.~Kuksin, A.~Shirikyan, 
{\it Mathematics of two-dimensional turbulence,} 
Cambridge Tracts in Mathematics, 194. Cambridge University Press, Cambridge, 2012. xvi+320 pp. 



\bibitem{Kuo}
H-H.~Kuo,  
{\it  Introduction to stochastic integration},
 Universitext. Springer, New York, 2006. xiv+278 pp.
 























\bibitem{LRS} 
J.~Lebowitz, H.~Rose, E.~Speer, 
{\it Statistical mechanics of the nonlinear Schr\"odinger equation},
J. Statist. Phys. 50 (1988), no. 3-4, 657--687.






\bibitem{LNR}
M.~Lewin, P.T.~Nam, N.~Rougerie, 
{\it Classical field theory limit of many-body quantum Gibbs states in 2D and 3D},
 Invent. Math. 224 (2021), no. 2, 315--444.
 
 


\bibitem{Gli}
G.~Li, {\it Deep-water and shallow-water limits of the intermediate long wave equation}, 
 Nonlinearity 37 (2024), no.7, Paper No. 075001, 44 pp. 



\bibitem{Gli2}
G.~Li, 
{\it Dee-water and shallow-water limits of the intermediate long wave equation: from deterministic and statistical viewpoints}, 
Thesis (Ph.D.)--The University of Edinburgh. 2022.





\bibitem{MM00}
Y.~Martel, F.~Merle,
{\it A Liouville theorem for the critical generalized  Korteweg-de Vries equation},
J. Math. Pures Appl.  79 (2000), no. 4, 339--425.







\bibitem{McKean}
H.P.~McKean, 
{\it Statistical mechanics of nonlinear wave equations. IV. Cubic Schr\"odinger}, 
Comm. Math. Phys. 168 (1995), no. 3, 479--491.
{\it Erratum: Statistical mechanics of nonlinear wave equations. IV. Cubic Schr\"odinger}, Comm. Math. Phys. 173 (1995), no. 3, 675.





\bibitem{Me01}
F.~Merle, 
{\it Existence of blow-up solutions in the energy space for the critical generalized KdV equation,}
J. Amer. Math. Soc. 14 (2001), no. 3, 555--578.





\bibitem{Mi}
E.~Michael, 
{\it On a theorem of Rudin and Klee,}
Proc. Amer. Math. Soc. 12 (1961), 921.






\bibitem{MPS} 
C.~Mu\~noz, G.~Ponce, J.-C.~Saut,
 \textit{On the long time behavior of solutions to the intermediate long wave equation}, 
SIAM J. Math. Anal. 53 (2021), no. 1, 1029--1048.




\bibitem{Ne65} 
 E.~Nelson,
 {\it A quartic interaction in two dimensions}, 
1966 Mathematical Theory of Elementary Particles (Proc. Conf., Dedham, Mass., 1965) pp. 69--73 M.I.T. Press, Cambridge, Mass.





\bibitem{Nualart06}
D.~Nualart,
{\it The Malliavin calculus and related topics,}
Second edition. Probability and its Applications (New York). Springer-Verlag, Berlin, 2006. xiv+382 pp.





\bibitem{OH3} 
T.~Oh, 
{\it Invariant Gibbs measures and a.s. global well-posedness for coupled KdV systems,}
Differential Integral Equations 22 (2009), no. 7-8, 637--668.






\bibitem{OHRIMS}
T.~Oh, 
{\it White noise for KdV and mKdV on the circle}, 
Harmonic analysis and nonlinear partial differential equations, 99--124, RIMS K\^oky\^uroku Bessatsu, B18, Res. Inst. Math. Sci. (RIMS), Kyoto, 2010. 




\bibitem{OOT1}
T.~Oh, M.~Okamoto, L.~Tolomeo, 
{\it Focusing $\Phi^4_3$-model with a Hartree-type nonlinearity},
Mem. Amer. Math. Soc. 304 (2024), no. 1529. 





\bibitem{OOT2}
T.~Oh, M.~Okamoto, L.~Tolomeo, 
{\it Stochastic quantization of the $\Phi^3_3$-model},
 Mem. Eur. Math. Soc., 16. EMS Press, Berlin, 2025, viii+145 pp. 




\bibitem{OQS}
T.~Oh, J.~Quastel, P.~Sosoe,
{\it Invariant Gibbs measures for the defocusing nonlinear Schr\"odinger equations on the real line},
preprint.




\bibitem{ORT} 
T.~Oh, G.~Richards, L.~Thomann,
{\it On invariant Gibbs measures for the generalized KdV equations,}
Dyn. Partial Differ. Equ. 13 (2016), no. 2, 133--153.




\bibitem{OSRW2}
T.~Oh, T.~Robert, P.~Sosoe, Y.~Wang,
{\it  Invariant Gibbs dynamics for the dynamical sine-Gordon model}, 
Proc. Roy. Soc. Edinburgh Sect. A 151 (2021), no. 5, 1450--1466. 






\bibitem{OST} 
T.~Oh, K.~Seong, L.~Tolomeo,
{\it  A remark on Gibbs measures with log-correlated Gaussian fields},
Forum Math. Sigma. 12 (2024), e50, 40 pp.






\bibitem{OST22} 
T.~Oh, P.~Sosoe, L.~Tolomeo,
{\it Optimal integrability threshold for Gibbs measures associated with focusing NLS on the torus},
 Invent. Math. 227 (2022), no. 3, 1323--1429. 




\bibitem{OST18}
T.~Oh,  P.~Sosoe,  N.~Tzvetkov,
{\it  An optimal regularity result on the quasi-invariant Gaussian measures for the cubic fourth order nonlinear Schr\"odinger equation}, 
J. \'Ec. polytech. Math. 5 (2018), 793--841.





\bibitem{OTh1} 
T.~Oh, L.~Thomann,
{\it A pedestrian approach to the invariant Gibbs measure for the 2-d defocusing nonlinear Schr\"odinger equations},
Stoch. Partial Differ. Equ. Anal. Comput. 6 (2018), no. 3, 397--445.



%
%\bibitem{OTWZ}
%T.~Oh, L.~Tolomeo, Y.~Wang, G.~Zheng,
%{\it Hyperbolic $P(\Phi)_2$-model on the plane},
%arXiv:2211.03735 [math.AP].



\bibitem{OTz}
T.~Oh, N.~Tzvetkov, 
{\it Quasi-invariant Gaussian measures for the two-dimensional defocusing cubic nonlinear wave equation}, 
J. Eur. Math. Soc. (JEMS) 22 (2020), no. 6, 1785--1826. 





\bibitem{Pollard}
D.~Pollard, 
{\it A user's guide to measure theoretic probability},  
Cambridge Series in Statistical and Probabilistic Mathematics, 8. Cambridge University Press, Cambridge, 2002. xiv+351 pp. 






\bibitem{R16} 
 G.~Richards,
{\it Invariance of the Gibbs measure for the periodic quartic gKdV},
Ann. Inst. H. Poincar\'e C Anal. Non Lin\'eaire 33 (2016), no. 3, 699--766.




\bibitem{Rider}
B.~Rider, 
{\it On the $\infty$-volume limit of the focusing cubic Schr\"odinger equation},
Comm. Pure Appl. Math. 55 (2002), no. 10, 1231--1248. 




\bibitem{Saut2019}
J.-C. Saut,
{\it Benjamin-Ono and intermediate long wave equations: modeling, IST and PDE. Nonlinear dispersive partial differential equations and inverse scattering}, 
95--160, Fields Inst. Commun., 83, Springer, New York, [2019], \textcircled{c}2019.
 


\bibitem{Simon}
B.~Simon, 
{\it  The $P(\varphi)_2$ Euclidean (quantum) field theory}, 
Princeton Series in Physics. Princeton University Press, 
Princeton, N.J., 1974. xx+392 pp.




\bibitem{STz}
C.~Sun, N.~Tzvetkov, 
{\it Gibbs measure dynamics for the fractional NLS},
 SIAM J. Math. Anal. 52 (2020), no. 5, 4638--4704. 





\bibitem{Tao}
T.~Tao, 
{\it Global well-posedness of the Benjamin-Ono equation in $H^1(\mathbf{R})$}, 
J. Hyperbolic Differ. Equ. 1 (2004), no. 1, 27--49.



\bibitem{TTz}
L.~Thomann, N.~Tzvetkov, 
{\it Gibbs measure for the periodic derivative nonlinear Schr\"odinger equation},
Nonlinearity 23 (2010), no. 11, 2771--2791.




\bibitem{TW}
L.~Tolomeo, H.~Weber,
{\it 
Phase transition for invariant measures of the focusing Schr\"odinger equation}, 
arXiv:2306.07697 [math.AP].


\bibitem{Tsy}
A.B.~Tsybakov, 
{\it Introduction to nonparametric estimation,} 
Revised and extended from the 2004 French original. Translated by Vladimir Zaiats. Springer Series in Statistics. Springer, New York, 2009. xii+214 pp.







\bibitem{TZ2}
 N.~Tzvetkov, 
{\it Invariant measures for the defocusing Nonlinear Schr\"odinger equation},
 Ann. Inst. Fourier (Grenoble) 58 (2008), no. 7, 2543--2604.



  \bibitem{Tz10} 
N.~Tzvetkov,
{\it Construction of a Gibbs measure associated to the periodic Benjamin-Ono equation},
Probab. Theory Related Fields 146 (2010), no. 3-4, 481--514.




\bibitem{TV1}
N.~Tzvetkov, N.~Visciglia, 
{\it Gaussian measures associated to the higher order conservation laws of the Benjamin-Ono equation}, 
Ann. Sci. \'Ec. Norm. Sup\'er.  46 (2013), no. 2, 249--299.





\bibitem{TV2}
N.~Tzvetkov, N.~Visciglia, 
{\it Invariant measures and long-time behavior for the Benjamin-Ono equation},
 Int. Math. Res. Not. IMRN 2014, no. 17, 4679--4714.
 
 
 
 

\bibitem{TV3}
N.~Tzvetkov, N.~Visciglia, 
{\it Invariant measures and long time behaviour for the Benjamin-Ono equation II},
J. Math. Pures Appl. 103 (2015), no. 1, 102--141. 




\bibitem{Ust}
A.~ \"Ust\"unel,
{\it Variational calculation of Laplace transforms via entropy on Wiener space and applications},
J. Funct. Anal. 267 (2014), no. 8, 3058--3083.




 \bibitem{Zhid}
P.E.~Zhidkov, 
{\it An invariant measure for the nonlinear Schr\"odinger equation,}
Dokl. Akad. Nauk SSSR 317 (1991), no. 3, 543--546; 
 translation in Soviet Math. Dokl. 43 (1991), no. 2, 431--434.
 



\bibitem{Zhi01} 
 P.E.~Zhidkov, 
{\it  Korteweg-de Vries and nonlinear Schr\"odinger equations: qualitative theory,}
Lecture Notes in Mathematics, 1756. Springer-Verlag, Berlin, 2001. vi+147 pp.



\bibitem{Zine1}
Y.~Zine, 
{\it Smoluchowski-Kramers approximation for the singular stochastic wave equations in two dimensions}, 
arXiv:2206.08717 [math.AP].



\bibitem{Zine2}
Y.~Zine, 
{\it  On the inviscid limits of the singular stochastic complex Ginzburg-Landau equation at statistical equilibrium}, 
arXiv:2212.00604 [math.AP].


 
\end{thebibliography}
\end{document}